%% file: qkpos.tex
\def\qkpossz{10pt}
\IfFileExists{qkpos.conf}{\input{qkpos.conf}}{}
\documentclass[\qkpossz]{amsart}
\usepackage{amsmath, amsfonts, amssymb}
\usepackage{stmaryrd}
\usepackage{accents}
\usepackage[all]{xy}
\usepackage{graphicx}
\usepackage{refcount}
\usepackage{hhline}
\usepackage{tableau}
\usepackage{tikz}
\usetikzlibrary{patterns}
\usepackage[colorlinks=true, linkcolor=blue,
  citecolor=blue, urlcolor=blue]{hyperref}
\ifcsname mypreamble\endcsname\mypreamble\fi

\input{preamble}

\newcommand{\semi}{contact}
\newcommand{\Semi}{Contact}

\begin{document}

\title{Positivity of minuscule quantum $K$-theory}

\ifdefined\gitdate
\date{\gitdate\ revision {\tt \gittag}}
\else
\date{March 13, 2026}
\fi

\author{Anders~S.~Buch}
\address{Department of Mathematics, Rutgers University, 110
  Frelinghuysen Road, Piscataway, NJ 08854, USA}
\email{asbuch@math.rutgers.edu}

\author{Pierre--Emmanuel Chaput}
\address{Domaine Scientifique Victor Grignard, 239, Boulevard des
  Aiguillettes, Universit{\'e} de Lorraine, B.P.  70239,
  F-54506 Vandoeuvre-l{\`e}s-Nancy Cedex, France}
\email{pierre-emmanuel.chaput@univ-lorraine.fr}

\author{Leonardo~C.~Mihalcea}
\address{Department of Mathematics, Virginia Tech University, 460
  McBryde, Blacksburg VA 24060, USA}
\email{lmihalce@math.vt.edu}

\author{Nicolas Perrin}
\address{Laboratoire de Math\'ematiques de Versailles, UVSQ, CNRS,
  Universit\'e Paris-Saclay, 78035 Versailles, France}
\email{nicolas.perrin@uvsq.fr}

\subjclass[2020]{Primary 14N35; Secondary 19E08, 14N15, 14M15}

\keywords{Quantum $K$-theory, Gromov-Witten invariants, cominuscule flag
varieties, Schubert varieties, Richardson varieties, Schubert structure
constants, positivity}

\thanks{The authors were partially supported by NSF grants DMS-1205351 and
DMS-1503662 and a Visiting Professorship at Universit{\'e} de Lorraine (Buch),
NSA grants H98230-13-1-0208 and H98320-16-1-0013 and a Simons Collaboration
Grant (Mihalcea), NSF grant DMS-1929284 while in residence at the Institute for
Computational and Experimental Research in Mathematics in Providence, RI, during
the Spring of 2021 (Buch and Mihalcea), and a public grant as part of the
Investissement d'avenir project, reference ANR-11-LABX-0056-LMH, LabEx LMH
(Perrin)}

\dedicatory{\mbox{}\vspace{2mm}\mbox{}}

\definecolor{uayellow}{rgb}{.7,.7,.0}

{\color{blue}
\begin{abstract}
  {\color{blue}%
  We prove that the Schubert structure constants of the quantum $K$-theory ring
  of any minuscule flag variety or quadric hypersurface have signs that
  alternate with codimension. We also prove that the powers of the deformation
  parameter $q$ that occur in the product of two Schubert classes in the quantum
  cohomology or quantum $K$-theory ring of a cominuscule flag variety form an
  integer interval. Our proofs are based on several new results, including an
  explicit description of the most general non-empty intersection}
  {\color{uayellow}%
  of two Schubert varieties in an arbitrary flag manifold, and a computation of
  the cohomology groups of any negative line bundle restricted to a Richardson
  variety in a cominuscule flag variety. We also give a type-uniform proof of
  the quantum-to-classical theorem, which asserts that the (3-point, genus 0)
  Gromov-Witten invariants of any cominuscule flag variety are classical
  triple-intersection numbers on an associated flag variety. Finally, we prove
  several new results about the geometry and combinatorics related to this
  theorem.\parfillskip=0pt\par}
\end{abstract}
}

\maketitle

\markboth{A.~BUCH, P.--E.~CHAPUT, L.~MIHALCEA, AND N.~PERRIN}
{POSITIVITY OF MINUSCULE QUANTUM $K$-THEORY}

\setcounter{tocdepth}{1}
{\hypersetup{hidelinks}
\tableofcontents
}

\input{intro}
\input{schubint}
\input{projrich}
\input{comin}
\input{qclassical}
\input{qcfibers}
\input{qhinterval}
\input{quantumk}
\input{divisor}
\input{claims}
\input{index}

\ifdefined\qkposbib
\bibliography{qkpos,\qkposbib}
\else

\input{bibliography}
\fi
\bibliographystyle{halpha}

\end{document}

%% file: preamble.tex
\def\newthm#1#2{\newtheorem{#1}[dummy]{#2}%
  \expandafter\def\csname#2\endcsname##1{\hyperref[#1:##1]{#2~\ref*{#1:##1}}}}
\newthm{thm}{Theorem}
\newthm{lemma}{Lemma}
\newthm{prop}{Proposition}
\newthm{cor}{Corollary}
\newthm{conj}{Conjecture}
\newthm{cons}{Consequence}
\newthm{claim}{Claim}
\newthm{fact}{Fact}
\newthm{obs}{Observation}
\newthm{guess}{Guess}
\theoremstyle{definition}
\newthm{defn}{Definition}
\newthm{remark}{Remark}
\newthm{example}{Example}
\newthm{question}{Question}
\newthm{notation}{Notation}

\newcommand{\Section}[1]{\hyperref[sec:#1]{Section~\ref*{sec:#1}}}
\newcommand{\Sec}[1]{\hyperref[sec:#1]{\S\ref*{sec:#1}}}
\newcommand{\Table}[1]{\hyperref[tab:#1]{Table~\ref*{tab:#1}}}
\newcommand{\Figure}[1]{\hyperref[fig:#1]{Figure~\ref*{fig:#1}}}
\newcommand{\eqn}[1]{\hyperref[eqn:#1]{(\ref*{eqn:#1})}}

\newcounter{symbolpage}
\def\spage#1{\ifnum\value{symbolpage}=\getpagerefnumber{#1}\else\hyperref[#1]{\textcolor{blue}{[\pageref*{#1}]}}\, \fi\setcounter{symbolpage}{\getpagerefnumber{#1}}}
\newif\ifsnext
\def\scomma{\ifsnext,\, \fi\snexttrue}
\def\sbreak{,\\\snextfalse}
\newcommand{\targetsec}[2]{\hypertarget{link:#1}{\label{page:#1}#2}}
\newcommand{\indexsec}[2]{\snextfalse\noin\Section{#1}: {\hypersetup{hidelinks}#2}.}
\newcommand{\slink}[2]{\scomma\spage{page:#1}\hyperlink{link:#1}{#2}}
\newcommand{\sref}[2]{\scomma\spage{#1}\hyperref[#1]{#2}}

\DeclareMathOperator{\GL}{GL}
\DeclareMathOperator{\SL}{SL}
\DeclareMathOperator{\PGL}{PGL}
\DeclareMathOperator{\Sp}{Sp}
\DeclareMathOperator{\SO}{SO}
\DeclareMathOperator{\Gr}{Gr}
\DeclareMathOperator{\Fl}{Fl}
\DeclareMathOperator{\LG}{LG}
\DeclareMathOperator{\IG}{IG}
\DeclareMathOperator{\SG}{SG}

\DeclareMathOperator{\OG}{OG}

\DeclareMathOperator{\Aut}{Aut}
\DeclareMathOperator{\Ker}{Ker}
\DeclareMathOperator{\Span}{Span}
\DeclareMathOperator{\Spec}{Spec}
\DeclareMathOperator{\QH}{QH}
\DeclareMathOperator{\QK}{QK}

\DeclareMathOperator{\codim}{codim}

\DeclareMathOperator{\Pic}{Pic}

\DeclareMathOperator{\dist}{dist}

\DeclareMathOperator{\lead}{lead}

\newcommand{\comin}{\mathrm{comin}}

\newcommand{\ssm}{\smallsetminus}
\newcommand{\bP}{{\mathbb P}}
\renewcommand{\P}{{\mathbb P}}
\newcommand{\C}{{\mathbb C}}

\newcommand{\Q}{{\mathbb Q}}
\newcommand{\Z}{{\mathbb Z}}
\newcommand{\N}{{\mathbb N}}
\newcommand{\cB}{{\mathcal B}}

\newcommand{\cD}{{\mathcal D}}

\newcommand{\cF}{{\mathcal F}}
\newcommand{\cI}{{\mathcal I}}
\newcommand{\cJ}{{\mathcal J}}
\newcommand{\cK}{{\mathcal K}}

\newcommand{\cO}{{\mathcal O}}
\newcommand{\cP}{{\mathcal P}}
\newcommand{\cR}{{\mathcal R}}
\newcommand{\cS}{{\mathcal S}}
\newcommand{\cT}{{\mathcal T}}
\newcommand{\hT}{\widehat{\mathcal T}}

\newcommand{\cZ}{{\mathcal Z}}
\newcommand\oE{\accentset{\circ}{E}}
\newcommand\oX{\accentset{\circ}{X}}
\newcommand\oY{\accentset{\circ}{Y}}
\newcommand\oZ{\accentset{\circ}{Z}}
\newcommand\oGamma{\accentset{\circ}{\Gamma}}
\newcommand\oPi{\accentset{\circ}{\Pi}}
\newcommand\ocZ{\accentset{\circ}{\mathcal Z}}
\newcommand\Zd[1]{Z_d^{(#1)}}
\newcommand\oZd[1]{\oZ_d^{(#1)}}
\newcommand\Zdd[1]{Z_{d-1,1}^{(#1)}}
\newcommand\hZdd[1]{\wh Z_{d-1,1}^{(#1)}}
\newcommand\Zhdd{{\wh Z_{d-1,1}}}

\newcommand{\gw}[2]{\langle #1 \rangle^{\mbox{}}_{#2}}

\newcommand{\qpair}[1]{(\!(#1)\!)}
\newcommand{\euler}[1]{\chi_{_{#1}}}
\newcommand{\pt}{\mathrm{point}}

\newcommand{\al}{{\alpha}}
\newcommand{\be}{{\beta}}
\newcommand{\ga}{{\gamma}}

\newcommand{\ka}{{\kappa}}
\newcommand{\la}{{\lambda}}

\newcommand{\om}{{\omega}}
\newcommand{\tal}{{\widetilde\alpha}}
\newcommand{\dmin}{d_{\min}}
\newcommand{\dmax}{d_{\max}}

\newcommand{\Bl}{{\mathrm{B}\ell}}

\newcommand{\ev}{\operatorname{ev}}
\newcommand{\wt}{\widetilde}
\newcommand{\wh}{\widehat}
\newcommand{\wb}{\overline}
\newcommand{\ov}{\overline}
\newcommand{\pic}[2]{\includegraphics[scale=#1]{Figures/#2}}
\newcommand{\ignore}[1]{}

\newcommand{\Mb}{\wb{\mathcal M}}
\newcommand{\ch}{\operatorname{ch}}
\newcommand{\noin}{\noindent}

\newcommand{\hmm}[1]{\mbox{}\hspace{#1mm}\mbox{}}
\newcommand{\YdoGam}[1]{Y_d(#1,\oGamma_d(#1))}

\newenvironment{abcenum}{\begin{enumerate}}{\end{enumerate}}

%% file: intro.tex

\section{Introduction}

\subsection{Positivity in quantum $K$-theory}

Let $X = G/P_X$ be a flag variety defined by a semi-simple complex linear
algebraic group $G$ and a parabolic subgroup $P_X$. The (small) quantum
$K$-theory ring $\QK(X)$ of Givental and Lee \cite{givental:wdvv, lee:quantum}
is a deformation of the $K$-theory ring $K(X)$ of algebraic vector bundles on
$X$, whose structure constants $N^{w,d}_{u,v}$ encode the arithmetic genera of
the (3 pointed, genus zero) Gromov-Witten varieties of $X$. Here $N^{w,d}_{u,v}$
denotes the coefficient of $q^d \cO^w$ in the product $\cO^u \star \cO^v$ in
$\QK(X)$, where $\cO^w = [\cO_{X^w}]$ denotes a $K$-theoretic Schubert class of
$X$, and $q^d$ is a monomial in the deformation parameters of $\QK(X)$, encoding
a degree $d \in H_2(X,\Z)$. The Schubert variety $X^w$ corresponds to a Weyl
group element $w$ such that $\codim(X^w,X)$ is equal to the length $\ell(w)$.

The constant $N^{w,d}_{u,v}$ is non-zero only if
\[
  \ell(w) + \int_d c_1(T_X) \geq \ell(u) + \ell(v) \,,
\]
and when this inequality is satisfied with equality, $N^{w,d}_{u,v}$ is the
(cohomological) Gromov-Witten invariant $\gw{[X^u], [X^v], [X_w]}{d}$, equal to
the number of parameterized curves $f : \bP^1 \to X$ of degree $d$ that map the
points $0, 1, \infty$ to general translates of the Schubert varieties $X^u$,
$X^v$, and $X_w$; here $X_w$ is the Schubert variety of dimension $\ell(w)$
opposite to $X^w$. In particular, $N^{w,d}_{u,v}$ is non-negative in this case.
More generally, it is conjectured \cite{lenart.maeno:quantum,
buch.mihalcea:quantum, buch.chaput.ea:chevalley} that the constants
$N^{w,d}_{u,v}$ have signs that alternate with codimension, in the sense that
\begin{equation}\label{eqn:qksigns}
  (-1)^{\ell(uvw) + \int_d c_1(T_X)} N^{w,d}_{u,v} \geq 0 \,.
\end{equation}
This generalizes the fact that the Schubert structure constants of the
$K$-theory ring $K(X)$ have alternating signs
\cite{buch:littlewood-richardson*1, brion:positivity}. Our main result is a
proof of the alternating signs conjecture for $\QK(X)$ when $X$ is a minuscule
flag variety or a quadric hypersurface.

\begin{thm}\label{thm:qkpos-minuscule}%
  Assume that $X$ is a minuscule flag variety or a quadric hypersurface. Then,
  $(-1)^{\ell(uvw) + \int_d c_1(T_X)} N^{w,d}_{u,v} \geq 0$.
\end{thm}

The family of \emph{minuscule} flag varieties includes Grassmannians of Lie type
A, maximal orthogonal Grassmannians, quadric hypersurfaces of even dimension,
and two exceptional spaces of type E called the Cayley plane and the Freudenthal
variety. The larger family of \emph{cominuscule} flag varieties also includes
Lagrangian Grassmannians and odd dimensional quadrics (see \Section{comin}). We
will prove more generally that the constant $N^{w,d}_{u,v}$ has the expected
sign \eqn{qksigns} whenever $X$ is cominuscule and $q^d$ occurs in the product
$[X^u] \star [X^v]$ in the quantum cohomology ring $\QH(X)$.

Earlier examples where alternating signs of the structure constants in quantum
$K$-theory have been proved include the Pieri formula for products with special
Schubert classes on Grassmannians of type A \cite{buch.mihalcea:quantum},
structure constants associated to `line' degrees corresponding to certain
fundamental weights on any $G/P$ \cite{li.mihalcea:k-theoretic}, Chevalley
formulas for products with Schubert divisors on some families of flag varieties
\cite{buch.chaput.ea:chevalley, lenart.naito.ea:combinatorial,
kouno.lenart.ea:quantum}, and all the structure constants of the quantum
$K$-theory of incidence varieties $\Fl(1,n-1;n)$ of type A \cite{rosset:quantum,
xu:quantum*1}.

The ring $\QK(X)$ has an alternative basis $\{\cI_q^w\}$ that is dual to the
structure sheaf basis $\{\cO_w\}$ under Givental's quantum $K$-metric. A formula
for the dual Schubert classes $\cI_q^w$ in $\QK(X)$ has recently been proved in
\cite{summers:dual} when $X$ is cominuscule. This formula together with
\Theorem{qkpos-minuscule} implies that the structure constants of $\QK(X)$
relative to the dual basis also have alternating signs when $X$ is minuscule or
a quadric. More generally, we conjecture that the dual structure constants of
the equivariant quantum $K$-theory ring $\QK_T(X)$ are positive in the sense of
\cite{griffeth.ram:affine,anderson.griffeth.ea:positivity}, see
\Section{qkdual}.

\subsection{Powers of $q$ in quantum products}

We also address the problem of finding the powers $q^d$ that occur in the
quantum product of two Schubert classes, either in the quantum cohomology ring
$\QH(X)$ or the quantum $K$-theory ring $\QK(X)$. When $X$ is a Grassmannian of
type A, the smallest power $\dmin(u,v)$ of $q$ in the product $[X^u] \star [X^v]
\in \QH(X)$ was determined by Fulton and Woodward
\cite{fulton.woodward:quantum}, as the number of diagonal units the (dual) Young
diagram of $X^u$ must be translated in order to contain the Young diagram of
$X^v$. Postnikov \cite{postnikov:affine} gave a similar rule for the largest
degree $\dmax(u,v)$, and also proved that the powers of $q$ that occur form an
integer interval. This answers the question for the quantum cohomology of
Grassmannians of type A.

For an arbitrary flag variety $X$, it is not clear if a quantum product $[X^u]
\star [X^v] \in \QH(X)$ contains a minimal or maximal power of $q$, since the
group $H_2(X,\Z)$ is linearly ordered only when it has rank 1. It turns out that
$[X^u] \star [X^v]$ always contains a minimal power $q^d$
\cite{postnikov:affine, buch.chung.ea:euler}, where $d = \dmin(u,v)$ is the
(unique) minimal degree of a rational curve connecting two general translates of
the Schubert varieties $X^u$ and $X^v$. The corresponding product $\cO^u \star
\cO^v$ in $\QK(X)$ contains the same minimal power $q^d$. However, a quantum
cohomology product $[X^u] \star [X^v]$ may not contain a unique maximal power of
$q$, and the powers that occur may not form a convex subset in the natural
partial order of $H_2(X,\Z)$. For example, the $q$-degrees in the square of
$[X^{164532}]$ in $\QH(\Fl(\C^6))$ do not form a convex subset, and no unique
maximal degree exists.

Any cominuscule flag variety $X$ has Picard rank one, so all quantum products
automatically contain a maximal power of $q$. We will show that the interval
property also holds when $X$ is cominuscule. More precisely, let $\cB = \{ q^d
[X^u] \}$ denote the natural $\Z$-basis of $\QH(X)$. It was proved in
\cite{belkale:transformation, chaput.manivel.ea:affine} that any product $[\pt]
\star [X^v]$ belongs to $\cB$ when $X$ is cominuscule. Define a partial order on
$\cB$ by $q^e [X^v] \leq q^d [X^u]$ if and only if the Schubert varieties $X_u$
and $X^v$ are connected by a rational curve of degree $d-e$.

\begin{thm}\label{thm:qh-powers-intro}%
  Assume that $X$ is cominuscule. Then the powers $q^d$ that occur in $[X^u]
  \star [X^v] \in \QH(X)$ form an integer interval. More precisely, $q^d$ occurs
  in $[X^u] \star [X^v]$ if and only if $[X^v] \leq q^d [X_u] \leq [\pt] \star
  [X^v]$.
\end{thm}

Postnikov's description of the extreme powers of $q$ in a quantum product on the
Grassmannian $\Gr(m,n)$ involves order ideals in the \emph{cylinder}
$\Z^2/(-m,n-m)\Z$ that extends the usual $m \times (n-m)$-rectangle of boxes
associated with the Grassmannian. \Theorem{qh-powers-intro} can be interpreted
as a type-uniform generalization of this construction. In fact, $\cB$ turns out
to be a distributive lattice, and the join-irreducible elements in $\cB$ can be
identified with a set of boxes in the plane that specializes to Postnikov's
cylinder in type A. Examples are provided in \Section{qshapes}. Isomorphic
partially ordered sets have been constructed in \cite{hagiwara:minuscule*2,
green:combinatorics*1}, where they are used to study to minuscule
representations.

In quantum $K$-theory it is known that the powers of $q$ in any product $\cO^u
\star \cO^v$ are bounded above \cite{buch.chaput.ea:finiteness,
buch.chaput.ea:rational, kato:quantum, anderson.chen.ea:finiteness}; this
is not apparent from Givental's definition of the product in $\QK(X)$.
\Theorem{qh-powers-intro} has the following generalization.

\begin{thm}\label{thm:qk-powers-intro}%
  \noin{\rm(a)} Assume that $X$ is minuscule. Then $q^d$ occurs in $\cO^u \star
  \cO^v$ if and only if $q^d$ occurs in $[X^u] \star [X^v]$.\smallskip

  \noin{\rm(b)} Assume that $X$ is cominuscule. Then the powers of $q$ that
  occur in $\cO^u \star \cO^v$ form an integer interval. The smallest power
  matches the smallest power in $[X^u] \star [X^v]$, and the largest power is at
  most one larger than the largest power in $[X^u] \star [X^v]$.
\end{thm}

In \Definition{exceptional} we give a combinatorial definition of an
\emph{exceptional degree} of a product $\cO^u \star \cO^v$ in the quantum
$K$-theory ring of any cominuscule flag variety $X$. Exceptional degrees occur
only when $X$ is not minuscule, and even in this case, most products have no
exceptional degrees (see \Table{exceptional-counts}). If $\cO^u \star \cO^v$ has
an exceptional degree, then this degree is $\dmax(u,v)+1$. We prove that if
$q^d$ occurs in $\cO^u \star \cO^v$, then either $\dmin(u,v) \leq d \leq
\dmax(u,v)$, or $d = \dmax(u,v)+1$ is an exceptional degree. In particular, this
means that $N^{w,d}_{u,v}$ has the expected sign whenever $d$ is not an
exceptional degree. We conjecture that $q^d$ occurs in $\cO^u \star \cO^v$
whenever $d$ is an exceptional degree. Equivalently, $d$ is an exceptional
degree if and only if $q^d$ occurs in the quantum $K$-theory product $\cO^u
\star \cO^v$ but not in the quantum cohomology product $[X^u] \star [X^v]$.

\subsection{Strategy of proof}

\targetsec{Md11}{}%
To illustrate the main strategy in our proofs, fix a cominuscule flag variety
$X$ and let $M_d = \Mb_{0,3}(X,d)$ denote the Kontsevich moduli space of
3-pointed stable maps to $X$ of degree $d$ and genus zero
\cite{fulton.pandharipande:notes}. Let $M_d(X_u,X^v) = \ev_1^{-1}(X_u) \cap
\ev_2^{-1}(X^v) \subset M_d$ denote the \emph{Gromov-Witten variety} of stable
maps that send the first two marked points to the Schubert varieties $X_u$ and
$X^v$, respectively. The image $\Gamma_d(X_u,X^v) = \ev_3(M_d(X_u,X^v)) \subset
X$ is a two-pointed \emph{curve neighborhood}, equal to the closure of the union
of all rational curves of degree $d$ in $X$ that meet $X_u$ and $X^v$. We also
let $M_{d-1,1} \subset M_d$ denote the divisor of stable maps $f : C \to X$ for
which the domain has (at least) two components, $C = C_1 \cup C_2$, such that
$C_1$ contains the first two marked points, $C_2$ contains the third marked
point, and the restrictions of $f$ to $C_1$ and $C_2$ have degrees $d-1$ and
$1$, respectively. Set $M_{d-1,1}(X_u,X^v) = M_{d-1,1} \cap M_d(X_u,X^v)$ and
$\Gamma_{d-1,1}(X_u,X^v) = \ev_3(M_{d-1,1}(X_u,X^v))$. In other words,
$\Gamma_{d-1,1}(X_u,X^v)$ is the closure of the set of points in $X$ that are
connected by a line to a rational curve of degree $d-1$ from $X_u$ to $X^v$.

Let $\cO_u = [\cO_{X_u}]$ and $\cO^v = [\cO_{X^v}]$ be two opposite Schubert
classes. It follows from \cite[Prop.~3.2]{buch.chaput.ea:chevalley} that the
product $\cO_u \star \cO^v \in \QK(X)$ is given by\footnote{The dual Weyl group
element $u^\vee$ satisfies $\cO^{u^\vee} = \cO_u$.}
\[
  \cO_u \star \cO^v =
  \sum_{w, d \geq 0} N^{w,d}_{u^\vee,v}\, q^d\, \cO^w =
  \sum_{d \geq 0} (\cO_u \star \cO^v)_d\, q^d \,,
\]
where the classes $(\cO_u \star \cO^v)_d \in K(X)$ are determined by
\begin{equation}\label{eqn:qkprod-d}%
  (\cO_u \star \cO^v)_d =
  (\ev_3)_* [\cO_{M_d(X_u,X^v)}] - (\ev_3)_* [\cO_{M_{d-1,1}(X_u,X^v)}] \,.
\end{equation}
It was proved in \cite[Thm.~4.1]{buch.chaput.ea:projected} that
$\Gamma_d(X_u,X^v)$ is a projected Richardson variety in $X$, and the restricted
map $\ev_3 : M_d(X_u,X^v) \to \Gamma_d(X_u,X^v)$ is cohomologically trivial,
that is, $(\ev_3)_* \cO_{M_d(X_u,X^v)} = \cO_{\Gamma_d(X_u,X^v)}$ and $R^j
(\ev_3)_* \cO_{M_d(X_u,X^v)} = 0$ for $j > 0$. In particular, we have $(\ev_3)_*
[\cO_{M_d(X_u,X^v)}] = [\cO_{\Gamma_d(X_u,X^v)}]$ in $K(X)$. Since projected
Richardson varieties have rational singularities
\cite{billey.coskun:singularities, knutson.lam.ea:projections}, it follows from
a theorem of Brion \cite{brion:positivity} that the expansion of
$[\cO_{\Gamma_d(X_u,X^v)}]$ in the Schubert basis of $K(X)$ has alternating
signs in the sense of \Theorem{qkpos-minuscule}.

The further restriction $\ev_3 : M_{d-1,1}(X_u,X^v) \to \Gamma_{d-1,1}(X_u,X^v)$
is not as well understood; for example, we do not know if
$\Gamma_{d-1,1}(X_u,X^v)$ has rational singularities. Our strategy is to
establish the following two properties.

\begin{itemize}
  \item[(i)] The general fibers of the map $\ev_3 : M_{d-1,1}(X_u,X^v) \to
  \Gamma_{d-1,1}(X_u,X^v)$ are cohomologically trivial.

  \item[(ii)] The variety $\Gamma_{d-1,1}(X_u,X^v)$ is either equal to
  $\Gamma_d(X_u,X^v)$ or a divisor in $\Gamma_d(X_u,X^v)$.
\end{itemize}

The first property (i) implies, by using a result of Koll\'ar
\cite{kollar:higher} (see \Corollary{push}), that $(\ev_3)_*
[\cO_{M_{d-1,1}(X_u,X^v)}]$ is equal to the class of a resolution of
singularities of $\Gamma_{d-1,1}(X_u,X^v)$. In particular, if
$\Gamma_{d-1,1}(X_u,X^v) = \Gamma_d(X_u,X^v)$, then $(\cO_u \star \cO^v)_d = 0$.
Otherwise, property (ii) predicts that $\Gamma_{d-1,1}(X_u,X^v)$ is a divisor in
$\Gamma_d(X_u,X^v)$. In this case Brion's theorem \cite{brion:positivity}
implies that the expansion of $(\ev_3)_* [\cO_{M_{d-1,1}(X_u,X^v)}]$ has signs
that are opposite to the signs of $[\cO_{\Gamma_d(X_u,X^v)}]$, so that all signs
are compatible in the difference \eqn{qkprod-d}. \Theorem{qkpos-minuscule} is
therefore a consequence of properties (i) and (ii). In addition, $q^d$ occurs in
the product $\cO_u \star \cO^v$ if and only if $\Gamma_{d-1,1}(X_u,X^v)
\subsetneq \Gamma_d(X_u,X^v)$.

We show that property (ii) is always true. More precisely,
$\Gamma_{d-1,1}(X_u,X^v)$ is empty for $d \leq \dmin(u^\vee,v)$, is a
(non-empty) divisor in $\Gamma_d(X_u,X^v)$ for $\dmin(u^\vee,v) < d \leq
\dmax(u^\vee,v)$, and is equal to $\Gamma_d(X_u,X^v)$ for $d > \dmax(u^\vee,v)$.
We also show that (i) is true, \emph{except} when $d$ is an exceptional degree
of $\cO_u \star \cO^v$. In this case the general fibers of the map $\ev_3 :
M_{d-1,1}(X_u,X^v) \to \Gamma_{d-1,1}(X_u,X^v)$ have arithmetic genus one (see
\Theorem{exceptional} and \Corollary{introclaims_birat}), which explains the
exceptional behavior of exceptional degrees.

\subsection{The quantum-to-classical construction}

Our proofs rely on the geometric construction underlying the \emph{quantum
equals classical} theorem for cominuscule flag varieties \cite{buch:quantum,
buch.kresch.ea:gromov-witten, chaput.manivel.ea:quantum*1,
buch.mihalcea:quantum, chaput.perrin:rationality}, which we proceed to discuss
when $X = \Gr(m,n)$ is the Grassmannian of $m$-planes in $\C^n$. Given any
stable map $f : C \to X$, let $\Ker(f) \subset \C^n$ be the intersection of the
$m$-planes contained in the image $f(C)$, and let $\Span(f) \subset \C^n$ be the
linear span of these $m$-planes. For $f$ in a dense open subset of $M_d =
\Mb_{0,3}(X,d)$, with $d \leq \min(m,n-m)$, we have $\dim \Ker(f) = m-d$ and
$\dim \Span(f) = m+d$, that is, $(\Ker(f),\Span(f))$ is a point in the two-step
flag variety $Y_d = \Fl(m-d,m+d;n)$. Define the three-step flag variety $Z_d =
\Fl(m-d,m,m+d;n)$, and let $p_d : Z_d \to X$ and $q_d : Z_d \to Y_d$ be the
projections. The quantum equals classical theorem states that any (3 point,
genus zero) Gromov-Witten invariant of $X$ is given by
\[
  \gw{\Omega_1, \Omega_2, \Omega_3}{d} =
  \int_{Y_d} {q_d}_* p_d^*(\Omega_1) \cdot {q_d}_* p_d^*(\Omega_2) \cdot
  {q_d}_* p_d^*(\Omega_1) \,.
\]
Define a rational map $\varphi : M_d \dashrightarrow Z_d$ by $\varphi(f) =
(\Ker(f), \ev_3(f), \Span(f))$, and define subvarieties of $Z_d$ by
\[
  Z_d(X_u,X^v) = \ov{\varphi(M_d(X_u,X^v))}
  \text{ \ and \ }
  Z_{d-1,1}(X_u,X^v) = \ov{\varphi(M_{d-1,1}(X_u,X^v))} \,.
\]
We will show that (completions of) the general fibers of the restricted maps
$\varphi : M_d(X_u,X^v) \dashrightarrow Z_d(X_u,X^v)$ and $\varphi :
M_{d-1,1}(X_u,X^v) \dashrightarrow Z_{d-1,1}(X_u,X^v)$ are cohomologically
trivial. As a consequence, we can replace $M_{d-1,1}(X_u,X^v)$ with
$Z_{d-1,1}(X_u,X^v)$ in property (i). Furthermore, $Z_d(X_u,X^v)$ is a
Richardson variety in $Z_d$, and $Z_{d-1,1}(X_u,X^v)$ is the inverse image of a
projected Richardson variety in $Y_d$. Geometric results about Schubert
varieties can therefore be utilized for studying the fibers of the map $p_d :
Z_{d-1,1}(X_u,X^v) \to \Gamma_{d-1,1}(X_u,X^v)$.

The quantum-to-classical construction can be generalized to any cominuscule flag
variety $X$ by replacing kernel-span pairs $\om = (K,S)$ with certain
well-behaved subvarieties $\Gamma_\om \subset X$, which in type A are the
sub-Grassmannians $\Gamma_\om = \Gr(d,S/K) \subset \Gr(m,n)$. These subvarieties
correspond to points in a related flag variety $Y_d = G/P_{Y_d}$, which in turn
defines the incidence variety $Z_d = G/(P_X \cap P_{Y_d}) = \{ (\om,x) \in Y_d
\times X \mid x \in \Gamma_\om \}$. This provides a type-independent framework
for studying properties (i) and (ii).

\subsection{\Semi-transversal intersections}

In order to establish the required geometric properties of (especially) fibers
of maps related to the quantum-to-classical construction, we prove a number of
new results about intersections of Schubert varieties in arbitrary flag
varieties that are not in general position. In particular, given two opposite
Schubert varieties with empty intersection, we define and study a
\emph{\semi-transversal intersection} obtained when these varieties are moved
towards each other until they just meet, using the group action. In fact,
\semi-transversal intersections can be defined for subvarieties of any variety
with a group action, but in this generality it is not guaranteed that a
\semi-transversal intersection exists. We show that the \semi-transversal
intersection of two Schubert varieties always exists, is a Richardson variety,
and we give explicit descriptions of the defining Weyl group elements. For
example, if $X = G/P_X$ is a cominuscule flag variety, then the \semi-transversal
intersection of $X_u$ with $X^v$ is the Richardson variety $X_u \cap X^{u \cap
v}$, where $u \cap v$ denotes the join operation on the set of minimal length
Weyl group elements, corresponding to the intersection of Young diagrams in type
A. We also prove the following result about the fibers of a projection of a
Schubert or Richardson variety to a smaller flag variety.

\begin{thm}
  Let $\pi : Z \to X$ be a projection of flag varieties. Each fiber $\pi^{-1}(x)
  \cong P_X/P_Z$ is again a flag variety.\smallskip

  \noin{\rm(a)} Let $Z_u \subset Z$ be a Schubert variety. Then $Z_u \cap
  \pi^{-1}(x)$ a (reduced) Schubert variety in $\pi^{-1}(x)$ for all $x \in
  \pi(Z_u)$.\smallskip

  \noin{\rm(b)} Let $Z_u^v = Z_u \cap Z^v \subset Z$ be a Richardson variety.
  Then $Z_u^v \cap \pi^{-1}(x)$ is a Richardson variety in $\pi^{-1}(x)$ for all
  $x$ in a dense open subset of $\pi(Z_u^v)$.
\end{thm}

The general fibers in both (a) and (b) are given by explicitly determined Weyl
group elements (see \Theorem{schubfib} and \Theorem{richfib}). The explicit
description of the fibers of the projection of a Richardson variety is crucial
for our proofs of \Theorem{qkpos-minuscule}, \Theorem{qh-powers-intro}, and
\Theorem{qk-powers-intro}.

We apply these results to the projection $p_d : Z_d \to X$ of the
quantum-to-classical construction. For each $x \in X$, the fiber $F_d =
p_d^{-1}(x) \cong P_X/P_{Z_d}$ is a product of cominuscule flag varieties. If
$x$ is a general point of $\Gamma_d(X_u,X^v)$, then $R = F_d \cap Z_d(X_u,X^v)$
is a Richardson variety in $F_d$ given by explicitly determined Weyl group
elements. The most interesting case of property (i) happens in the range of
degrees $\dmax(u^\vee,v) < d \leq \min(\dmax(u^\vee),\dmax(v))$, where
$\dmax(v)$ denotes the unique power of $q$ in $[\pt] \star [X^v]$. In this case,
$\Gamma_{d-1,1}(X_u,X^v) = \Gamma_d(X_u,X^v)$. We show that $D = F_d \cap
Z_{d-1,1}(X_u,X^v)$ is a Cartier divisor in $R$, obtained as the intersection of
$R$ with a divisor in $F_d$. We then prove that $D$ is cohomologically trivial
if and only if $d$ is not an exceptional degree. This is done by explicitly
computing the cohomology groups of any negative line bundle restricted to a
Richardson variety in any cominuscule flag variety (\Theorem{Jcohom}). A special
case is the following statement.

\begin{thm}\label{thm:vanish}%
  Let $X$ be a minuscule flag variety, let $R \subset X$ be a Richardson variety
  of positive dimension, and let $\cO_X(-1) \subset \cO_X$ be the ideal sheaf of
  the Schubert divisor in $X$. Then $H^i(R,\cO_X(-1)) = 0$ for all $i$.
\end{thm}

Our results can also be used to prove that the Seidel representation of the
fundamental group $\pi_1(\Aut(X))$ on the localized quantum cohomology ring
$\QH(X)_{q=1}$, as studied in \cite{belkale:transformation,
chaput.manivel.ea:affine}, has a natural generalization to the quantum
$K$-theory ring in the cominuscule case. We plan to discuss this elsewhere
together with applications to Pieri formulas in quantum $K$-theory.

\subsection{Organization}

This paper is organized as follows. In \Section{schubint} we fix our notation
for flag varieties and related combinatorics and prove several results about
intersections of Schubert varieties in special position that are valid for
arbitrary flag varieties. In particular, we introduce the notion of
\semi-transversal intersections. In \Section{projrich} we give short proofs of
some related results about projections of Richardson varieties, which were first
obtained in \cite{knutson.lam.ea:projections}. In \Section{comin} we introduce
our notation for cominuscule flag varieties. We also compute the (top)
cohomology group of any negative line bundle restricted to a Richardson variety
in a cominuscule flag variety, as a representation of the maximal torus $T
\subset G$. This allows us to determine when the intersection of a Richardson
variety with an effective Cartier divisor is cohomologically trivial.
\Section{qclassical} gives a detailed account of the quantum equals classical
theorem, focusing on Gromov-Witten invariants of degrees no larger than the
\emph{diameter} of a cominuscule variety $X$, that is, the smallest degree of a
rational curve connecting two general points. We take this opportunity to
provide a type-uniform proof of this theorem, something that has so far not been
available in the literature. At the same time we further develop the geometry
and combinatorics of the quantum-to-classical construction. \Section{qcfibers}
provides explicit descriptions of the general fibers of several maps of
varieties related to this construction. \Section{qhinterval} proves that the
$q$-degrees in the quantum cohomology product of two cominuscule Schubert
classes form an integer interval. We also construct our generalization of
Postnikov's cylinder, which provides a combinatorial description of the minimal
and maximal degrees in a quantum product. \Section{quantumk} proves the results
about alternating signs and $q$-degrees in the quantum $K$-theory ring of a
cominuscule flag variety. We also discuss the structure constants of this ring
relative to its basis of dual Schubert classes. The proofs of some technical
facts are postponed to \Section{divisors}. Finally, \Section{claims} proves that
the general fibers of the rational map $M_{d-1,1}(X_u,X^v) \dashrightarrow
Z_{d-1,1}(X_u,X^v)$ have cohomologically trivial completions. The bibliography
is preceded by an index of symbols.

\subsection{Acknowledgments}

Parts of this work was carried out while the authors visited the Hausdorff
Research Institute for Mathematics in Bonn, the Department of Mathematical
Sciences at the University of Copenhagen, or participated in the semester
program in Combinatorial Algebraic Geometry at the Institute for Computational
and Experimental Research in Mathematics at Brown University. We are grateful to
these institutions for their hospitality and stimulating environments. We thank
David Anderson, Jesper Thomsen, and Chenyang Xu for helpful discussions, and we
thank Prakash Belkale and Robert Proctor for making us aware of the references
\cite{belkale:transformation, hagiwara:minuscule*2, green:combinatorics*1}. We
finally thank an anonymous referee for a very careful reading of our manuscript
and for several insightful suggestions.

%% file: schubint.tex

\section{Intersections and fibers of Schubert and Richardson varieties}
\label{sec:schubint}

In this section we fix our notation for flag varieties. In addition, we prove
some results about intersections of Schubert and Richardson varieties in special
position that are valid for arbitrary flag varieties over any algebraically
closed field. In this paper, a \emph{variety} is always reduced but not
necessarily irreducible. By a \emph{point} we will always mean a closed point.
The \emph{fibers} of a morphism are understood to be fibers over closed points,
and the \emph{general fibers} mean the fibers over all closed points in a dense
open subset of the target.

\subsection{Flag varieties}
\label{sec:flagvar}

\targetsec{group}{}%
Let $G$ be a connected linear algebraic group, and fix a maximal torus $T$ and a
Borel subgroup $B$ such that $T \subset B \subset G$. The opposite Borel
subgroup $B^- \subset G$ is defined by $B \cap B^- = R_u(G) T$, where $R_u(G)$
is the unipotent radical. Let $\Phi$ be the root system of $(G,T)$, with
positive roots $\Phi^+$ and simple roots $\Delta \subset \Phi^+$ given by $B$.
This means that $\Phi$ is the set of roots of the reductive quotient $G/R_u(G)$,
see \cite[\S7.4.3]{springer:linear*1}. Let $W = N_G(T)/Z_G(T)$ be the Weyl group
of $G$. The reflection along a root $\al \in \Phi$ is denoted by $s_\al$.

\targetsec{flagvar}{}%
A complete homogeneous $G$-variety $X$ will be called a \emph{flag variety} of
$G$ if the morphism $G \to X$, $g \mapsto g.x$ defined by any point $x \in X$ is
separable. Any such flag variety $X$ contains a unique $B$-invariant point. We
denote the parabolic subgroup stabilizing this point by $P_X \subset G$ and the
point itself by $1.P_X$. We identify $X$ with the quotient $G/P_X$. Each element
$g \in G$ defines a point $g.P_X = g.(1.P_X)$ in $X$. Let $\Phi_X$ be the root
system of (the reductive quotient of) $P_X$, and set $\Phi_X^+ = \Phi^+ \cap
\Phi_X$ and $\Delta_X = \Delta \cap \Phi_X$. Let $W_X = N_{P_X}(T)/Z_{P_X}(T)$
be the Weyl group of $P_X$, and let $W^X \subset W$ be the subset of minimal
representatives of the cosets in $W/W_X$. Each element $u \in W$ defines a
$T$-fixed point $u.P_X \in X$ and the \emph{Schubert varieties} $X_u =
\ov{Bu.P_X}$ and $X^u = \ov{B^-u.P_X}$. We denote the corresponding
\emph{Schubert cells} by $\oX_u = Bu.P_X$ and $\oX^u = B^-u.P_X$. These Schubert
varieties and cells depend only on the coset $u W_X$ in $W/W_X$, and for $u \in
W^X$ we have $\dim(X_u) = \codim(X^u,X) = \ell(u)$.

\targetsec{bruhat}{}%
Let $\leq$ denote the Bruhat order on $W$. For $u, v \in W^X$ we then have $v
\leq u$ $\Leftrightarrow$ $X_v \subset X_u$ $\Leftrightarrow$ $X_u \cap X^v \neq
\emptyset$. In this case the intersection $X_u^v = X_u \cap X^v$ is called a
\emph{Richardson variety}; this variety is reduced, irreducible, rational, and
$\dim(X_u^v) = \ell(u) - \ell(v)$ \cite{richardson:intersections}. The
\emph{Richardson cell} $\oX_u^v = \oX_u \cap \oX^v$ is a dense open subset of
$X_u^v$. Any translate of $X_u^v$ will be called a Richardson variety. In other
words, a non-empty closed subvariety $\Omega \subset X$ is a Richardson variety
if and only if $\Omega = g.X_u^v$ for some $u,v \in W^X$ and $g \in G$.
Similarly, arbitrary translates of Schubert varieties will be called Schubert
varieties.

\targetsec{factor}{}%
Each element $u \in W$ has a unique factorization $u = u^X u_X$ for which $u^X
\in W^X$ and $u_X \in W_X$, called the \emph{parabolic factorization} of $u$
with respect to $P_X$. This factorization is \emph{reduced} in the sense that
$\ell(u) = \ell(u^X) + \ell(u_X)$. The parabolic factorization of the longest
element $w_0 \in W$ is $w_0 = w_0^X w_{0,X}$, where $w_0^X$ and $w_{0,X}$ are
the longest elements in $W^X$ and $W_X$, respectively. The Poincar\'e dual
element of $u \in W^X$ is $u^\vee = w_0 u w_{0,X} \in W^X$, which satisfies
$X^{u^\vee} = w_0.X_u$.

\targetsec{factor2}{}%
Suppose $Y$ is an additional flag variety of $G$ such that $P_X \subset P_Y$,
and let $\pi : X \to Y$ be the projection. We then have $(u_Y)^X = (u^X)_Y$ for
any $u \in W$, so this element of $W_Y \cap W^X$ may be written as $u_Y^X$
without ambiguity. The parabolic factorizations of $u$ with respect to both
$P_Y$ and $P_X$ can be simultaneously expressed as $u = u^Y u_Y^X u_X$. Notice
also that $\pi^{-1}(Y^u) = X^u$ and $\pi^{-1}(Y_u) = X_{u w_{0,Y}^X}$ for any $u
\in W^Y$.

\targetsec{hecke}{}%
We need the \emph{Hecke product} on $W$, which by definition is the unique
associative monoid product satisfying
\[
u \cdot s_\be = \begin{cases}
  u s_\be & \text{if $u < u s_\be$,}\\
  u & \text{if $u > u s_\be$}
\end{cases}
\]
for all $u \in W$ and $\be \in \Delta$. Equivalently, the Hecke product $u \cdot
v$ of $u,v \in W$ is given by $\ov{B (u \cdot v) B} = \ov{B u B v B}$ \cite[\S
8.3]{springer:linear*1}. In particular we have $u.X_v \subset X_{u\cdot v}$. The
product of $u$ and $v$ is reduced (i.e.\ $\ell(uv) = \ell(u)+\ell(v)$) if and
only if $uv = u \cdot v$. Several other useful properties of the Hecke product
can be found in e.g.\ \cite[\S3]{buch.mihalcea:curve}.

\targetsec{weak}{}%
The \emph{left weak Bruhat order} on $W$ is defined by $u \leq_L w$ if and only
if $\ell(w u^{-1}) = \ell(w) - \ell(u)$. Equivalently, there exists $x \in W$
such that $w = xu$ is a reduced factorization of $w$. We also need the
$P_X$-\emph{Bruhat order} $\leq_X$ on $W$, which we define by
\[
  v \leq_X u \ \ \ \ \ \text{if and only if} \ \ \ \ \
  v \leq u \ \ \text{and} \ \ u_X \leq_L v_X \,.
\]
It follows from \Corollary{rich_birat} below that this definition is equivalent
to the definition given in \cite[\S2]{knutson.lam.ea:projections} (see also
\cite{bergeron.sottile:schubert}).

\begin{lemma}\label{lemma:hecke}
  For $x, u \in W$ we have $u \leq_L x \cdot u$ and $(x \cdot u) u^{-1} \leq x$.
\end{lemma}
\begin{proof}
  By definition of the Hecke product, there exists $x' \leq x$ such that $x
  \cdot u = x' u$ and the product $x' u$ is reduced. The lemma follows from
  this.
\end{proof}

\targetsec{fundweight}{}%
Given a positive root $\al \in \Phi^+ \ssm \Phi_X$ there exists a unique
$T$-stable curve $C \subset X$ through $1.P_X$ and $s_\al.P_X$. For any simple
root $\be \in \Delta \ssm \Delta_X$ we then have $\int_C [X^{s_\be}] =
(\al^\vee,\om_\be)$, where $\al^\vee$ denotes the coroot of $\al$ and $\om_\be$
is the fundamental weight corresponding to $\be$ (see
\cite[\S3]{fulton.woodward:quantum}).

Let $Y$ be a variety with a transitive action of a connected algebraic group
$H$, and choose any point $y_0 \in Y$. The associated map $H \to Y$ defined by
$h \mapsto h.y_0$ is locally trivial (in the Zariski topology) if and only if
there exists a local section $s : U \to H$, with $\emptyset \neq U \subset Y$
open, such that $s(y).y_0 = y$ for all $y \in U$. Many actions encountered in
the study of Schubert varieties have this property, including the action of $B$
on a Schubert cell $\oX_u$ and the action of $B \times B$ on a double Bruhat
cell $BuB$; this can be proved using the Bruhat decomposition, see e.g.\
\cite[Prop.~2.2]{buch.chaput.ea:finiteness}. Much more generally, the associated
map $H \to Y$ is locally trivial whenever the stabilizer $H_{y_0} \subset H$ is
a \emph{special group} \cite{serre:espaces*2, grothendieck:torsion}. If $f : Z
\to Y$ is any equivariant morphism of $H$-varieties such that the action of $H$
on $Y$ has a section $s: U \to H$ as above, then $f$ is a locally trivial
fibration. In fact, the map $\varphi : U \times f^{-1}(y_0) \to f^{-1}(U)$
defined by $\varphi(y,z) = s(y).z$ is an isomorphism. We will make repeated use
of this fact throughout this paper.

\subsection{\Semi-transversal intersections}\label{sec:semitrans}

Let $\Omega_1$ and $\Omega_2$ be closed subsets of a flag variety $X$ of the
group $G$. It is customary to study the intersection of general translates of
these varieties, that is, any intersection of the form $\Omega_1 \cap
g.\Omega_2$, where $g$ is a general element of $G$. In particular, the product
of (Chow) cohomology classes is given by $[\Omega_1] \cdot [\Omega_2] =
[\Omega_1 \cap g.\Omega_2]$. In this section we consider the situation where the
intersection of general translates of $\Omega_1$ and $\Omega_2$ is empty. We
then seek to understand non-empty intersections of the form $\Omega_1 \cap
g.\Omega_2$ that are as general as possible. Such intersections will be called
\emph{\semi-transversal} intersections of $\Omega_1$ and $\Omega_2$ (when they
exist). Intuitively, a \semi-transversal intersection is obtained by moving
general translates of $\Omega_1$ and $\Omega_2$ towards each other until they
just meet. \Semi-transversal intersections make sense for actions of arbitrary
algebraic groups, so we will formulate our definition in this setting.

\targetsec{stab}{}%
Let $X$ be an algebraic variety and let $G$ be an algebraic group acting on $X$.
For any subset $\Omega \subset X$ we let $G_\Omega = \{ g \in G \mid g.\Omega =
\Omega \}$ denote the stabilizer. Given two closed subsets $\Omega_1, \Omega_2
\subset X$, we define a subset of $G$ by
\[
  G(\Omega_1,\Omega_2) =
  \{ g \in G \mid \Omega_1 \cap g.\Omega_2 \neq \emptyset \} \,.
\]
This set is stable under the action of $G_{\Omega_1} \times G_{\Omega_2}$
defined by $(a_1,a_2).g = a_1 g a_2^{-1}$, so we have $G_{\Omega_1} G_{\Omega_2}
\subset G(\Omega_1,\Omega_2)$ if and only if $\Omega_1 \cap \Omega_2 \neq
\emptyset$. Notice also that $G(\Omega_2,\Omega_1) = G(\Omega_1,\Omega_2)^{-1}$,
and $G(g_1.\Omega_1, g_2.\Omega_2) = g_1 G(\Omega_1,\Omega_2) g_2^{-1}$ for
$g_1, g_2 \in G$. If $X$ is a complete variety, then $G(\Omega_1,\Omega_2)$ is
closed in $G$; this follows because $G(\Omega_1,\Omega_2) =
p(\mu^{-1}(\Omega_1))$, where $p : G \times \Omega_2 \to G$ is the projection
and $\mu : G \times \Omega_2 \to X$ is defined by the action.

\begin{defn}\label{defn:semitrans}%
  We will say that $\Omega_1$ and $\Omega_2$ \emph{meet \semi-transversally} if
  $G_{\Omega_1} G_{\Omega_2}$ is a dense subset of $G(\Omega_1, \Omega_2)$. A
  \emph{\semi-transversal intersection} of $\Omega_1$ and $\Omega_2$ is any
  subscheme of the form $g_1.\Omega_1 \cap g_2.\Omega_2$ for which
  $g_1.\Omega_1$ and $g_2.\Omega_2$ meet \semi-transversally (with $g_1,g_2 \in
  G$).
\end{defn}

If $\Omega_1$ and $\Omega_2$ meet \semi-transversally, then for all $g$ in a
dense open subset of $G(\Omega_1,\Omega_2)$, the intersection $\Omega_1 \cap
g.\Omega_2$ is a translate of $\Omega_1\cap \Omega_2$, so $\Omega_1 \cap
\Omega_2$ is as general as possible among non-empty intersections. A
\semi-transversal intersection of $\Omega_1$ and $\Omega_2$ exists if and only
if $G(\Omega_1,\Omega_2)$ contains a dense orbit for the action of $G_{\Omega_1}
\times G_{\Omega_2}$, in which case $\Omega_1 \cap g.\Omega_2$ is a
\semi-transversal intersection whenever $g$ belongs to this orbit. Any
\semi-transversal intersection of $\Omega_1$ and $\Omega_2$ is a $G$-translate
of any other. Notice that $\Omega_1 \cap \Omega_2$ may be a \semi-transversal
intersection even though $\Omega_1$ and $\Omega_2$ fail to meet
\semi-transversally. The condition that $\Omega_1$ and $\Omega_2$ meet
\semi-transversally is stronger because it concerns the relative position of the
two varieties and not just their intersection. When the group $G$ is not clear
from the context, we will write ``$G$-\semi-transversal'' to clarify the action.

\begin{example}
\noin{\rm(a)} Let $G$ act trivially on $X$. Then $\Omega_1$ meets $\Omega_2$
\semi-transversally if and only if $\Omega_1 \cap \Omega_2 \neq \emptyset$, in
which case $\Omega_1 \cap \Omega_2$ is the only \semi-transversal
intersection.\smallskip

\noin{\rm(b)} Let $\GL(3)$ act on $\bP^2$. Then a line and a conic meet
\semi-transversally if and only if they have two points in common.\smallskip

\noin{\rm(c)} Let $\GL(4)$ act on $\bP^3$. Now a line and an irreducible curve
of degree two meet \semi-transversally if and only if their intersection is a
single reduced point.\smallskip

\noin{\rm(d)} Let $G = \PGL(2) \times \PGL(2) \rtimes S_2$ be the automorphism
group of $X = \bP^1 \times \bP^1$, and set $L = \bP^1 \times \{0\}$. Then $G_L$
is connected while $G(L,L)$ is disconnected, so no \semi-transversal
intersection exists of $L$ with itself. However, if we restrict the action to
the identity component $G^0 = \PGL(2) \times \PGL(2)$, then $L$ meets itself
\semi-transversally.\smallskip

\noin{\rm(e)} Let $\PGL(n+1)$ act on $\bP^n$, let $H \subset \bP^n$ be a
hyperplane, and let $S \subset \bP^n$ be any hypersurface with finite
automorphism group. Then the dimension of $G_H \times G_S$ is smaller than the
dimension of $G(H,S) = G$, so no \semi-transversal intersection of $H$ and $S$
exists. In particular, transversal intersections may fail to be
\semi-transversal.\smallskip

\noin{\rm(f)} The maximal orthogonal Grassmannian $X = \OG(4,8)$ is one
component of the variety of maximal isotropic subspaces in an orthogonal vector
space $V$ of dimension 8. This space $X$ is a flag variety of both $\SO(8)$ and
its subgroup $\SO(7)$. A Schubert line in $X$ corresponds to a 2-dimensional
isotropic subspace $E \subset V$. Two Schubert lines given by $E_1$ and $E_2$
meet $\SO(8)$-\semi-transversally if and only if $E_1 + E_2$ is a point of $X$,
whereas they meet $\SO(7)$-\semi-transversally only if $E_1 + E_2$ is an
isotropic subspace of dimension 3. It is therefore impossible for two Schubert
lines to meet \semi-transversally for both groups at the same time.\smallskip

\noin{\rm(g)} The projective space $\bP^{2n-1}$ is a flag variety of both
$\SL(2n)$ and its subgroup $\Sp(2n)$ of elements preserving a symplectic form on
a vector space $V$ of dimension $2n$. A Schubert line for the action of
$\Sp(2n)$ is given by a 2-dimensional isotropic subspace $E \subset V$. Two
Schubert lines given by isotropic subspaces $E_1$ and $E_2$ meet
$\SL(2n)$-\semi-transversally if and only if $\dim(E_1 + E_2) = 3$, whereas they
meet $\Sp(2n)$-\semi-transversally if and only if $\dim(E_1 + E_2) = 3$ and
$E_1+E_2$ is not isotropic. Two Schubert lines can therefore meet
$\SL(2n)$-\semi-transversally without meeting $\Sp(2n)$-\semi-transversally.
\end{example}

\subsection{\Semi-transversal intersections of Schubert varieties}

We are mostly interested in \semi-transversal intersections of Schubert
varieties in a flag variety, so assume again that $X$ is a flag variety of a
connected linear algebraic group $G$. Our first result shows that a
\semi-transversal intersection of two Schubert varieties exists and is a
Richardson variety. The explicit description of this Richardson variety is used
to determine the general fibers of the projection of a Richardson variety in
\Theorem{richfib}; cominuscule specializations are obtained in
\Proposition{comint} and \Corollary{fibers_geom}(d). Notice that $X_u \cap
w_0.X_v$ is a \semi-transversal intersection of $X_u$ and $X_v$ whenever this
intersection is not empty; this happens exactly when $\ka = w_0$ in the
following result.

\begin{thm}\label{thm:schubint} Let $u, v \in W^X$, and set $\kappa = u \cdot
  w_{0,X} \cdot v^{-1}$. Then $G(X_u,X_v) = \ov{B \kappa B}$, $w_0 \kappa v \in
  W^X$, and $X_u \cap \kappa.X_v = X_u \cap w_0.X_{w_0 \kappa v}$ as (reduced
  and irreducible) subschemes of $X$. In particular, the Richardson variety $X_u
  \cap w_0.X_{w_0\ka v}$ is a \semi-transversal intersection of $X_u$ and $X_v$.
\end{thm}
\begin{proof}
  Let $\mu : G \times X_v \to X$ be the map defined by $\mu(h,x) = h.x$. If we
  let $G$ act on the left factor of $G \times X_v$, then $\mu$ is
  $G$-equivariant, so $\mu$ is a locally trivial fibration. It follows that
  $\mu^{-1}(X_u) \subset G \times X_v$ is a closed irreducible subvariety. Let
  $p : G \times X_v \to G$ be the first projection and let $B \times B$ act on
  $G \times X_v$ by $(b_1,b_2).(h,x) = (b_1 h b_2^{-1}, b_2.x)$. Then $p$ is
  proper and $B\times B$-equivariant, and since $p^{-1}(g) \cap \mu^{-1}(X_u) =
  \{g\} \times (g^{-1}.X_u \cap X_v) \cong X_u \cap g.X_v$, it follows that
  $G(X_u,X_v) = p(\mu^{-1}(X_u)) \subset G$ is closed, irreducible, and $B\times
  B$-stable. Since $\kappa v \leq u w_{0,X}$ by \Lemma{hecke}, we obtain $\kappa
  v.P_X \in X_u \cap \kappa.X_v \neq \emptyset$, so $B \kappa B \subset
  G(X_u,X_v)$. On the other hand, if $w \in W$ is any element such that $X_u
  \cap w.X_v \neq \emptyset$, then let $v'.P_X \in w^{-1}.X_u \cap X_v$ be any
  $T$-fixed point, with $v' \in W^X$. Then $v' \leq v$ and $w v' \leq u
  w_{0,X}$, which implies $w = (w v') (v')^{-1} \leq (w v') \cdot (v')^{-1} \leq
  u w_{0,X} \cdot v^{-1} = \kappa$. This shows that $G(X_u,X_v) = \ov{B \kappa
  B}$.

  Since $p : \mu^{-1}(X_u) \to G(X_u,X_v)$ is $B\times B$-equivariant, this map
  is locally trivial over the open orbit $B\kappa B$. In particular, the
  intersection $X_u \cap \kappa.X_v$ is reduced. Since $\ell(\kappa v) =
  \ell(\kappa) - \ell(v)$ by \Lemma{hecke}, we have $\ell(w_0 \kappa v) =
  \ell(w_0 \kappa) + \ell(v)$. It follows that $w_0 \kappa.X_v \subset X_{w_0
  \kappa v}$, so we have $X_u \cap \kappa.X_v \subset X_u \cap w_0.X_{w_0 \kappa
  v}$. Using that $\mu : G \times X_v \to X$ is locally trivial, we get
  $\dim(\mu^{-1}(X_u)) = \dim(P_X) + \ell(u) + \ell(v)$, and since $p :
  \mu^{-1}(X_u) \to G(X_u,X_v)$ is locally trivial over $B \kappa B$, we obtain
  $\dim(X_u \cap \kappa.X_v) = \ell(u) + \ell(w_{0,X}) + \ell(v) -
  \ell(\kappa)$. On the other hand, the dimension of the Richardson variety $X_u
  \cap w_0.X_{w_0 \kappa v}$ is bounded by $\dim(X_u \cap w_0.X_{w_0 \kappa v})
  \leq \ell(u) + \ell(w_0 \kappa v) - \dim(X) = \ell(u) + \ell(w_{0,X}) +
  \ell(v) - \ell(\kappa)$. We deduce that $X_u \cap \kappa.X_v = X_u \cap
  w_0.X_{w_0 \kappa v}$ and that $\dim(X_{w_0 \kappa v}) = \ell(w_0 \kappa v)$,
  so that $w_0 \kappa v \in W^X$.
\end{proof}

The following corollary gives a geometric characterization of the weak left
Bruhat order that we have not seen before.

\begin{cor}\label{cor:weakorder} Let $u, v \in W^X$. The following are
  equivalent. {\rm(1)} There exists $g \in G$ such that $X_u \cap g.X_v$ is a
  single point. {\rm(2)} The product $u \cdot w_{0,X} \cdot v^{-1}$ is reduced.
  {\rm(3)} We have $u \leq_L v^\vee$.
\end{cor}
\begin{proof}
  Set $\kappa = u \cdot w_{0,X} \cdot v^{-1}$. It follows from
  \Theorem{schubint} that $X_u \cap \kappa.X_v$ is a point if and only if $w_0
  \kappa v = u^\vee = w_0 u w_{0,X}$, or $\kappa = u w_{0,X} v^{-1}$, so (1) and
  (2) are equivalent. Conditions (2) and (3) are equivalent because we have
  $\ell(v w_{0,X} u^{-1}) = \ell(v w_{0,X}) + \ell(u)$ if and only if
  $\ell(v^\vee u^{-1}) = \ell(v^\vee) - \ell(u)$.
\end{proof}

\begin{remark}
  Using the Stein factorization of the map $p : \mu^{-1}(X_u) \to G(X_u,X_v)$,
  it follows that $X_u \cap g.X_v$ is connected for all $g \in G(X_u,X_v)$.
  However, the intersection $X_u \cap g.X_v$ is not always irreducible, and
  Jesper~F.~Thomsen has shown us an example where $X_u \cap g.X_v$ fails to be
  reduced. For example, the intersection of two Schubert divisors in the
  Grassmannian $\Gr(2,4)$ may be a union of two projective planes, and two
  Schubert divisors in the quadric hypersurface $Q^3 \subset \bP^4$ can meet in
  a double line.
\end{remark}

\begin{remark}
  If $X$ is a flag variety of two different groups $G$ and $H$, and $\Omega_1,
  \Omega_2 \subset X$ are Schubert varieties with respect to both groups, then
  one can show that any \semi-transversal intersection of $\Omega_1$ and
  $\Omega_2$ for the action of $H$ is also a \semi-transversal intersection for
  the action of $G$, up to translation by an automorphism of $X$. In fact, there
  are only three families of flag varieties of groups of distinct Lie types,
  namely odd-dimensional projective spaces $\bP^{2n-1} = A_{2n-1}/P_1 =
  C_n/P_1$, maximal orthogonal Grassmannians $\OG(n,2n) = D_n/P_n =
  B_{n-1}/P_{n-1}$, and the 5-dimensional quadric $Q^5 = B_3/P_1 = G_2/P_2$. The
  first two families consist of (co)minuscule flag varieties, in which case the
  claim follows from \Proposition{comint}. The last case $B_3/P_1 = G_2/P_2$ has
  been checked from \Theorem{schubint} with help from a computer. We will not
  need this fact in the following.
\end{remark}

\subsection{Fibers of Schubert varieties}

Let $Y = G/P_Y$ be an additional flag variety of $G$ such that $P_X \subset
P_Y$, and let $\pi : X \to Y$ be the projection. Then $F = \pi^{-1}(1.P_Y) =
P_Y/P_X$ is a flag variety of $P_Y$. The Schubert varieties in $F$ are the
$B$-orbit closures $F_w = X_w$ for $w \in W_Y$, as well as their
$P_Y$-translates. Similarly, any fiber $\pi^{-1}(g.P_Y) = g.F$ for $g \in G$ is
a flag variety of $g P_Y g^{-1}$. Given a subvariety $\Omega \subset X$, the
fibers of the restricted map $\pi : \Omega \to \pi(\Omega)$ will be called
fibers of $\Omega$. Our next result shows that, if $\Omega$ is a Schubert
variety in $X$, then any fiber $\Omega \cap \pi^{-1}(y)$ with $y \in
\pi(\Omega)$ is a Schubert variety in $\pi^{-1}(y)$.

\begin{thm}\label{thm:schubfib}%
  Any intersection $X_u \cap \pi^{-1}(y)$ defined by $u \in W^X$ and $y \in
  \pi(X_u)$ is a (reduced) Schubert variety in $\pi^{-1}(y)$. Let $u = u^Y u_Y$
  be the parabolic factorization with respect to $P_Y$. Any fiber $X_u \cap
  \pi^{-1}(y)$ given by $y \in \oY_u$ is a translate of $X_u \cap u.F =
  u^Y.F_{u_Y}$.
\end{thm}
\begin{proof}
  We may assume that $G$ is reductive. Let $L \subset P_Y$ be the Levi subgroup
  containing $T$ and set $B_L = B \cap L$. We then have $F_w = \ov{B_L w.P_X}$
  for each $w \in W_Y$. Since $\pi$ is $B$-equivariant, we may assume that $y =
  v.P_Y \in \pi(X_u)$ is a $T$-fixed point, given by an element $v \in W^Y$.
  Then $v \leq u$. Since $v \in W^Y$ we have $v.\al > 0$ for each positive root
  $\al$ of $L$, and this implies that $v B_L v^{-1} \subset B$ and therefore
  $B_L \subset v^{-1} B v \cap L$. Since $v^{-1} B v \cap L$ is a Borel subgroup
  of $L$, we obtain $B_L = v^{-1} B v \cap L$. It follows that $v^{-1}.X_u \cap
  F$ is a closed $B_L$-stable subset of $F$, so this set is a union of
  $B_L$-stable Schubert varieties in $F$. The set of $T$-fixed points in
  $v^{-1}.X_u \cap F$ is $\{ z.P_X \mid z \in W_Y \text{ and } vz \leq uw_{0,X}
  \}$. It follows from \cite[Prop.~2.1]{knutson.lam.ea:projections} that the set
  $\{ z \in W_Y \mid vz \leq u w_{0,X} \}$ has a unique maximal element, say $z'
  \in W_Y$. This implies that $v^{-1}.X_u \cap F = F_{z'}$ (as sets), hence $X_u
  \cap \pi^{-1}(v.P_Y) = v.F_{z'}$ is a translate of this Schubert variety.
  Notice also that, if $y \in \oY_u$, then $v = u^Y$, $z' = u_Yw_{0,X}$, and
  $F_{z'} = F_{u_Y}$. To see that $X_u \cap \pi^{-1}(v.P_Y)$ is reduced, set $Z
  = X_u \cap \pi^{-1}(Y_v)$. Since $Z$ is an intersection of $B$-stable Schubert
  varieties, it follows that $Z$ is a reduced union of Schubert varieties in
  $X$, see \cite[\S2]{brion.kumar:frobenius}. Since the map $\pi : Z \to Y_v$
  has irreducible fibers by the above argument, it follows that $Y_v$ is
  dominated by a unique irreducible component $Z'$ of $Z$, and $Z' \subset X$ is
  a $B$-stable Schubert variety. Since $\pi : Z' \to Y_v$ is $B$-equivariant, it
  is locally trivial over the open cell $\oY_v \subset Y_v$. We conclude that
  $X_u \cap \pi^{-1}(v.P_Y) = Z' \cap \pi^{-1}(v.P_Y)$ is reduced.
\end{proof}

\begin{remark}\label{remark:oppschubfib}
  The parabolic factorization of $u \in W^X$ commutes with dualization in the
  sense that $w_0 u w_{0,X} = (w_0 u^Y w_{0,Y}) (w_{0,Y} u_Y w_{0,X})$. It
  follows that the general fibers of the projection $\pi : X^u \to \pi(X^u)$ are
  translates of the Schubert variety $F^{u_Y}$.
\end{remark}

\subsection{Fibers of Richardson varieties}

We finally consider the fibers $R \cap \pi^{-1}(y)$ of a Richardson variety $R
\subset X$ under the projection of flag varieties $\pi : X \to Y$. While the
fibers of a Richardson variety may fail to be irreducible
\cite[Ex.~3.1]{buch.ravikumar:pieri}, we will show that $R \cap \pi^{-1}(y)$ is
a Richardson variety in $\pi^{-1}(y)$ for all points $y$ in a dense open subset
of $\pi(R)$. Some very special cases of this were proved in
\cite{buch.ravikumar:pieri}. The projection $\pi : R \to \pi(R)$ was also
studied in \cite{billey.coskun:singularities, knutson.lam.ea:projections}, were
it was proved that this map is cohomologically trivial and that the projected
Richardson variety $\pi(R)$ is Cohen-Macaulay with rational singularities.

\begin{thm}\label{thm:richfib}%
  Let $R = X_u \cap w_0.X_v$ be a Richardson variety in $X$, with $u,v \in W^X$.
  Let $u = u^Y u_Y$ and $v = v^Y v_Y$ be the parabolic factorizations with
  respect to $P_Y$. Then for all points $y$ in a dense open subset of $\pi(R)$,
  the fiber $R \cap \pi^{-1}(y)$ is a $G$-translate of a \semi-transversal
  intersection of $F_{u_Y}$ and $F_{v_Y}$ in $F = \pi^{-1}(1.P_Y)$.
\end{thm}
\begin{proof}
  Using \cite[Lemma~8.3.6]{springer:linear*1} we may choose two morphisms
  $\phi_1 : \oY_u \to B$ and $\phi_2 : \oY_v \to B$ such that $\phi_1(y)u.P_Y =
  y$ for all $y \in \oY_u$ and $\phi_2(y)v.P_Y = y$ for all $y \in \oY_v$. For
  each element $w \in W$ we choose a fixed representative $\Dot w \in N_G(T)$.
  Set $\pi(R)^0 = \pi(R) \cap \oY_u \cap w_0.\oY_v$ and define $\psi : \pi(R)^0
  \to G$ by
  \[
  \psi(y) = (\Dot u^Y)^{-1}\, \phi_1(y)^{-1}\, \Dot w_0\, \phi_2(\Dot
  w_0^{-1}.y)\, \Dot v^Y \,.
  \]
  Since $\Dot w_0 \phi_2(\Dot w_0^{-1}.y) \Dot v^Y.P_Y = y = \phi_1(y) \Dot
  u^Y.P_Y$, we have $\psi(y) \in P_Y$ for all $y \in \pi(R)^0$.
  \Theorem{schubfib} implies that $X_u \cap \pi^{-1}(y) = \phi_1(y) u^Y.F_{u_Y}$
  and $w_0.X_v \cap \pi^{-1}(y) = \Dot w_0 \phi_2(\Dot w_0^{-1}.y) v^Y.F_{v_Y}$.
  It follows that translation by $\phi_1(y) \Dot u^Y$ maps the triple of
  varieties $(F_{u_Y}, \psi(y).F_{v_Y}, F)$ to $(X_u \cap \pi^{-1}(y), w_0.X_v
  \cap \pi^{-1}(y), \pi^{-1}(y))$. We deduce that $F_{u_Y} \cap \psi(y).F_{v_Y}
  \neq \emptyset$ for all $y \in \pi(R)^0$, hence $\psi(y) \in P_Y(F_{u_Y},
  F_{v_Y})$. Since \Theorem{schubint} shows that $F_{u_Y} \cap g.F_{v_Y}$ is a
  \semi-transversal intersection in $F$ for all elements $g$ in a dense open
  subset of $P_Y(F_{u_Y}, F_{v_Y})$, it is enough to show that $X_u \cap
  \pi^{-1}(y)$ meets $w_0.X_v \cap \pi^{-1}(y)$ \semi-transversally in
  $\pi^{-1}(y)$ for at least one point $y \in \pi(R)^0$.

  Fix any point $y' = g.P_Y \in \pi(R)^0$. Then $\pi^{-1}(y') = g.F$ is a flag
  variety of the group $gP_Yg^{-1}$, and the set $H = (gP_Yg^{-1})(X_u \cap g.F,
  w_0.X_v \cap g.F)$ of \Sec{semitrans} is an irreducible closed subvariety
  of $gP_Yg^{-1}$ by \Theorem{schubint}. This theorem also implies that $X_u
  \cap g.F$ meets $h w_0.X_v \cap g.F$ \semi-transversally in $g.F$ for all
  elements $h$ in a dense open subset $H^0 \subset H$. Since $B B^-$ is a dense
  open subset of $G$ and $1 \in B B^- \cap H$, it follows that $B B^- \cap H$ is
  another dense open subset of $H$. Let $U^- \subset B^-$ be the unipotent
  radical, let $\rho : B B^- = B \times U^- \to B$ be the first projection, and
  define $f : B B^- \to Y$ by $f(h) = \rho(h)^{-1}.y'$. For $h \in B B^-$ we
  have $\rho(h)^{-1} h \in U^-$, which implies that $X_u \cap w_0.X_v \cap
  \pi^{-1}(f(h)) = \rho(h)^{-1}.(X_u \cap h w_0.X_v \cap \pi^{-1}(y'))$. It
  follows that $f$ restricts to a morphism $f : B B^- \cap H \to \pi(R)$. Since
  $f(1) = y' \in \pi(R)^0$, it follows that $f^{-1}(\pi(R)^0)$ is a dense open
  subset of $H$. Finally, let $h \in f^{-1}(\pi(R)^0) \cap H^0$ be any element
  and set $y = f(h)$. Then we have $y \in \pi(R)^0$, and $X_u \cap \pi^{-1}(y)$
  meets $w_0.X_v \cap \pi^{-1}(y)$ \semi-transversally in $\pi^{-1}(y)$, as
  required.
\end{proof}

\begin{cor}\label{cor:richfib}
  The general fibers of a projection $\pi : R \to \pi(R)$ from a Richardson
  variety are Richardson varieties.
\end{cor}

The following result was proved in \cite[Cor.~3.4]{knutson.lam.ea:projections}
with a different but equivalent definition of the $P_X$-Bruhat order $\leq_X$.

\begin{cor}\label{cor:rich_birat}
  Let $u, v \in W^X$. The projection $\pi : X_u^v \to \pi(X_u^v)$ is a
  birational map of non-empty varieties if and only if $v \leq_Y u$.
\end{cor}
\begin{proof}
  This follows from \Theorem{richfib}, \Corollary{weakorder}, and
  \Remark{oppschubfib}.
\end{proof}

For later use we state the following result, which was proved in
\cite{billey.coskun:singularities, knutson.lam.ea:projections}.

\begin{thm}\label{thm:projrich}
  Let $\pi : X \to Y$ be a projection of flag varieties and let $R \subset X$ be
  a Richardson variety.\smallskip

  \noin{\rm(a)} The image $\pi(R)$ is Cohen-Macaulay and has rational
  singularities.\smallskip

  \noin{\rm(b)} The map $\pi : R \to \pi(R)$ is cohomologically trivial, that is,
  $\pi_*(\cO_R) = \cO_{\pi(R)}$ and $R^j \pi_* \cO_R = 0$ for $j > 0$.
\end{thm}

%% file: projrich.tex

\section{Projected Richardson varieties}
\label{sec:projrich}

We need some additional results about projections of Richardson varieties that
were proved in \cite{knutson.lam.ea:projections}. In this expository section we
give short proofs of these results. The statements of some results deviate
slightly from the original versions, for example the bounds on $u'$ and $v'$ in
\Theorem{richmodel} are important for our applications but not immediately clear
from \cite[Prop.~3.3 and Prop.~3.6]{knutson.lam.ea:projections}. Another
difference is our simpler but equivalent definition of the $P_X$-Bruhat order
from \Section{flagvar}: for $u,v \in W$ we have $v \leq_X u$ if and only if $v
\leq u$ and $u_X \leq_L v_X$.

\targetsec{projrich}{}%
We work over any algebraically closed field. Let $E = G/B$ be the variety of
complete flags, let $X = G/P_X$ be any flag variety of $G$, and let $\pi : E \to
X$ be the projection. Given $v, u \in W$ with $v \leq u$, the images in $X$ of
the corresponding Richardson variety and Richardson cell in $E$ are denoted by
$\Pi_u^v(X) = \pi(E_u^v)$ and $\oPi_u^v(X) = \pi(\oE_u^v)$.

\begin{lemma}\label{lemma:minparab}%
  Let $u, v \in W$ satisfy $v \leq u$, and let $s \in W_X$ be a simple
  reflection such that $u < us$ and $v < vs$. Then the following hold.

  \noin{\rm(a)}
  $\Pi_u^v(X) = \Pi_{us}^{vs}(X) = \Pi_{us}^v(X)$.

  \noin{\rm(b)}
  $\oPi_u^v(X) = \oPi_{us}^{vs}(X)$.

  \noin{\rm(c)}
  $\oPi_{us}^v(X) = \oPi_u^v(X) \cup \oPi_u^{vs}(X)$.

  \noin{\rm(d)} $\pi : \oE_u^v \to \oPi_u^v(X)$ is an isomorphism of varieties
  if and only if $\pi : \oE_{us}^{vs} \to \oPi_{us}^{vs}(X)$ is an isomorphism
  of varieties.
\end{lemma}
\begin{proof}
  We may assume that $P_X$ is the minimal parabolic subgroup defined by $W_X =
  \{1,s\}$. We then have $u, v \in W^X$. Part (a) follows from $\pi(E_u^v) =
  \pi(E_u \cap \pi^{-1}(X^v)) = \pi(E_u) \cap X^v = X_u^v$ and symmetric
  identities. We also have $\pi(\oE_{us}^v) \subset \oX_u^v = \pi(\oE_u) \cap
  \oX^v = \pi(\oE_u \cap \pi^{-1}(\oX^v)) = \pi(\oE_u \cap (\oE^v \cup
  \oE^{vs})) = \pi(\oE_u^v) \cup \pi(\oE_u^{vs})$. Given $x \in \oX_u^v$, the
  fiber $\pi^{-1}(x) \cong \bP^1$ is contained in $\pi^{-1}(\oX_u^v) =
  \oE_{us}^v \cup \oE_u^v \cup \oE_{us}^{vs} \cup \oE_u^{vs}$. Since $\pi$ is
  injective on $\oE_u$ and on $\oE^{vs}$, $\pi^{-1}(x) \cap (\oE_u^v \cup
  \oE_{us}^{vs} \cup \oE_u^{vs})$ is finite, so $\pi^{-1}(x) \cap \oE_{us}^v
  \neq \emptyset$. This proves (c). A symmetric argument gives $\oX_u^v =
  \pi(\oE_{us}^{vs}) \cup \pi(\oE_u^{vs})$. Again using that $\pi$ is injective
  on $\oE_u$ and on $\oE^{vs}$, part (b) follows because $\pi(\oE_u^v) = \oX_u^v
  - \pi(\oE_u^{vs}) = \pi(\oE_{us}^{vs})$. Finally, both maps in part (d) are
  isomorphisms because $\pi : \oE_u \to \oX_u$ and $\pi : \oE^{vs} \to \oX^{vs}$
  are isomorphisms by \cite[Lemma~8.3.6]{springer:linear*1}.
\end{proof}

\begin{prop}\label{prop:rich_iso}
  Let $u,v \in W$ satisfy $v \leq u$. The following are equivalent.\smallskip

  \noin{\rm(a)} We have $v \leq_X u$.\smallskip

  \noin{\rm(b)} The dimension of $\Pi_u^v(X)$ is $\ell(u) - \ell(v)$.\smallskip

  \noin{\rm(c)} The map $\pi : \oE_u^v \to \oPi_u^v(X)$ is an isomorphism of
  varieties.
\end{prop}
\begin{proof}
  If $u \in W^X$, then (a) holds by definition of $\leq_X$, and (b) and (c) hold
  because $\pi : \oE_u \to \oX_u$ is an isomorphism. Otherwise let $s \in W_X$
  be a simple reflection such that $u s < u$. If $v < vs$, then $v \not\leq_X
  u$, and \Lemma{minparab}(a) implies that $\dim \Pi_u^v(X) < \dim E_u^v$. On
  the other hand, if $vs < v$, then $vs \leq_X us$ holds if and only if $v
  \leq_X u$, so the result follows from \Lemma{minparab}(b,d) by induction on
  $\ell(u_X)$.
\end{proof}

If part (c) of \Proposition{rich_iso} is replaced with the requirement that $\pi
: E_u^v \to \pi(E_u^v)$ is birational, then this theorem holds when $E$ is an
arbitrary flag variety and $u,v \in W^E$. However, the following example shows
that $\pi : \oE_u^v \to \pi(\oE_u^v)$ may fail to be injective when $E \neq G/B$
and $v \leq_X u$ in $W^E$.

\begin{example}
  Let $G = \GL(5)$, let $B \subset G$ be the subgroup of upper triangular
  matrices, and let $B \subset P_X \subset P_Y \subset G$ be the parabolic
  subgroups such that $W_X$ is generated by $s_4$ and $W_Y$ is generated by
  $s_1, s_3, s_4$. Then $X = G/P_X = \Fl(1,2,3;5)$ and $Y = G/P_Y = \Gr(2,5)$.
  Let $\pi : X \to Y$ be the projection and set $v = s_3$ and $u = s_3 s_2 s_1
  s_4 s_3 s_2 s_3 = 45213$. Since $u_Y = v_Y = s_3$, it follows from
  \Proposition{rich_iso} or \Corollary{rich_birat} that $\pi : X_u^v \to
  \pi(X_u^v)$ is birational. Let $E = G/B = \Fl(5)$ and $v' = s_3 s_4$, and
  consider the composition of projections
  \[
    \oE_u^{v'} \ \longrightarrow \ \oX_u^v \ \xrightarrow{\ \pi\ }
    \pi(\oX_u^v) \,.
  \]
  Since $u \in W^X$, the map $\oE_u \to \oX_u$ is an isomorphism of affine
  spaces, so the first projection is injective. However, since $u_Y \not\leq_L
  v'_Y$, it follows from \Proposition{rich_iso} that the composed projection is
  not injective. We conclude that $\pi: \oX_u^v \to \pi(\oX_u^v)$ is not
  injective.
\end{example}

\targetsec{simx}{}%
Let $\sim_X$ denote the equivalence relation on the set $\{ (v,u) \in W \times W
\mid v \leq u \}$ generated by $(v,u) \sim_X (v,us) \sim_X (vs,us)$ whenever $s
\in W_X$ is a simple reflection such that $u < us$ and $v < vs$.

\begin{thm}\label{thm:simX}
  Let $u, v, u', v' \in W$ satisfy $v \leq u$ and $v' \leq u'$.\smallskip

  \noin{\rm(a)} If $\Pi_u^v(X) = \Pi_{u'}^{v'}(X)$ and $u,u' \in W^X$, then
  $u=u'$ and $v=v'$.\smallskip

  \noin{\rm(b)} We have $\Pi_u^v(X) = \Pi_{u'}^{v'}(X)$ if and only if $(v,u)
  \sim_X (v',u')$.\smallskip

  \noin{\rm(c)} If $v \leq_X u$ and $v' \leq_X u'$, then either $\oPi_u^v(X) =
  \oPi_{u'}^{v'}(X)$ or $\oPi_u^v(X) \cap \oPi_{u'}^{v'}(X) = \emptyset$.
\end{thm}
\begin{proof}
  It follows from \Lemma{minparab}(a) that $(v,u) \sim_X (v',u')$ implies
  $\Pi_u^v(X) = \Pi_{u'}^{v'}(X)$. Assume that $u,u' \in W^X$ and $\oPi_u^v(X)
  \cap \oPi_{u'}^{v'}(X) \neq \emptyset$. Then $\oX_u \cap \oX_{u'} \neq
  \emptyset$, which implies $u = u'$. Since $\pi$ is injective on $\oE_u$, we
  deduce that $v = v'$ as well. This proves (a), and also establishes (b) and
  (c) when $u, u' \in W^X$. Assume next that $u \notin W^X$. Choose a simple
  reflection $s \in W_X$ such that $\wt u = u s < u$, and set $\wt v = (v \cdot
  s)s$. Then we have $(v,u) \sim_X (\wt v,\wt u)$, $\Pi_u^v(X) = \Pi_{\wt
  u}^{\wt v}(X)$, and $\wt u_X < u_X$, so part (b) follows by induction on
  $\ell(u_X) + \ell(u'_X)$. If $v \leq_X u$, then we must have $\wt v = vs < v$
  and $\wt v \leq_X \wt u$. Since \Lemma{minparab}(b) shows that $\oPi_u^v(X) =
  \oPi_{\wt u}^{\wt v}(X)$, part (c) also follows by induction on $\ell(u_X) +
  \ell(u'_X)$.
\end{proof}

\begin{thm}\label{thm:richmodel}
  Let $u,v \in W$ satisfy $v \leq u$.\smallskip

  \noin{\rm(a)} The set $\oPi_u^v(X)$ is the disjoint union of some of the sets
  $\oPi_{u'}^{v'}(X)$ for which $v \leq v' \leq_X u' \leq u$.\smallskip

  \noin{\rm(b)} We have $\displaystyle\Pi_u^v(X) = \bigcup_{v \leq v' \leq_X u'
  \leq u} \oPi_{u'}^{v'}(X)$.\smallskip

  \noin{\rm(c)} There exists $u', v' \in W$ such that $v \leq v' \leq_X u' \leq
  u$ and $\Pi_u^v(X) = \Pi_{u'}^{v'}(X)$.
\end{thm}
\begin{proof}
  We first prove part (a). If $u \in W^X$, then $v \leq_X u$ and the result is
  clear. Otherwise let $s \in W_X$ be a simple reflection such that $us < u$. If
  $v < vs$, then \Lemma{minparab}(c) shows that $\oPi_u^v(X) = \oPi_{us}^v(X)
  \cup \oPi_{us}^{vs}(X)$, so the result follows by induction on $\ell(u_X)$.
  The obtained union is automatically disjoint by \Theorem{simX}(c). Assume that
  $vs < v$. Then \Lemma{minparab}(b) gives $\oPi_u^v(X) = \oPi_{us}^{vs}(X)$,
  and by induction on $\ell(u_X)$ we can express $\oPi_u^v(X)$ as a union of
  sets $\oPi_{u'}^{v'}(X)$ for which $vs \leq v' \leq_X u' \leq us$. Any such
  set $\oPi_{u'}^{v'}(X)$ for which $v \leq v'$ satisfies the requirements of
  part (a), so assume that $vs \leq v' \leq_X u' \leq us$ and $v \not\leq v'$.
  Then we must have $v' < v's$. Since $v' \leq_X u'$, we also obtain $u' < u's$,
  so \Lemma{minparab}(b) shows that $\oPi_{u'}^{v'}(X) = \oPi_{u's}^{v's}(X)$.
  Since $v \leq v's \leq_X u's \leq u$, this set has the required form. This
  completes the proof of (a). Part (b) follows from (a) because $E_u^v$ is the
  union of all sets $\oE_{u'}^{v'}$ for which $v \leq v' \leq u' \leq u$, and
  (c) follows because $\oPi_{u'}^{v'}(X)$ must be dense in $\oPi_u^v(X)$ for
  some $u', v' \in W$ with $v \leq v' \leq_X u' \leq u$.
\end{proof}

\begin{cor}\label{cor:rich2rich}
  If $v \leq u$ in $W^X$, then $X_u^v = \Pi_u^v(X)$ and $\oX_u^v = \oPi_{u
  w_{0,X}}^v(X)$.
\end{cor}
\begin{proof}
  The first equality is true because $X_u^v = \pi(E_u) \cap X^v = \pi(E_u \cap
  \pi^{-1}(X^v)) = \Pi_u^v(X)$. For the second, notice first that $\oPi_{u
  w_{0,X}}^v(X) \subset \pi(\oE_{u w_{0,X}}) \cap \pi(\oE^v) = \oX_u^v$. We also
  have $\oX_u^v = \oX_u \cap \pi(\oE^v) = \pi(\pi^{-1}(\oX_u) \cap \oE^v) =
  \bigcup_{x \in W_X} \oPi_{ux}^v(X)$. Since \Lemma{minparab}(c) implies that
  $\oPi_{ux}^v(X) \subset \oPi_{u w_{0,X}}^v(X)$, we deduce that $\oX_u^v
  \subset \oPi_{u w_{0,X}}^v(X)$.
\end{proof}

%% file: comin.tex

\section{Cominuscule flag varieties}
\label{sec:comin}

\subsection{Schubert varieties}\label{sec:comin-schubert}

\targetsec{gamma}{}%
In the remainder of this paper we let $X = G/P_X$ be a \emph{cominuscule flag
variety} defined over $\C$. This means that $P_X$ is a maximal parabolic
subgroup of $G$, and the unique simple root $\ga$ in $\Delta \ssm \Delta_X$ is
cominuscule, that is, when the highest root of $\Phi$ is expressed as a linear
combination of simple roots, the coefficient of $\ga$ is one. If in addition the
root system $\Phi$ is simply laced, then $X$ is also called \emph{minuscule}. It
was proved by Proctor that the Bruhat order on $W^X$ is a distributive lattice
that agrees with the left weak Bruhat order \cite{proctor:bruhat}. Stembridge
has proved that all elements of $W^X$ are \emph{fully commutative}, which means
that any reduced expression of an element of $W^X$ can be obtained from any
other by interchanging commuting simple reflections \cite{stembridge:fully}. We
proceed to summarize the facts we need in more detail, following the notation
from \cite{buch.chaput.ea:chevalley}. Proofs of our claims can be found in
\cite{proctor:bruhat, stembridge:fully, perrin:small*1, buch.samuel:k-theory}.

\targetsec{poset}{}%
The root lattice $\Span_\Z(\Delta)$ has a partial order defined by $\al' \leq
\al$ if and only if $\al-\al'$ is a sum of positive roots. Set $\cP_X = \{ \al
\in \Phi \mid \al \geq \ga \}$, with the induced partial order. For any element
$u \in W$ we let $I(u) = \{ \al \in \Phi^+ \mid u.\al < 0 \}$ denote the
inversion set. We then have $\ell(u) = |I(u)|$, and $u \in W^X$ if and only if
$I(u) \subset \cP_X$. Moreover, the assignment $u \mapsto I(u)$ restricts to a
bijection between the elements of $W^X$ and the (lower) order ideals of $\cP_X$.
This bijection is an order isomorphism in the sense that $v \leq u$ if and only
if $I(v) \subset I(u)$. The order ideals in $\cP_X$ generalize the Young
diagrams known from the Schubert calculus of classical Grassmannians. For this
reason the roots in $\cP_X$ will sometimes be called \emph{boxes}. An order
ideal in $\cP_X$ will be called a \emph{straight shape}, and a difference
between order ideals is called a \emph{skew shape}.

\targetsec{cominrep}{}%
Given a straight shape $\la \subset \cP_X$, let $\la = \{ \al_1, \al_2, \dots,
\al_{|\la|} \}$ be any \emph{increasing ordering} of its elements, i.e.\ $\al_i
< \al_j$ implies $i<j$. Then the element of $W^X$ corresponding to $\la$ is the
product of reflections $w_\la = s_{\al_1} s_{\al_2} \dots s_{\al_{|\la|}}$ (see
e.g.\ \cite[Thm.~2.4]{buch.samuel:k-theory}). Given $v, u \in W^X$ with $v \leq
u$, we will use the notation $u/v = u v^{-1} \in W$. Since the Bruhat order on
$W^X$ agrees with the left weak Bruhat order \cite{proctor:bruhat} (see also
\cite[Thm.~7.1]{stembridge:fully}), we have $\ell(u/v) = \ell(u)-\ell(v)$.

\targetsec{labeling}{}%
For any root $\al \in \cP_X$, consider the shape $\la(\al) = \{ \al' \in \cP_X
\mid \al' < \al \}$, and set $\delta(\al) = w_{\la(\al)}.\al$. Then
$s_{\delta(\al)} = w_{\la(\al)} s_\al w_{\la(\al)}^{-1} = w_{\la(\al) \cup
\{\al\}}/w_{\la(\al)}$ has length one. It follows that $\delta : \cP_X \to
\Delta$ is a labeling of the boxes in $\cP_X$ by simple roots. Examples of this
labeling are provided in \Table{tablez1}.

\input{tablez1}

The element $u/v$ depends only on the skew shape $I(u) \ssm I(v)$. If $I(u) \ssm
I(v) = \{ \al_1, \al_2, \dots, \al_\ell \}$ is any increasing ordering, then
$u/v = s_{\delta(\al_\ell)} \cdots s_{\delta(\al_2)} s_{\delta(\al_1)}$ is a
reduced expression for $u/v$. In the special case $v=1$, every reduced
expression for $u$ can be obtained in this way. We will say that $u/v$ is a
\emph{rook strip} if this element of $W$ is a product of commuting simple
reflections. Equivalently, no pair of roots in $I(u)\ssm I(v)$ are comparable by
the partial order $\leq$ on $\cP_X$. We call $u/v$ a \emph{short rook strip} if
it is a product of commuting reflections defined by short simple roots. Notice
that if the root system $\Phi$ is simply laced, then all roots are long by
convention, so $u/v$ is a short rook strip if and only if $u=v$.

The following result holds for any (reduced and crystallographic) root system
$\Phi$.

\begin{lemma}\label{lemma:invcover}%
  Let $\Phi$ be a root system with associated Weyl group $W$, let $w \in W$, and
  let $\al' \lessdot \al$ be a covering relation in $I(w)$. Then $\al - \al' \in
  \Phi^+$.
\end{lemma}
\begin{proof}
  We can write $\al = \al' + \be_1 + \dots + \be_k$, with $\be_i \in \Phi^+$ and
  $\be_i+\be_j \notin \Phi$ for all $i$ and $j$. By
  \cite[Thm.~2.4]{gay.pilaud:weak} we may assume $\al' + \be_1 + \dots +
  \be_{k-1} \in \Phi^+$, and \cite[Thm.~2.5]{gay.pilaud:weak} implies that $\al'
  + \be_k \in \Phi^+$. Using that $w.\al'$ and $w.\al$ are negative roots, it
  follows that $\al'+\be_1+\dots+\be_{k-1}$ or $\al'+\be_k$ belongs to $I(w)$.
  In both cases, since $\al' \lessdot \al$ is a covering relation, we deduce
  that $k=1$.
\end{proof}

\begin{lemma}\label{lemma:skew}%
  Let $u, v \in W^X$ satisfy $v \leq u$.\smallskip

  \noin{\rm(a)} The action of $v$ defines an order isomorphism $v : I(u)\ssm
  I(v) \to I(u/v)$.\smallskip

  \noin{\rm(b)} Let $\al \in \Phi$. Then $\al$ is a minimal box of $\cP_X \ssm
  I(v)$ if and only if $\al \geq \ga$ and $v.\al \in \Delta$. In this case we
  have $v.\al = \delta(\al)$.
\end{lemma}
\begin{proof}
  We have $v.(I(u) \ssm I(v)) \subset I(u/v)$ by definition of the inversion
  sets, and both sets have the same cardinality. The map $\al \mapsto v.\al$ is
  order-preserving on all of $\cP_X$ since $v.\be > 0$ for all $\be \in \Delta
  \ssm \{\ga\}$. To prove that $v^{-1} : I(u/v) \to I(u) \ssm I(v)$ also
  preserves the order, let $\al' \lessdot \al$ be a covering relation in
  $I(u/v)$. Then $v^{-1}.\al - v^{-1}.\al' = v^{-1}(\al - \al') \in \Phi$ is a
  root by \Lemma{invcover}, so $v^{-1}.\al'$ and $v^{-1}.\al$ are comparable
  elements in $\cP_X$. Since $v$ is order-preserving and $\al' < \al$, we deduce
  that $v^{-1}.\al' < v^{-1}.\al$. This proves (a). If $\al$ is a minimal box of
  $\cP_X \ssm I(v)$, then $\la(\al) \subset I(v)$, and any box $\al' \in I(v)
  \ssm \la(\al)$ is incomparable to $\al$, hence $s_{\al'}.\al = \al$ by
  \cite[Lemma~2.2]{buch.samuel:k-theory}. This implies that $v.\al =
  w_{\la(\al)}.\al = \delta(\al) \in \Delta$. On the other hand, the conditions
  $\al \geq \ga$ and $v.\al \in \Delta$ imply that $\al \in \cP_X \ssm I(v)$.
  Let $\al' \in \cP_X \ssm I(v)$ be any minimal box such that $\al' \leq \al$.
  Since $0 < v.\al' \leq v.\al \in \Delta$, we must have $\al' = \al$, which
  proves (b).
\end{proof}

\begin{remark}\label{remark:skew}%
  \noin{\rm(a)} If $\al_1 \neq \al_2 \in \cP_X$ are incomparable boxes, then
  \Lemma{skew}(b) implies that $\delta(\al_1) = w_\la.\al_1 \neq w_\la.\al_2 =
  \delta(\al_2)$ where $\la = \la(\al_1) \cup \la(\al_2)$. In addition, we have
  $(\delta(\al_1),\delta(\al_2)) = (\al_1,\al_2) = 0$ by
  \cite[Lemma~2.2]{buch.samuel:k-theory}.\smallskip

  \noin{\rm(b)} If $\al_1 \lessdot \al_2$ is a covering relation in $\cP_X$,
  then $(\al_1,\al_2)>0$. In fact, since $\al_1$ is a maximal box of
  $\la(\al_2)$, we obtain $(\al_1,\al_2) = (w_{\la(\al_2)}.\al_1,
  w_{\la(\al_2)}.\al_2) = (-\delta(\al_1),\delta(\al_2)) \geq 0$. If
  $(\al_1,\al_2)=0$, then $s_{\delta(\al_1)}$ and $s_{\delta(\al_2)}$ are
  commuting simple reflections. Set $\la = \la(\al_2) \ssm \{\al_1\}$. Since
  $s_{\delta(\al_1)} s_{\delta(\al_2)} w_\la \in W^X$ is a reduced product, it
  follows that $s_{\delta(\al_2)} w_\la \in W^X$ and $\la \subsetneq
  I(s_{\delta(\al_2)} w_\la) \subsetneq \la \cup \{\al_1,\al_2\}$. But then
  $I(s_{\delta(\al_2)} w_\la) = \la \cup \{\al_1\}$, a contradiction.
\end{remark}

\begin{lemma}\label{lemma:PXinvol}%
  \noin{\rm(a)} The action $w_{0,X} : \cP_X \to \cP_X$ is an order-reversing
  involution, and $\delta(w_{0,X}.\al) = -w_0.\delta(\al)$ is the Cartan
  involution of $\delta(\al)$ for each $\al \in \cP_X$.

  \noin{\rm(b)} The Poincar\'e dual element of $u \in W^X$ is determined by
  $I(u^\vee) = I(w_0 u w_{0,X}) = \cP_X \ssm w_{0,X}.I(u)$.
\end{lemma}
\begin{proof}
  The action of $w_{0,X}$ is an order-reversing involution on $\cP_X$ since it
  does not change the coefficient of $\ga$, and $w_{0,X}.\be < 0$ for each $\be
  \in \Delta \ssm \{\ga\}$. For $u \in W^X$ and $\al \in \Phi^+$ we deduce that
  $w_0 u w_{0,X}.\al < 0$ holds if and only if $w_{0,X}.\al \in \cP_X \ssm
  I(u)$, which proves (b). Since $w_{0,X}.\al$ is a maximal box of $\cP_X \ssm
  w_{0,X}.\la(\al)$, it follows from \Lemma{skew}(b) that $\delta(w_{0,X}.\al) =
  - w_{\la(\al)}^\vee.(w_{0,X}.\al) = - w_0 w_{\la(\al)}.\al =
  -w_0.\delta(\al)$, which completes the proof of (a).
\end{proof}

\targetsec{distlattice}{}%
The Bruhat order on $W^X$ is a distributive lattice, with join and meet
operations defined by $I(u \cup v) = I(u) \cup I(v)$ and $I(u \cap v) = I(u)
\cap I(v)$. Notice that $(u \cup v)/v = u/(u \cap v)$. It follows that $u/(u
\cap v)$ and $v/(u \cap v)$ are commuting elements of $W$, as their product in
either order is $(u \cup v)/(u \cap v)$. We also have $(u \cap v)^\vee = u^\vee
\cup v^\vee$ and $(u \cup v)^\vee = u^\vee \cap v^\vee$.

\begin{prop}\label{prop:comint}%
  Let $u, v \in W^X$. The Richardson variety $X_u^{u \cap v}$ is a
  \semi-transversal intersection of $X_u$ and $X^v$ in $X$.
\end{prop}
\begin{proof}
  Set $z = v^\vee = w_0 v w_{0,X}$ and $\kappa = u \cdot w_{0,X} \cdot z^{-1}$.
  It follows from \Theorem{schubint} that $X_u \cap w_0.X_{w_0 \kappa z} =
  X_u^{\kappa z w_{0,X}}$ is a \semi-transversal intersection of $X_u$ and $X^v$.
  We must therefore show that, for all $u, z \in W^X$ we have
  \[
  (u \cdot w_{0,X} \cdot z^{-1}) z w_{0,X} = u \cap z^\vee \,.
  \]
  Set $u' = u \cap z^\vee$. Since $z^\vee w_{0,X} z^{-1} = w_0$ is a reduced
  product and $u' \leq_L z^\vee$, it follows that $u' \cdot w_{0,X} \cdot z^{-1}
  = u' w_{0,X} z^{-1}$ is also a reduced product, so we obtain $(u' \cdot
  w_{0,X} \cdot z^{-1}) z w_{0,X} = u' = u \cap z^\vee$. Let $\al \in I(u) \ssm
  I(z^\vee)$ and set $\be = \delta(\al)$. It suffices to show that $s_\be \cdot
  (u' w_{0,X} z^{-1}) = u' w_{0,X} z^{-1}$. We may assume that $s_\be u' w_{0,X}
  > u' w_{0,X}$, in which case $s_\be u' \in W^X$. We then have $I(s_\be u') =
  I(u') \cup \{\al'\}$ for some root $\al'$ with $\delta(\al') = u'.\al' = \be$.
  Since $s_\be u' \not\leq z^\vee$ we have $\al' \notin I(z^\vee)$. It follows
  that $(u' w_{0,X} z^{-1})^{-1}.\be = z w_{0,X}.\al' = w_0 z^\vee.\al' < 0$,
  and hence $s_\be \cdot (u' w_{0,X} z^{-1}) = u' w_{0,X} z^{-1}$, as required.
\end{proof}

\subsection{Cohomology of negative line bundles on Richardson varieties}
\label{sec:cohom-rich}

\targetsec{ktheory}{}%
Let $K_T(X)$ denote the $T$-equivariant $K$-theory ring of $X$, see e.g.\
\cite[\S2.1]{buch.chaput.ea:chevalley} and the references therein. Pullback
along the structure morphism $X \to \{\pt\}$ makes $K_T(X)$ an algebra over the
ring $K_T(\pt)$ of virtual representations of $T$. Let $\euler{X} : K_T(X) \to
K_T(\pt)$ be the pushforward along the structure morphism. The Schubert classes
in $K_T(X)$ are denoted by $\cO^v = [\cO_{X^v}]$ and $\cO_u = [\cO_{X_u}]$. Let
$J \subset \cO_X$ be the ideal sheaf of the Schubert divisor $X^{s_\ga}$. Then
$J^{-1}$ is the ample generator of $\Pic(X)$. In addition, $J$ inherits a
structure of $T$-equivariant line bundle from $\cO_X$. An equivalent definition
is $J = (G \times^P \C_{\om_\ga}) \otimes \C_{-\om_\ga}$, where $\om_\ga$ is the
fundamental weight corresponding to the cominuscule simple root $\ga$, see
\cite[\S4.1]{buch.chaput.ea:chevalley}.

Given any integer $p \in \Z$, we set $p' = p - \frac{1}{2}$. The half-integers
$\frac{1}{2}\Z$ is then the set of primed and unprimed integers.

\begin{defn}\label{defn:dec_primed}%
  Let $\cS \subset \cP_X$ be a skew shape. A \emph{decreasing primed tableau} of
  shape $\cS$ is a labeling $\cT : \cS \to \frac{1}{2}\Z$ such that (i) $\al' <
  \al$ in $\cS$ implies $\cT(\al') > \cT(\al)$, and (ii) $\cT(\al) \in \Z$ for
  all long boxes $\al \in \cS$.
\end{defn}

\targetsec{labelweight}{}%
Given any labeling $\hT : \cP_X \to \frac{1}{2}\Z$ of $\cP_X$ such that
$\hT(\al) \in \Z$ for each long box $\al \in \cP_X$, define the weight
\[
  \la(\hT) = \sum_{\al \in \cP_X} \hT(\al)\, (\om_\ga, \al^\vee)\,
  \delta(\al) \,.
\]
Notice that, if $\hT(\al)$ is not an integer, then $\al$ is a
short root, hence $(\om_\ga,\al^\vee) = 2$.

\targetsec{tabrep}{}%
Let $u, v \in W^X$ satisfy $v \leq u$, let $m \in \Z$, and let $a \in
\frac{1}{2}\Z$. Given a decreasing primed tableau $\cT$ of shape $I(u)\ssm
I(v)$, let $\cT[m]: \cP_X \to \frac{1}{2}\Z$ denote the extension of $\cT$ that
maps all boxes of $I(v)$ to $m$ and maps all boxes of $\cP_X \ssm I(u)$ to $0$,
see \Example{topcohom}. Using this notation, we define a representation of $T$
by
\[
  C^v_{u,[a,m)} = \bigoplus_{\cT} \C_{-\la(\cT[m])} \,,
\]
where the sum is over all decreasing primed tableaux $\cT$ of shape $I(u)\ssm
I(v)$ with labels in $[a,m)$, i.e.\ $a \leq \cT(\al) < m$ for all $\al \in
I(u)\ssm I(v)$.

\begin{lemma}\label{lemma:tabrep}%
  Let $u, v \in W^X$, $m, p \in \Z$ and $a \in \frac12\Z$, and assume that $v
  \leq u$, $a \leq m$, and $p \geq 0$. Then
  \[
    C^v_{u,[a,m+p)} \cong \bigoplus_{w\in W^X:\, v \leq w \leq u}
    C^v_{w,[0,p)} \otimes_\C C^w_{u,[a,m)} \,.
  \]
\end{lemma}
\begin{proof}
  Given a decreasing primed tableau $\cT$ of shape $I(u)\ssm I(v)$ with labels
  in $[a,m+p)$, let $\cT''$ be the tableau consisting of the boxes with labels
  smaller than $m$, and let $\cT'$ be the tableau obtained by subtracting $m$
  from all boxes with labels greater than or equal to $m$. Then $\cT'$ has shape
  $I(w)\ssm I(v)$ and $\cT''$ has shape $I(u)\ssm I(w)$ for a unique element $w
  \in W^X$ with $v \leq w \leq u$, $\cT'$ has labels in $[0,p)$, $\cT''$ has
  labels in $[a,m)$, and we have $\cT[m+p] = \cT'[p] + \cT''[m]$ with pointwise
  addition. The assumption $a \leq m$ ensures that the assignment $\cT \mapsto
  (\cT',\cT'')$ has a well defined inverse map. The lemma follows from this.
\end{proof}

The following identities generalize Theorems 3.6 and 3.7 from
\cite{buch.chaput.ea:chevalley}. A more general Chevalley formula that holds in
the $K$-theory of arbitrary flag varieties was proved in
\cite[Thm.~13.1]{lenart.postnikov:affine}.

\begin{prop}\label{prop:chevalley}
  {\rm(a)} For $v \in W^X$ and $m \geq 0$ we have in $K_T(X)$ that
  \[
    [J]^m \cdot \cO^v \ = \ \sum_{u \in W^X:\, v \leq u}
    (-1)^{\ell(u/v)}\, [C^v_{u,[0,m)}]\, \cO^u \,.
  \]

  \noin{\rm(b)} For $v \leq u$ in $W^X$ and $m \geq 1$ we have in $K_T(\pt)$
  that
  \[
    \chi(X_u^v, J^m) \ = \ (-1)^{\ell(u/v)}\,
    [C^v_{u,[\frac{1}{2},m)}] \,.
  \]
\end{prop}
\begin{proof}
  Part (a) is clear for $m=0$ and is equivalent to
  \cite[Thm.~3.6]{buch.chaput.ea:chevalley} for $m=1$. For $m \geq 2$ it follows
  by induction using \Lemma{tabrep}. Let $\cI^w \in K_T(X)$ be dual to $\cO_w$,
  i.e.\ $\euler{X}(\cO_u \cdot \cI^w) = \delta_{u,w}$ for $u \in W^X$. By
  \cite[Lemma~3.5]{buch.chaput.ea:chevalley} we have $\cI^w = \sum_u
  (-1)^{\ell(u/w)} \cO^u$, the sum over all $u \geq w$ for which $u/w$ is a rook
  strip. Using that $C^v_{u,[0,m)} = \bigoplus_w C^v_{w,[\frac{1}{2},m)}$, with
  the sum over all $w \in W^X$ for which $v \leq w \leq u$ and $u/w$ is a rook
  strip, it follows that part (a) for $m \geq 1$ is equivalent to the identity
  \[
  [J]^m \cdot \cO^v \ = \ \sum_{w \in W^X: v \leq w}
  (-1)^{\ell(w/v)}\, [C^v_{w,[\frac{1}{2},m)}]\, \cI^w \,.
  \]
  Part (b) follows by multiplying both sides by $\cO_u$ and applying
  $\euler{X}$.
\end{proof}

\begin{thm}\label{thm:Jcohom}%
  Let $u, v \in W^X$ satisfy $v \leq u$ and let $m \geq 1$. Then $H^i(X_u^v,J^m)
  = 0$ for all $0 \leq i < \dim(X_u^v) = \ell(u/v)$. Moreover, we have
  $H^{\ell(u/v)}(X_u^v,J^m) \cong C^v_{u,[\frac{1}{2},m)}$ as representations of
  $T$.
\end{thm}
\begin{proof}
  To prove the vanishing of cohomology groups, we may assume that $X = G/P_X$ is
  defined over an algebraically closed field of positive characteristic \cite[\S
  1.6]{brion.kumar:frobenius}. Then \cite[Thm.~2.3.1]{brion.kumar:frobenius}
  together with \cite[Lemma~1.1.8]{brion.kumar:frobenius} applied to the
  projection $G/B \to X$ shows that $X_u^v$ is Frobenius split. Since $J^{-1}$
  is ample and $X_u^v$ is Cohen-Macaulay and irreducible, the Kodaira vanishing
  theorem for split varieties \cite[Thm.~1.2.9]{brion.kumar:frobenius} implies
  that $H^i(X_u^v,J^m) = 0$ for $ i < \dim(X_u^v)$. We therefore have
  $\chi(X_u^v, J^m) = (-1)^{\ell(u/v)} [H^{\ell(u/v)}(X_u^v,J^m)]$, so the
  result follows from \Proposition{chevalley}.
\end{proof}

\begin{remark}\label{remark:smt}%
  Using standard monomial theory, it is possible to compute the cohomology
  groups of the restriction of any ample line bundle on $G/P$ to a Richardson
  variety; see \cite[Thm.~3]{brion.lakshmibai:geometric} or
  \cite[Thm.~20]{lakshmibai.littelmann:richardson}. However, we have not seen
  the computation of the (top) cohomology of negative line bundles in the
  literature, and this cannot be deduced using Serre duality since the canonical
  sheaf of a Richardson variety is not a line bundle in general.
\end{remark}

\begin{example}\label{example:topcohom}%
  Let $X = \LG(3,6) = C_3/P_3$ be the Lagrangian Grassmannian of maximal
  isotropic subspaces in a complex symplectic vector space of dimension 6. Let
  $\Delta = \{\be_1, \be_2, \be_3\}$ be the set of simple roots, where $\ga =
  \be_3$ is the long root. The labeling $\delta : \cP_X \to \Delta$ is given by
  the following diagram, where the upper-left box represents $\ga$ and the
  bottom-right box represents the highest root of $\Phi$.
  \[
    \tableau{14}{{\be_3}&{\be_2}&{\be_1}\\&{\be_3}&{\be_2}\\&&{\be_3}}
  \]
  Set $v = s_2 s_3$ and $u = s_2 s_3 s_1 s_2 s_3$. Then we have
  \[
    H^3(X_u^v, J^{\otimes 2}) \ \cong \ C^v_{u,[\frac{1}{2},2)} =
    \C_{-2\be_1 - 5\be_2 - 3\be_3} \oplus \C_{-3\be_1 - 5\be_2 - 3\be_3} \,.
  \]
  The extensions $\cT[2]$ of the decreasing primed tableaux $\cT$ corresponding
  to the summands are displayed below. The coefficient of the simple root
  $\be_i$ in the weight obtained from of each tableau is the (negative) sum of
  the half-integers in the $i$-th diagonal, multiplied by 2 if $\be_i$ is short.
  \[
    \tableau{12}{[ltb]{\scriptstyle 2}&
      [tbr]{\scriptstyle 2}&[aA]1\\
      &{1}&{1'}\\&&[a]{\scriptstyle 0}}
    \hspace{10mm}
    \tableau{12}{[ltb]{\scriptstyle 2}&
      [tbr]{\scriptstyle 2}&[aA]2'\\
      &{1}&{1'}\\&&[a]{\scriptstyle 0}}
  \]
\end{example}

An algebraic variety $D$ is called \emph{cohomologically trivial} if
$H^0(D,\cO_D) = \C$ and $H^i(D,\cO_D) = 0$ for $i > 0$.

\begin{cor}\label{cor:trivial}%
  Let $D \subset X_u^v$ be an effective Cartier divisor of class $[D] =
  m[X^{s_\ga}] |_{X_u^v}$. Then $D$ is cohomologically trivial if and only if
  there are no decreasing primed tableaux of shape $I(u)\ssm I(v)$ with labels
  in $[\frac{1}{2}, m) \cap \frac{1}{2}\Z$.
\end{cor}
\begin{proof}
  The Richardson variety $X_u^v$ is cohomologically trivial, as it is rational
  with rational singularities. The long exact sequence of cohomology groups
  derived from $0 \to J^m|_{X_u^v} \to \cO_{X_u^v} \to \cO_D \to 0$ then shows
  that $D$ is cohomologically trivial if and only if $H^i(X_u^v, J^m) = 0$ for
  all $i$. The result therefore follows from \Theorem{Jcohom}.
\end{proof}

\begin{example}
  Let $X_u^v$ be a Richardson variety of positive dimension and let $D \subset
  X_u^v$ be a Cartier divisor. If $[D] = [X^{s_\ga}] |_{X^u_v}$, then $D$ is
  cohomologically trivial if and only if $u/v$ is not a short rook strip. In
  particular, $D$ is cohomologically trivial if $X$ is minuscule. If $[D] = 2
  [X^{s_\ga}] |_{X_u^v}$ and $X$ is minuscule, then $D$ is cohomologically
  trivial if and only if $u/v$ is not a rook strip.
\end{example}

If $X = X_1 \times \dots \times X_k$ is a product of cominuscule flag varieties,
then the Schubert varieties in $X$ are given by sequences $(\la_1, \dots,
\la_k)$ of order ideals $\la_i \subset \cP_{X_i}$. Such a sequence can be
identified with an order ideal in the disjoint union $\cP_X = \cP_{X_1} \coprod
\dots \coprod \cP_{X_k}$. We will consider a point as a product of (zero)
cominuscule varieties with associated set $\cP_{\{\pt\}} = \emptyset$. The
results in this section hold for products of cominuscule varieties with minor
modifications. In \Lemma{skew}(b), the condition $\al \geq \ga$ can be replaced
with $\al \in \cP_X$. In the results of \Section{cohom-rich}, $J$ should be the
ideal sheaf of the union of the Schubert divisors $X^{s_{\ga_i}}$ for $1 \leq i
\leq k$.

%% file: tablez1.tex

\def\vmm#1{\vspace{#1mm}}
\begin{table}
\caption{Partially ordered sets $\cP_X$ of cominuscule varieties with $I(z_1)$
  highlighted (see \Definition{zdkappad}). In each case the partial order is
  given by $\al' \leq \al$ if and only if $\al'$ is weakly north-west of $\al$.}
\label{tab:tablez1}
\begin{tabular}{|c|c|}
\hline
&\vmm{-2}\\
Grassmannian $\Gr(3,7)$ of type A & Max.\ orthog.\ Grassmannian $\OG(6,12)$
\\
&\vmm{-3}\\
\pic{1}{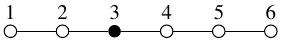} &\\
&\\
&\vmm{-3}\\
$\tableau{12}{
[aLlTt]3 & [aTtBb]4 & 5 & [aTtBbRr]6 \\
[aLlRr]2 & [a]3 & 4 & 5 \\
[aLlRrBb]1 & [a]2 & 3 & 4
}$
&\vmm{-27}\\
&\pic{1}{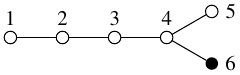}
\\ &
$\tableau{12}{
[aLlTtBb]6 & [aTt]4 & 3 & 2 & [aTtRr]1 \\
  & [aLlBb]5 & [aBb]4 & 3 & [aBbRr]2 \\
  &   & [a]6 & 4 & 3 \\
  &   &   & 5 & 4 \\
  &   &   &   & 6
}$
\\ &
\vmm{-2}\\
\hline
&\vmm{-2}\\
Lagrangian Grassmannian $\LG(6,12)$ & Cayley Plane $E_6/P_6$
\\
&\vmm{-3}\\
\pic{1}{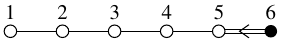} & \\
&\vmm{-2}\\
$\tableau{12}{
[aLlTtBb]6 & [aTtBb]5 & 4 & 3 & 2 & [aTtBbRr]1 \\
  & [a]6 & 5 & 4 & 3 & 2 \\
  &   & 6 & 5 & 4 & 3 \\
  &   &   & 6 & 5 & 4 \\
  &   &   &   & 6 & 5 \\
  &   &   &   &   & 6
}$
&\vmm{-36}\\
& \pic{1}{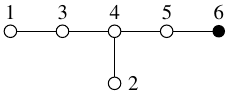} \\
&\vmm{-2}\\
&
$\tableau{12}{
[aLlTtBb]6 & [aTtBb]5 & [aTt]4 & [aTtRr]2 \\
  &   & [aLl]3 & [a]4 & [aTtRr]5 & [a]6 \\
  &   & [aLlBb]1 & [aBb]3 & [aRr]4 & [a]5 \\
  &   &   &   & [aLlBbRr]2 & [a]4 & 3 & 1
}$
\\
& \vmm{-2}\\
\hline
& \vmm{-2}\\
Even quadric $Q^{10} \subset \bP^{11}$ & Freudenthal variety $E_7/P_7$
\\
&\vmm{-3}\\
\pic{1}{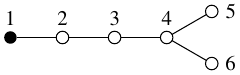} &
\\
&\vmm{-3}\\
$\tableau{12}{
[aLlTtBb]1 & [aTtBb]2 & 3 & [aTt]4 & [aTtRr]5 \\
  &   &   & [aLlBb]6 & [aBb]4 & [aTtBb]3 & [aTtBbRr]2 & [a]1
}$
& \\
&\\
\hhline{-~}
&\vmm{-2}\\
Odd quadric $Q^{11} \subset \bP^{12}$ &
\\
&\vmm{-3}\\
\pic{1}{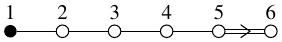} & \\
&\vmm{-1}\\
$\tableau{12}{
[aLlTtBb]1 & [aTtBb]2 & 3 & 4 & 5 & 6 & 5 & 4 & 3 & [aTtBbRr]2 & [a]1
}$
& \vmm{-52}\\
& \pic{1}{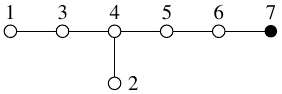} \\
&\vmm{-3}\\
&
$\tableau{12}{
[aLlTtBb]7 & [aTtBb]6 & 5 & [aTt]4 & 3 & [aTtRr]1 \\
  &   &   & [aLlBb]2 & [a]4 & [aRr]3 \\
  &   &   &   & [aLl]5 & [a]4 & [aTtRr]2 \\
  &   &   &   & [aLlBb]6 & [aBb]5 & 4 & [aTtBb]3 & [aTtBbRr]1 \\
  &   &   &   & [a]7 & 6 & 5 & 4 & 3 \\
  &   &   &   &   &   &   & 2 & 4 \\
  &   &   &   &   &   &   &   & 5 \\
  &   &   &   &   &   &   &   & 6 \\
  &   &   &   &   &   &   &   & 7
}$
\vmm{-2}\\
& \\
\hline
\end{tabular}
\end{table}

%% file: qclassical.tex

\section{The quantum to classical principle}
\label{sec:qclassical}

\subsection{Introduction}\label{sec:qcintro}%

The quantum to classical principle allows Gromov-Witten invariants of certain
flag varieties to be computed as classical intersection numbers on related flag
varieties. The goal in this section is to derive the quantum-to-classical
theorem with as little type-by-type checking as possible. In addition we will
develop the associated combinatorics and geometry in order to support the main
results of this paper. We will restrict our discussion to Gromov-Witten
invariants of cominuscule flag varieties of small degrees. Here a degree $d$ is
considered \emph{small} if $q^d$ occurs in a product of two Schubert classes in
the small quantum cohomology ring $\QH(X)$. Equivalently, $d$ is less than or
equal to the \emph{diameter} $d_X(2)$ defined in \Section{nbhd}. Our main
references include \cite{buch:quantum, buch.kresch.ea:gromov-witten,
chaput.manivel.ea:quantum*1, buch.kresch.ea:quantum, buch.mihalcea:quantum,
chaput.perrin:rationality}.

\targetsec{kerspan}{}%
The first version of the quantum-to-classical theorem applied to the enumerative
(cohomological) Gromov-Witten invariants of classical Grassmannians. Let $X =
\Gr(m,n)$ be the Grassmannians of $m$-dimensional vector subspaces of $\C^n$. A
rational curve $C \subset X$ has a \emph{kernel} and a \emph{span} defined by
\cite{buch:quantum}
\[
  \Ker(C) = \bigcap_{V \in C} V \hspace{1cm};\hspace{1cm}
  \Span(C) = \sum_{V \in C} V \,.
\]
If $C$ is a general curve of small degree $d$, then one can show that $\dim
\Ker(C) = m-d$ and $\dim \Span(C) = m+d$, which means that $(\Ker(C),\Span(C))$
is a point in the two-step flag variety $Y_d = \Fl(m-d,m+d;n)$. Given three
Schubert varieties $\Omega_1, \Omega_2, \Omega_3 \subset X$ in general position,
the Gromov-Witten invariant $\gw{[\Omega_1],[\Omega_2],[\Omega_3]}{d}$ is the
number of rational curves of degree $d$ meeting these Schubert varieties
(assuming that this number is finite). The quantum-to-classical theorem states
that the map $C \mapsto (\Ker(C),\Span(C))$ gives a bijection between the
counted curves and the intersection of three Schubert varieties in $Y_d$. As a
consequence, the Gromov-Witten invariant
$\gw{[\Omega_1],[\Omega_2],[\Omega_3]}{d}$ is equal to a classical Schubert
structure constant of $H^*(Y_d)$ \cite{buch.kresch.ea:gromov-witten}.

Subsequent work \cite{chaput.manivel.ea:quantum*1} demonstrated that the
quantum-to-classical theorem can be understood in a type-independent way if the
points $(K,S)$ of $Y_d$ are replaced with the corresponding subvarieties of $X$
defined by $\Gr(d,S/K) = \{ V \in X \mid K \subset V \subset S \}$. Indeed,
these subvarieties of $X$ are non-singular Schubert varieties, and also
cominuscule flag varieties themselves. They can also be described as unions of
rational curves of degree $d$ that pass through two given points in $X$. Such
Schubert varieties will be called \emph{primitive} cominuscule varieties in this
paper, see \Section{primitive}.

The quantum-to-classical theorem was extended to non-enumerative (equivariant
and $K$-theoretic) Gromov-Witten invariants in \cite{buch.mihalcea:quantum} by
showing that the moduli space $\Mb_{0,3}(X,d)$ of stable maps to $X$ is
birational to the space $\{(K,S,V_1,V_2,V_3)\}$ of kernel-span pairs $(K,S) \in
Y_d$ together with 3 additional points $V_i \in \Gr(d,S/K)$. Indeed, given a
general 5-tuple of this type, there exists a unique rational curve $C \subset
\Gr(d,S/K) \subset X$ of degree $d$ which contains the three points $V_1, V_2,
V_3$.

While we only discuss Gromov-Witten invariants of small degrees $d \leq d_X(2)$,
the definition of the quantum $K$-theory ring $\QK(X)$ also depends on
Gromov-Witten invariants of higher degrees. Such Gromov-Witten invariants can be
computed with similar methods, granted that certain Gromov-Witten varieties of
large degrees are rational; this was proved in \cite{buch.mihalcea:quantum} for
Grassmannians of type A and in \cite{chaput.perrin:rationality} for cominuscule
varieties of other Lie types (see also
\cite[Remark~3.4]{buch.chaput.ea:finiteness}).

\subsection{Curve neighborhoods}\label{sec:nbhd}

\targetsec{gwinv}{}%
Let $X = G/P_X$ be a cominuscule flag variety defined over $\C$. For any
non-negative integer $d \geq 0$ we let $M_d = \Mb_{0,3}(X,d)$ denote the
Kontsevich moduli space of 3-pointed stable maps to $X$ of degree $d$ and genus
zero \cite{fulton.pandharipande:notes}. The evaluation map $\ev_i : M_d \to X$,
defined for $1 \leq i \leq 3$, sends a stable map to the image of the $i$-th
marked point in its domain. Given classes $\Omega_1, \Omega_2, \Omega_3 \in
H^*(X;\Z)$, the corresponding cohomological Gromov-Witten invariant of $X$ of
degree $d$ is defined by
\[
  \gw{\Omega_1, \Omega_2, \Omega_3}{d} =
  \int_{M_d} \ev_1^*(\Omega_1) \cdot \ev_2^*(\Omega_2) \cdot
  \ev_3^*(\Omega_3) \,.
\]
More generally, three $K$-theory classes $\cF_1, \cF_2, \cF_3 \in K_T(X)$ define
the $K$-theoretic Gromov-Witten invariant
\[
  I_d(\cF_1, \cF_2, \cF_3) =
  \euler{M_d}(\ev_1^*(\cF_1) \cdot \ev_2^*(\cF_2) \cdot \ev_3^*(\cF_3)) \,.
\]

\targetsec{cnbhd}{}%
Given subvarieties $\Omega_1, \Omega_2 \subset X$, let $M_d(\Omega_1,\Omega_2) =
\ev_1^{-1}(\Omega_1) \cap \ev_2^{-1}(\Omega_2)$ denote the \emph{Gromov-Witten
variety} of stable maps that send the first two marked points to the given
subvarieties. Let $\Gamma_d(\Omega_1, \Omega_2) = \ev_3(M_d(\Omega_1,
\Omega_2))$ be the union of all stable curves of degree $d$ in $X$ that connect
$\Omega_1$ and $\Omega_2$. We also consider the special cases $M_d(\Omega_1) =
\ev_1^{-1}(\Omega_1)$ and $\Gamma_d(\Omega_1) = \ev_3(M_d(\Omega_1))$.

\targetsec{dist}{}%
Define the \emph{degree distance} $\dist(x,y)$ between two points $x,y \in X$ to
be the smallest degree of a rational curve $C \subset X$ with $x,y \in C$. This
is the minimal degree $d$ for which $\Gamma_d(x,y) \neq \emptyset$. We need the
following key result from \cite{chaput.manivel.ea:quantum*1}. A type-independent
proof based on properties satisfied by all flag manifolds was given in \cite[\S
5.4]{buch.mihalcea:curve}.

\begin{thm}\label{thm:dist}
  Let $u \in W^X$. Then $\dist(1.P_X, u.P_X)$ is the number of occurrences of
  $s_\ga$ in any reduced expression for $u$.
\end{thm}

\targetsec{diameter}{}%
Equivalently, $\dist(1.P_X, u.P_X)$ is the number of boxes $\al \in I(u)$ with
label $\delta(\al) = \ga$. Define the \emph{diameter} of $X$ to be the integer
$d_X(2) = \dist(1.P_X, w_0.P_X)$. The subset of boxes in $\cP_X$ labeled by
$\ga$ is totally ordered by \Remark{skew}(a). We denote these boxes by
\[
\delta^{-1}(\ga) = \{ \tal_1 < \tal_2 < \dots < \tal_{d_X(2)} \} \,.
\]

\begin{defn}\label{defn:zdkappad}%
  For $0 \leq d \leq d_X(2)$ we define elements $\ka_d$ and $z_d$ in $W^X$ by
  \[
    I(\ka_d) = \{ \al \in \cP_X \mid \al \leq \tal_d \}
    \text{ \ \ and \ \ }
    I(z_d) = \{ \al \in \cP_X \mid \al \not\geq \tal_{d+1} \} \,.
  \]
  We set $\ka_0 = z_0 = 1$ and $z_{d_X(2)} = w_0^X$.
\end{defn}

\begin{example}
  Let $X = \Gr(4,9)$ be the Grassmannian of $4$-planes in $\C^9$. Then the
  elements $\ka_2$ and $z_2$ are given by the following shapes:
  \[
    I(\ka_2) = \
    \tableau{12}{
      [aLT]&[aTR]&[t]&{}&[tr]\\
      [aLB]&[aBR]&&&[r]\\
      [l]&&&&[r]\\
      [lb]&[b]&{}&{}&[br]
    }
    \hmm{10} \text{and} \hmm{10}
    I(z_2) = \
    \tableau{12}{
      [aLT]&[aT]&{}&{}&[aTR]\\
      [aL]&[a]&[aB]&{}&[aBR]\\
      [aL]&[aR]&&&[r]\\
      [aLB]&[aBR]&[b]&{}&[br]
    } \,.
  \]
  The shapes $I(z_1)$ for a representative selection of cominuscule flag
  varieties are displayed in \Table{tablez1}.
\end{example}

\begin{lemma}\label{lemma:ka_d_inv}%
  We have $\ka_d^{-1} = \ka_d$, $(z_dw_{0,X})^{-1} = z_dw_{0,X}$, and $z_d^\vee
  z_d = w_0^X$. In addition, $\Gamma_d(1.P_X) = X_{z_d}$.
\end{lemma}
\begin{proof}
  It follows from \Theorem{dist} that $\ka_d$ and $z_d w_{0,X}$ are the unique
  minimal and maximal elements of the set  $\{ w \in W \mid \dist(1.P_X, w.P_X)
  = d \}$. Since $\dist(1.P_X, w.P_X) = \dist(1.P_X, w^{-1}.P_X)$ for any
  element $w \in W$, we deduce that $\ka_d$ and $z_d w_{0,X}$ are self-inverse.
  We then obtain
  \[
    z_d^\vee z_d = w_0 z_d w_{0,X} z_d = w_0 (z_d w_{0,X})^{-1} z_d
    = w_0 w_{0,X} = w_0^X \,.
  \]
  The identity $\Gamma_d(1.P_X) = X_{z_d}$ follows from \Theorem{dist} since
  $\Gamma_d(1.P_X)$ is a Schubert variety by
  \cite[Cor.~3.3(a)]{buch.chaput.ea:finiteness}.
\end{proof}

\begin{remark}
  We have $W_X \ka_d W_X = \{ u \in W \mid \dist(1.P_X, u.P_X) = d \}$.
\end{remark}

\begin{lemma}\label{lemma:orbits}%
  The orbits of the diagonal action of $G$ on $X \times X$ are given by
  $\ocZ_{d,2} = \{ (x_1,x_2) \in X \times X \mid \dist(x_1,x_2) = d \}$, for $0
  \leq d \leq d_X(2)$.
\end{lemma}
\begin{proof}
  Each set $\ocZ_{d,2}$ is stable under the action of $G$. Given $(x_1,x_2) \in
  \ocZ_{d,2}$, we can choose $g \in G$ such that $g.x_1 = 1.P_X$, and then
  choose $b \in B$ such that $bg.x_2 = u.P_X$ is a $T$-fixed point, with $u \in
  W^X$. Since $\dist(1.P_X, u.P_X) = d$, \Theorem{dist} implies that $u/\kappa_d
  \in W_X$. The lemma follows from this because $(u/\kappa_d)^{-1} bg.(x_1,x_2)
  = (1.P_X, \kappa_d.P_X)$.
\end{proof}

\begin{lemma}\label{lemma:gammaperp}%
  We have $(\al,\ga^\vee) = 1$ for $\al \in I(z_1)\ssm\{\ga\}$, and
  $(\al,\ga^\vee)=0$ for $\al \in \cP_X\ssm I(z_1)$.
\end{lemma}
\begin{proof}
  Notice that $(\al,\ga^\vee) \geq 0$ for every $\al \in \cP_X$, since otherwise
  the coefficient of $\ga$ in the root $s_\ga.\al = \al - (\al,\ga^\vee)\ga$ is
  at least 2. In addition, if $\al',\al \in \cP_X$ satisfy $\al' \leq \al$, then
  $(\al',\ga^\vee) \geq (\al,\ga^\vee)$, as $\al-\al'$ is a non-negative linear
  combination of $\Delta\ssm\{\ga\}$. Finally, since $\ga$ is a long root, we
  have $(\al,\ga^\vee) \leq 1$ for any root $\al \neq \ga$. It is therefore
  enough to show that $(\al,\ga^\vee)\neq 0$ for $\al \in I(z_1)$ and that
  $(\tal_2,\ga^\vee)=0$. If $\al \in I(z_1) \cup \{\tal_2\}$ is any root with
  $\al \neq \ga$, then we have $\delta(\al) = y s_\ga.\al = y.(\al -
  (\al,\ga^\vee)\ga) \in \Delta$ for some $y \in W_X$. Since the action of $y$
  does not change the coefficient of $\ga$, we deduce that $(\al,\ga^\vee)=0$ if
  and only if $\delta(\al)=\ga$, as required.
\end{proof}

\begin{cor}\label{cor:degq}%
  We have $\int_{X_{s_\ga}} c_1(T_X) = \ell(z_1) + 1$.
\end{cor}
\begin{proof}
  By \cite[Lemma~3.5]{fulton.woodward:quantum} we have $\int_{X_{s_\ga}}
  c_1(T_X) = \sum_{\al \in \cP_X} (\al, \ga^\vee)$, so the corollary follows
  from \Lemma{gammaperp}.
\end{proof}

\begin{prop}\label{prop:firstbiject}%
  {\rm(a)} The map $z_d : \cP_X \ssm I(z_d) \to I(z_d^\vee)$ defined by $\al
  \mapsto z_d.\al$ is an order isomorphism, and $\delta(z_d.\al) = \delta(\al)$
  for each $\al \in \cP_X \ssm I(z_d)$.\smallskip

  \noin{\rm(b)} The map $-\ka_d : I(\ka_d) \to I(\ka_d)$ defined by $\al \mapsto
  -\ka_d.\al$ is an order-reversing involution, and $\delta(-\ka_d.\al) =
  \delta(\al)$ for each $\al \in I(\ka_d)$.
\end{prop}
\begin{proof}
  Since $z_d^\vee = w_0^X/z_d$ by \Lemma{ka_d_inv}, it follows from
  \Lemma{skew}(a) that $z_d: \cP_X\ssm I(z_d) \to I(z_d^\vee)$ is an
  order-preserving bijection. Since $z_d^{-1} = w_{0,X} z_d w_{0,X}$, the
  inverse bijection is also order-preserving. Let $\al \in \cP_X \ssm I(z_d)$
  and set $\la = \la(\al) \cup I(z_d)$. Since $\al$ is a minimal box of $\cP_X
  \ssm \la$, we have $\delta(\al) = w_\la.\al$. Write $\la = I(z_d) \coprod
  \{\al_1, \al_2, \dots,\al_\ell \}$, where the roots are listed in increasing
  order. Then $w_\la = z_d s_{\al_1} s_{\al_2} \cdots s_{\al_\ell}$. Using that
  $z_d$ is an order isomorphism, we obtain $\la(z_d.\al) = \{ z_d.\al_1,
  z_d.\al_2, \dots, z_d.\al_\ell\}$, with the roots listed in increasing order,
  hence
  \[
  w_{\la(z_d.\al)} z_d = s_{z_d.\al_1}
  s_{z_d.\al_2} \cdots s_{z_d.\al_\ell} z_d = z_d s_{\al_1} s_{\al_2}\cdots
  s_{\al_\ell} z_d^{-1} z_d = w_\la \,,
  \]
  and $\delta(z_d.\al) = w_{\la(z_d.\al)}.(z_d.\al) = w_\la.\al = \delta(\al)$.
  This proves part (a).

  For $\al \in I(\ka_d)$ we have $-\ka_d.\al \in \Phi^+$ and $\ka_d.(-\ka_d.\al)
  = -\al < 0$, hence $-\ka_d.\al \in I(\ka_d)$. The map $-\ka_d : I(\ka_d) \to
  I(\ka_d)$ is order-reversing because $\ka_d.\be > 0$ for $\be \in \Delta \ssm
  \{\ga\}$. Given any element $u \leq \ka_d$ we have $I(u \ka_d) =
  -\ka_d.(I(\ka_d) \ssm I(u))$; in fact, the containment $\supseteq$ follows
  from the definition of inversion sets, and both sides have the same
  cardinality because $(u \ka_d)^{-1} = \ka_d/u$. Now choose $u \leq \ka_d$ such
  that $\al$ is a minimal box of $I(\ka_d) \ssm I(u)$. Then $-\ka_d.\al$ is a
  maximal box of $I(u \ka_d)$, so it follows from \Lemma{skew}(b) that
  $\delta(-\ka_d.\al) = - u \ka_d.(-\ka_d.\al) = u.\al = \delta(\al)$. This
  proves part (b).
\end{proof}

\begin{prop}\label{prop:z1k1}%
  The element $z_1 s_\ga = z_1/\ka_1$ permutes $\cP_X$ and satisfies\newline
  $w_{0,X} (z_1 s_\ga) w_{0,X} = (z_1 s_\ga)^{-1}$. The action of $z_1 s_\ga$ on
  $\cP_X$ has the following properties.\smallskip

  \noin{\rm(a)} $z_1 s_\ga : \cP_X \ssm I(z_1) \to I(z_1^\vee)$ is an order
    isomorphism, and $\delta(z_1 s_\ga.\al) = \delta(\al)$ for all $\al \in
    \cP_X \ssm I(z_1)$.\smallskip

  \noin{\rm(b)} $z_1 s_\ga : I(z_1) \ssm \{\ga\} \to w_{0,X}.(I(z_1) \ssm
  \{\ga\})$ is an order isomorphism.\smallskip

  \noin{\rm(c)} We have $z_1 s_\ga.\wt\al_d = \wt\al_{d-1}$ for $2 \leq d \leq
  d_X(2)$, and $z_1 s_\ga.\ga = w_{0,X}.\ga$ is the highest root of $\Phi^+$.
\end{prop}
\begin{proof}
  Since $z_1 s_\ga \in W_X$ we have $z_1 s_\ga.\cP_X = \cP_X$, and
  \Lemma{gammaperp} implies
  \[
    z_1 s_\ga.\al = \begin{cases}
      z_1.\al & \text{if $\al \in \cP_X \ssm I(z_1)$;} \\
      z_1.\al - z_1.\ga & \text{if $\al \in I(z_1) \ssm \{\ga\}$;} \\
      -z_1.\ga & \text{if $\al=\ga$}
    \end{cases}
  \]
  for any $\al \in \cP_X$. Part (a) therefore follows from
  \Proposition{firstbiject}(a). In particular, we have $z_1 s_\ga.\wt\al_d =
  \wt\al_{d-1}$ for $d \geq 2$. Since $z_1 s_\ga.\al < z_1 s_\ga.\ga$ for every
  $\al \in I(z_1) \ssm \{\ga\}$, we deduce that $z_1 s_\ga.\ga = w_{0,X}.\ga$ is
  the maximal box of $\cP_X$, which proves part (c). Using \Lemma{ka_d_inv} we
  also obtain $z_1 s_\ga w_{0,X} = z_1 w_{0,X} s_{w_{0,X}.\ga} = w_{0,X}
  z_1^{-1} s_{z_1.\ga} = w_{0,X} (z_1 s_\ga)^{-1}$. Finally, $z_1 s_\ga$ is
  order preserving on $I(z_1) \ssm \{\ga\}$ because $z_1.\be > 0$ for each $\be
  \in \Delta \ssm \{\ga\}$, and the identity $(z_1 s_\ga)^{-1} = w_{0,X} (z_1
  s_\ga) w_{0,X}$ shows that the inverse map is also order preserving. This
  proves part (b).
\end{proof}

\begin{cor}\label{cor:zd=z1k1d}%
  We have $(z_1 s_\ga)^d.\al = z_d.\al$ for each $\al \in \cP_X \ssm I(z_d)$.
\end{cor}
\begin{proof}
  Noting that $\cP_X \ssm I(z_d) = \{ \al \in \cP_X \mid \wt\al_{d+1} \leq \al
  \leq w_{0,X}.\wt\al_1 \}$ and $I(z_d^\vee) = \{ \al \in \cP_X \mid \wt\al_1
  \leq \al \leq w_{0,X}.\wt\al_{d+1} \}$, it follows from \Proposition{z1k1}
  that $(z_1 s_\ga)^d : \cP_X \ssm I(z_d) \to I(z_d^\vee)$ is an order
  isomorphism that preserves the labeling $\delta$. The result therefore follows
  from \Proposition{firstbiject}(a) and \Remark{skew}(a).
\end{proof}

\begin{remark}
  We will show in \Corollary{zdkd} that $z_d/\ka_d = (z_1/\ka_1)^d = (z_1
  s_\ga)^d$. Together with \Corollary{zd=z1k1d}, this implies that $\ka_d.\al =
  \al$ for all $\al \in \cP_X \ssm I(z_d)$. \Proposition{biject} implies that
  $z_d/\ka_d : I(z_d)\ssm I(\ka_d) \to I(\ka_d^\vee)\ssm I(z_d^\vee)$ is an
  order isomorphism, which generalizes \Proposition{z1k1}(b). Using
  \Proposition{firstbiject}(b), it follows that $z_d/\ka_d : I(\ka_d) \to
  w_{0,X}.I(\ka_d)$ is an order isomorphism. These remarks will not be used in
  the following.
\end{remark}

\targetsec{Sd}{}%
For $1 \leq d \leq d_X(2)$ we define $\cS_d = (I(z_d)\ssm I(z_{d-1})) \cup
(I(\ka_d)\ssm I(\ka_{d-1}))$.

\begin{prop}\label{prop:Sd}%
  We have $\cS_d = \{ \al \in \cP_X \mid (\al,\wt\al_d) > 0 \}$ for $1 \leq d
  \leq d_X(2)$, and $z_1 s_\ga.\cS_d = \cS_{d-1}$ for $2 \leq d \leq d_X(2)$.
\end{prop}
\begin{proof}
  The second identity follows from the first identity together with
  \Proposition{z1k1}(c). We must therefore show that $(\al,\wt\al_d) > 0$ if and
  only if $\al \in \cS_d$, for any $\al \in \cP_X$. If $\al$ and $\wt\al_d$ are
  not comparable in the partial order on $\cP_X$, then $(\al,\wt\al_d) = 0$ and
  $\al \notin \cS_d$. If $\al \geq \wt\al_d$, then \Proposition{firstbiject}(a)
  shows that $z_{d-1}.\wt\al_d = \ga$, and also that $\al \in \cS_d$ if and only
  if $z_{d-1}.\al \in I(z_1)$, so the claim follows from \Lemma{gammaperp},
  noting that $(\al,\wt\al_d) = (z_{d-1}.\al,\ga)$. Finally, if $\al \leq
  \wt\al_d$, then \Proposition{firstbiject}(b) shows that $-\ka_d.\wt\al_d =
  \ga$, and also that $\al \in \cS_d$ if and only if $-\ka_d.\al \in I(z_1)$, so
  the claim again follows from \Lemma{gammaperp}, this time noting that
  $(\al,\wt\al_d) = (-\ka_d.\al,\ga)$.
\end{proof}

\begin{cor}\label{cor:dimMd}%
  Let $0 \leq d \leq d_X(2)$.\smallskip

  \noin{\rm(a)} We have $\dim M_d = \dim(X) + \ell(z_d) +
  \ell(\ka_d)$.\smallskip

  \noin{\rm(b)} The variety $M_d(1.P_X, \kappa_d.P_X)$ is irreducible of
  dimension $\ell(\ka_d)$.
\end{cor}
\begin{proof}
  Since $\ell(z_d)+\ell(\ka_d) - \ell(z_{d-1})-\ell(\ka_{d-1}) = \# \cS_d + 1$
  for $d \geq 1$, it follows from \Proposition{Sd} and \Corollary{degq} that
  $\ell(z_d)+\ell(\ka_d) = d \int_{X_\ga} c_1(T_X)$. This proves part (a). Since
  $\ev_1 : M_d \to X$ is a locally trivial fibration, it follows that
  $M_d(1.P_X)$ is irreducible of dimension $\ell(z_d) + \ell(\ka_d)$. Part (b)
  follows from this, using that $\ev_2 : M_d(1.P_X) \to X_{z_d}$ is a locally
  trivial fibration over a dense open subset of $X_{z_d} = \ov{P_X \ka_d.P_X}$
  that contains $\ka_d.P_X$.
\end{proof}

\subsection{Incidence varieties}\label{sec:incidence}

Given any flag variety $Y = G/P_Y$, the \emph{incidence variety} of $X$ and $Y$
is the flag variety $Z = G/P_Z$ defined by $P_Z = P_X \cap P_Y$. Let $p : Z \to
X$ and $q : Z \to Y$ be the projections and set $F = p^{-1}(1.P_X) = P_X/P_Z$
and $\Gamma = q^{-1}(1.P_Y) = P_Y/P_Z$. For example, if $X = \Gr(m,n)$ and $Y =
\Fl(m-d,m+d;n)$, then $Z = \Fl(m-d,m,m+d;n)$, $\Gamma = \Gr(d,2d)$, and $F =
\Gr(m-d,m) \times \Gr(d, n-m)$. For $\om \in Y$ we write $\Gamma_\om =
p(q^{-1}(\om)) \subset X$. We identify $Z$ with the subvariety $\{(\om,x) \in Y
\times X \mid x \in \Gamma_\om\}$ of $Y \times X$. The restricted maps $p :
\Gamma \to p(\Gamma)$ and $q : F \to q(F)$ are isomorphisms, hence $p(\Gamma)
\subset X$ and $q(F) \subset Y$ are non-singular Schubert varieties. More
precisely we have $F = Z_{w_{0,X}^Z}$ and $q(F) = Y_{w_{0,X}^Z}$, and also
$\Gamma = Z_\ka$ and $p(\Gamma) = X_\ka$, where $\ka = w_{0,Y}^Z$. In our
applications of this construction we have $\ka = \ka_d$ for some degree $d$, see
\Corollary{comin_kapd} (but $\ka$ is not related to \Theorem{schubint}).

If $\ga \notin \Delta_Y$, then $P_Y \subset P_X$ and $\Gamma$ is a point. Assume
that $\ga \in \Delta_Y$. Since $\Delta_Y \ssm \Delta_Z = \{\ga\}$ consists of a
cominuscule simple root, it follows that $\Gamma$ is a cominuscule flag variety.
The corresponding partially ordered set is given by  $\cP_\Gamma = I(\ka) =
\Phi_Y^+ \ssm \Phi_Z = \cP_X \cap \Phi_Y$. The labeling $\cP_\Gamma \to
\Delta_Y$ is the restriction of the labeling $\delta : \cP_X \to \Delta$, and a
curve $C \subset \Gamma$ has the same degree in $\Gamma$ as in $X$. The variety
$\Gamma = P_Y/(P_Y \cap P_X)$ depends only on the connected component of (the
Dynkin diagram of) $\Delta_Y$ that contains $\ga$.

\targetsec{dynkinint}{}%
If $S \subset \Delta$ is any subset that is connected in the Dynkin diagram of
$\Phi$, then the sum of all simple roots in $S$ is a root in $\Phi$. Given $\be
\in \Delta$, let $[\ga,\be]$ denote the smallest connected subset of $\Delta$
that contains $\ga$ and $\be$. A simple root $\be \in \Delta \ssm \Delta_Y$ is
an \emph{essential excluded root} of $Y$ if $[\ga,\be] \subset \Delta_Y \cup
\{\be\}$, i.e.\ $\be$ is connected to the component of $\ga$ in $\Delta_Y$. The
group $P_Y$ is contained in the stabilizer of $X_\ka$ in $G$, and is equal to
this stabilizer if and only if all roots in $\Delta \ssm \Delta_Y$ are essential
excluded roots of $Y$.

\begin{remark}
  Assume that $P_Y$ is the stabilizer of $X_\ka$ in $G$. Then $\Delta_Y = \{\be
  \in \Delta \mid (\ka.\om_\ga, \be^\vee) \leq 0 \}$. In fact, if $\be \in
  \Delta$ and $(\ka.\om_\ga, \be^\vee) > 0$, then $\al = \ka^{-1}.\be$ must be a
  minimal box of $\cP_X \ssm I(\ka)$ by \Lemma{skew}(b), which implies that $\be
  \in \Delta \ssm \Delta_Y$. On the other hand, if $\be \in \Delta \ssm
  \Delta_Y$, then let $\al$ be the sum of the simple roots in the interval
  $[\ga,\be]$. Then $\al$ is a minimal box of $\cP_X \ssm I(\ka)$, $(\om_\ga,
  \al^\vee) > 0$, and \Lemma{skew}(b) implies that $\be = \ka.\al$.
\end{remark}

\subsection{Primitive cominuscule varieties}\label{sec:primitive}

The cominuscule flag variety $X$ will be called \emph{primitive} if the excluded
cominuscule simple root $\ga$ is invariant under the Cartan involution, that is,
$\ga = - w_0.\ga$. The list of all primitive cominuscule varieties is contained
in \Table{primitive}.

\begin{prop}\label{prop:primitive}%
  Let $X$ be a cominuscule flag variety of diameter $d = d_X(2)$, and let $\rho
  \in \Phi^+$ be the highest root. The following conditions are equivalent and
  hold if and only if $X$ is primitive. {\rm(1)} $\ga = -w_0.\ga$. {\rm(2)}
  $\delta(\rho) = \ga$. {\rm(3)} $\ka_d = z_d$. {\rm(4)} $(w_0^X)^{-1} = w_0^X$.
  {\rm(5)} $\dim(M_d) = 3 \dim(X)$.
\end{prop}
\begin{proof}
  Condition (1) is our definition of primitive. \Lemma{PXinvol} shows that
  $\delta(\rho) = \delta(w_{0,X}.\ga) = -w_0.\delta(\ga) = -w_0.\ga$, so (1) is
  equivalent to (2). The implication (2) $\Rightarrow$ (3) is clear from the
  definitions, (3) $\Rightarrow$ (4) holds because $\ka_d^{-1} = \ka_d$ and $z_d
  = w_0^X$, and (4) $\Rightarrow$ (2) holds because (4) implies that
  $(w_0^X)^{-1} \in W^X$. Finally, (5) is equivalent to (3) by
  \Corollary{dimMd}(a).
\end{proof}

\begin{table}
  \caption{Primitive cominuscule varieties.}
  \label{tab:primitive}
  \begin{tabular}{|c|c|}
    \hline
    $X$ & $d_X(2)$ \\
    \hline
    & {}\vspace{-.9em}\\
    $\Gr(d,2d)$ & $d$ \\
    $\LG(d,2d)$ & $d$ \\
    $\OG(2d,4d)$ & $d$ \\
    $Q^N$ & 2 \\
    $E_7/P_7$ & 3 \\
    \hline
  \end{tabular}
\end{table}

The following identity is a consequence of the structure theorems for quantum
cohomology proved in \cite{bertram:quantum, buch.kresch.ea:gromov-witten,
kresch.tamvakis:quantum, kresch.tamvakis:quantum*1,
chaput.manivel.ea:quantum*1}. It can also be checked by constructing the unique
rational curve through three general points in each case. We sketch how this is
done in the proof.

\begin{thm}\label{thm:prim3pt}%
  Let $X$ be a primitive cominuscule variety of diameter $d = d_X(2)$. Then
  $\gw{\pt,\pt,\pt}{d} = 1$.
\end{thm}
\begin{proof}
  Assume first that $V, V', V'' \subset \C^{2d}$ are three general points in the
  primitive Grassmannian $\Gr(d,2d)$ of type A. Choose any basis $\{v_1, \dots,
  v_d\}$ of $V$, and write $v_i = v'_i + v''_i$ for each $i$, with $v'_i \in V'$
  and $v''_i \in V''$. Then the only rational curve of degree $d$ through $V$,
  $V'$, $V''$ is $C = \{ \langle s v'_1 + t v''_1, \dots, s v'_d + t v''_d
  \rangle \mid (s:t) \in \bP^1 \}$, see
  \cite[Prop.~1]{buch.kresch.ea:gromov-witten} for details. The same
  construction works for the Lagrangian Grassmannian $\LG(d,2d)$ and the maximal
  orthogonal Grassmannian $\OG(2d,4d)$, see \cite[Prop.~2 and
  Prop.~4]{buch.kresch.ea:gromov-witten}. If $V, V', V'' \subset \C^{N+2}$ are
  general points in the quadric $Q^N = \OG(1,N+2)$, then $E = V \oplus V' \oplus
  V''$ is an orthogonal vector space of dimension 3, and the unique curve of
  degree 2 through $V$, $V'$, $V''$ is $C = \bP(E) \cap Q^N$. A similar
  construction of the unique cubic curve through three general points of the
  Freudenthal variety $E_7/P_7$ can be found in
  \cite[Lemma~5]{chaput.manivel.ea:quantum*1}.
\end{proof}

\begin{lemma}\label{lemma:prim_dual}%
  Let $X$ be a primitive cominuscule variety. For $0 \leq d \leq d_X(2)$ we have
  $\ka_d^\vee = z_{d_X(2)-d}$.
\end{lemma}
\begin{proof}
  This follows from \Lemma{PXinvol}, noting that $w_{0,X}.\tal_d =
  \tal_{d_X(2)-d+1}$.
\end{proof}

In the following we will consider a single point to be a primitive cominuscule
flag variety of diameter zero.

\begin{prop}\label{prop:defYd}%
  Let $X$ be any cominuscule variety and $0 \leq d \leq d_X(2)$. There exists a
  unique largest parabolic subgroup $P_{Y_d} \subset G$ containing $B$ such that
  $\Gamma_d = P_{Y_d}/(P_X \cap P_{Y_d})$ is a primitive cominuscule variety of
  diameter $d$. In addition, $F_d = P_X/(P_X \cap P_{Y_d})$ is a product of
  cominuscule varieties.
\end{prop}
\begin{proof}
  This must be checked from the classification of cominuscule flag varieties,
  but only the associated Dynkin diagrams need to be considered. For $d=0$ we
  have $Y_0=X$ and $\Gamma_0=F_0=\{\pt\}$. If $X$ is a primitive cominuscule
  variety and $d=d_X(2)$, then $Y_d=F_d=\{\pt\}$ and $\Gamma_d=X$. The choice of
  $P_{Y_d}$ in all other cases is indicated in \Table{incidence}, where the
  roots of $\Delta \ssm \Delta_{Y_d}$ are colored gray and $\ga$ is colored
  black.
\end{proof}

\input{tableqc}

\subsection{Primitive curve neighborhoods}\label{sec:prim-nbhd}%

\targetsec{qcspaces}{}%
Let $X$ be a cominuscule variety. For $0 \leq d \leq d_X(2)$ we define $Y_d =
G/P_{Y_d}$ by the parabolic subgroup of \Proposition{defYd}, and we let $Z_d =
G/P_{Z_d}$ be the incidence variety defined by $P_{Z_d} = P_{Y_d} \cap P_X$. Let
$p_d : Z_d \to X$ and $q_d : Z_d \to Y_d$ denote the projections, with fibers
$F_d = P_X/P_{Z_d}$ and $\Gamma_d = P_{Y_d}/P_{Z_d}$. For any point $\om \in
Y_d$ we will use the notation $\Gamma_\om = p_d q_d^{-1}(\om) \subset X$.
We identify $Z_d$ with its image by the map $q_d\times p_d$, that is
\[
  Z_d = \{(\om,x) \in Y_d\times X \mid x \in \Gamma_\om\} \,.
\]
We will also frequently identify $\Gamma_d$ with $p_d(\Gamma_d) =
\Gamma_{1.P_Y}$.\footnote{If the symbol $\Gamma$ has a subscript that is not a
degree, then the subscript is a point $\omega$ in $Y_d$ for some degree $d$, in
which case $\Gamma_\omega$ is a translate of $\Gamma_d = \Gamma_{1.P_{Y_d}}$.}
Since $P_{Y_d}$ is the stabilizer of $\Gamma_{1.Y_d}$ by the maximality
condition of \Proposition{defYd}, the following result shows that the assignment
$\om \mapsto \Gamma_\om$ is a bijection from $Y_d$ to the set of all translates
of $X_{\ka_d}$ in $X$. Recall that $\Gamma_d(x,y)$ is the union of all stable
curves of degree $d$ through $x$ and $y$, and $\ocZ_{d,2} = \{ (x,y) \in X
\times X \mid \dist(x,y) = d \}$.

\begin{cor}\label{cor:comin_kapd}%
  {\rm(a)} We have $X_{\ka_d} = \Gamma_d = \Gamma_d(1.P_X,
  \ka_d.P_X)$.\smallskip

  \noin{\rm(b)} Let $x, y \in X$. We have $\dist(x,y) \leq d$ if and only if
  $x,y \in \Gamma_\om$ for some $\om \in Y_d$. The element $\om$ is unique if
  $\dist(x,y) = d$.\smallskip

  \noin{\rm(c)} The function $\varphi : \ocZ_{d,2} \to Y_d$ defined by
  $\Gamma_{\varphi(x,y)} = \Gamma_d(x,y)$ is a morphism of varieties.
\end{cor}
\begin{proof}
  Since $\Gamma_d$ is a primitive cominuscule variety of diameter $d$, it
  follows that $\wt\al_d$ is the largest root in $\cP_{\Gamma_d}$, hence
  $\cP_{\Gamma_d} = I(\ka_d)$ and $\Gamma_d = X_{\ka_d}$. \Corollary{dimMd}(b)
  implies that $\Gamma_d(1.P_X, \ka_d.P_X) = \ev_3(M_d(1.P_X, \ka_d.P_X))$ is an
  irreducible subvariety of $X$ of dimension at most $\ell(\ka_d)$, and
  \Theorem{prim3pt} shows that $X_{\ka_d} \subset \Gamma_d(1.P_X, \ka_d.P_X)$.
  This proves part (a). For part (b) we may assume that $x = 1.P_X$ and $y =
  \ka_{d'}.P_X$ by \Lemma{orbits}, where $d' = \dist(x,y)$. If $\dist(x,y) \leq
  d$, then $x,y \in \Gamma_{1.P_{Y_d}}$ by part (a). On the other hand, since
  the diameter of $\Gamma_\om$ is $d$ for each $\om \in Y_d$, we have
  $\dist(x,y) \leq d$ whenever $x,y \in \Gamma_\om$. Finally, if $x,y \in
  \Gamma_\om$ and $\dist(x,y) = d$, then \Theorem{prim3pt} applied to
  $\Gamma_\om$ shows that $\Gamma_\om \subset \Gamma_d(x,y) =
  \Gamma_{1.P_{Y_d}}$, hence $\om = 1.P_{Y_d}$. Part (b) follows from this.
  Choose morphisms $s_1 : \oX^1 \to B^-$ and $s_2 : \oX_{z_d} \to B$ so that $x
  = s_1(x).P_X$ for all $x \in \oX^1$, and $y = s_2(y)z_d.P_X$ for all $y \in
  \oX_{z_d}$. For all points $(x,y)$ in a dense open subset of $\ocZ_{d,2}$ we
  have $(x,y) = s_1(x) s_2(s_1(x)^{-1}.y) z_d \ka_d.(1.P_X, \ka_d.P_X)$, so
  $\varphi$ is defined by $\varphi(x,y) = s_1(x) s_2(s_1(x)^{-1}.y)
  z_d\ka_d.P_{Y_d}$ on this subset. This proves part (c).
\end{proof}

\subsection{A blow-up of the Kontsevich moduli space}\label{sec:qcblowup}%

\targetsec{blowup}{}%
Let $0 \leq d \leq d_X(2)$. The map $q_d : Z_d \to Y_d$ is a locally trivial
fibration with fibers $\Gamma_\om \cong X_{\kappa_d}$, $\om \in Y_d$. Define
a new family $\Bl_d \to Y_d$ by replacing each fiber $\Gamma_\om$ with the
moduli space $\Mb_{0,3}(\Gamma_\om, d)$. Since $\Gamma_\om$ is a
subvariety of $X$, we have $\Mb_{0,3}(\Gamma_\om,d) \subset M_d$, and these
inclusions define a morphism $\pi : \Bl_d \to M_d$. Equivalently, we have $\Bl_d
= G \times^{P_{Y_d}} \Mb_{0,3}(\Gamma_d, d)$. We will identify $\Bl_d$ with its
image in $Y_d \times M_d$, that is
\[
\Bl_d = \{ (\om,f) \in Y_d \times M_d \mid \operatorname{Image}(f)
\subset \Gamma_\om \} \,.
\]
\targetsec{Zd3}{}%
We also define the space
\[
Z_d^{(3)} \ = \ Z_d \times_{Y_d} Z_d \times_{Y_d} Z_d \ = \ \{
(\om,x_1,x_2,x_3) \in Y_d \times X^3 \mid x_1,x_2,x_3 \in
\Gamma_\om \} \,.
\]
Define a morphism $\phi : \Bl_d \to Z_d^{(3)}$ by $\phi(\om,f) =
(\om,\ev_1(f),\ev_2(f),\ev_3(f))$, and let $e_i : Z_d^{(3)} \to Z_d$ denote
the $i$-th projection. We obtain the following commutative diagram from
\cite{buch.mihalcea:quantum}:
\begin{equation}\label{eqn:qcdiagram}
  \raisebox{10ex}{
  \xymatrix{\Bl_d \ar[rr]^{\pi} \ar[d]^{\phi} & & M_d \ar[d]^{\ev_i} \\
    Z_d^{(3)}\ar[r]^{e_i} & Z_d \ar[r]^{p_d} \ar[d]^{q_d} & X \\ & Y_d}}
\end{equation}

\begin{prop}\label{prop:qc_birat}
  The maps $\pi : \Bl_d \to M_d$ and $\phi : \Bl_d \to Z_d^{(3)}$ are
  birational.
\end{prop}
\begin{proof}
  It follows from \Theorem{prim3pt} that $\phi$ is birational. Since the image
  of the map $\ev_1 \times \ev_2 : M_d \to X \times X$ is contained in
  $\cZ_{d,2} = \{ (x_1,x_2) \mid \dist(x_1,x_2) \leq d \}$, it follows that $U =
  (\ev_1 \times \ev_2)^{-1}(\ocZ_{d,2})$ is a dense open subset of $M_d$. Given
  any stable map $f \in U$, we have $\dist(\ev_1(f),\ev_2(f)) = d$, so
  \Corollary{comin_kapd} implies that the image of $f$ is contained in
  $\Gamma_\om$ for a unique point $\om \in Y_d$. It follows that $\pi^{-1}(f) =
  (\om,f)$, so $\pi$ is birational.
\end{proof}

The (three point, genus zero) Gromov-Witten invariants of small degrees of a
cominuscule flag variety are given by the following result. Generalizations to
larger degrees can be found in \cite{buch.mihalcea:quantum,
chaput.perrin:rationality, buch.chaput.ea:projected}.

\begin{cor}\label{cor:qc}
  Let $X$ be cominuscule and $0 \leq d \leq d_X(2)$.

  \noin{\rm(a)} For $\Omega_1, \Omega_2, \Omega_3 \in H^*_T(X;\Z)$ we have
  \[
    \gw{\Omega_1, \Omega_2, \Omega_3}{d} = \int_{Y_d} {q_d}_* p_d^*(\Omega_1)
    \cdot {q_d}_* p_d^*(\Omega_2) \cdot {q_d}_* p_d^*(\Omega_3) \,.
  \]

  \noin{\rm(b)} For $\cF_1, \cF_2, \cF_3 \in K_T(X)$ we have
  \[
    I_d(\cF_1,\cF_2,\cF_3) = \euler{Y_d} \big( {q_d}_* p_d^*(\cF_1)
    \cdot {q_d}_* p_d^*(\cF_2) \cdot {q_d}_* p_d^*(\cF_3) \big) \,.
  \]
\end{cor}
\begin{proof}
  Since all varieties in the diagram \eqn{qcdiagram} have rational
  singularities, it follows from \Proposition{qc_birat} that $\pi_*[\cO_{\Bl_d}]
  = [\cO_{M_d}]$ and $\phi_*[\cO_{\Bl_d}] = [\cO_{Z_d^{(3)}}]$. We obtain
  \[
  \begin{split}
    \euler{M_d}(\ev_1^*(\cF_1)\cdot \ev_2^*(\cF_2)\cdot \ev_3^*(\cF_3))
    &= \euler{Z_d^{(3)}}(e_1^*p_d^*(\cF_1)\cdot e_2^*p_d^*(\cF_2)\cdot
    e_3^*p_d^*(\cF_3)) \\
    &= \euler{Y_d}({q_d}_*p_d^*(\cF_1)\cdot
    {q_d}_*p_d^*(\cF_2)\cdot {q_d}_*p_d^*(\cF_3)) \,,
  \end{split}
  \]
  where the first identity follows from the projection formula (twice) together
  with commutativity of the diagram \eqn{qcdiagram}, and the second follows from
  \cite[Lemma~3.5]{buch.mihalcea:quantum}. This proves part (b). Part (a) is
  proved by repeating the same argument with cohomology classes, or by
  extracting the initial terms of both sides in part (b), see
  \cite[\S4.1]{buch.mihalcea:quantum}.
\end{proof}

\targetsec{qcrich}{}%
For any subvarieties $\Omega_1, \Omega_2 \subset X$ and $0 \leq d \leq d_X(2)$
we define
\[
  \begin{split}
    Y_d(\Omega_1,\Omega_2)
    &= q_d(p_d^{-1}(\Omega_1)) \cap q_d(p_d^{-1}(\Omega_2)) \\
    &= \{ \om \in Y_d \mid \Gamma_\om \cap \Omega_1 \neq \emptyset
    \text{ and } \Gamma_\om \cap \Omega_2 \neq \emptyset \} \,, \\
    Z_d(\Omega_1,\Omega_2) &= q_d^{-1}(Y_d(\Omega_1,\Omega_2)) \,.
  \end{split}
\]
We also define the special cases $Y_d(\Omega_1) = q_d p_d^{-1}(\Omega_1)$ and
$Z_d(\Omega_1) = q_d^{-1}(Y_d(\Omega_1))$. Notice that for $\om \in Y_d$ we have
$\Gamma_\om \cap \Omega_1 \neq \emptyset$ if and only if $\om \in
Y_d(\Omega_1)$.

\begin{cor}\label{cor:qc_nbhd}%
  We have $\Gamma_d(\Omega_1, \Omega_2) = p_d(Z_d(\Omega_1, \Omega_2))$. As a
  special case we obtain $\Gamma_d(\Omega_1) = p_d(Z_d(\Omega_1)) = p_d q_d^{-1}
  q_d p_d^{-1}(\Omega_1)$.
\end{cor}
\begin{proof}
  Since the diagram \eqn{qcdiagram} is commutative and the maps $\pi$ and $\phi$
  are surjective, we obtain
  \[
    \begin{split}
      \Gamma_d(\Omega_1, \Omega_2)
      &= \ev_3\pi \left((\ev_1\pi)^{-1}(\Omega_1)
      \cap (\ev_2\pi)^{-1}(\Omega_2)\right) \\
      &= p_d e_3 \left((p_d e_1)^{-1}(\Omega_1)
      \cap (p_d e_2)^{-1}(\Omega_2)\right) \\
      &= p_d\left(\{(\om,x_3) \in Z_d \mid \Gamma_\om
      \cap \Omega_1 \neq \emptyset
      \text{ and } \Gamma_\om \cap \Omega_2 \neq \emptyset \}\right)
    \end{split}
  \]
  as required.
\end{proof}

%% file: tableqc.tex

\begin{table}
  \def\hspem#1{\mbox{}\hspace{#1em}\mbox{}}
  \def\vspex#1{\vspace{#1ex}}
  \def\topbox{\vspex{-1.7}\\}
  \def\botbox{\vspex{-2.1}\\}
  \caption{The quantum-to-classical construction.}
  \label{tab:incidence}
  \def\mydynsz{.8}
  \begin{tabular}{|l|}
    \hline\topbox 
    Grassmannian $X=\Gr(m,n+1)$ of type $A_n$\\
    $d_X(2)=\min(m,n\!+\!1\!-\!m)$
    \vspex{-.5}\\
    \hspem{10}\pic{1}{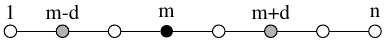}\\
    \vspex{-2.5}\\
    $Y_d=\Fl(m\!-\!d,m\!+\!d;n\!+\!1)$ \ ; \ $\Gamma_d=\Gr(d,2d)$ \ ;\\
    $F_d=\Gr(m\!-\!d,m)\times\Gr(d,n\!+\!1\!-\!m)$\\
    \botbox\hline\topbox 
    Lagrangian Grassmannian $X=\LG(n,2n)$ of type $C_n$\\
    $d_X(2) = n$
    \vspex{-3}\\
    \hspem{12.5}\pic{1}{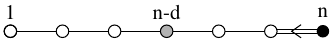}\\
    \vspex{-2.5}\\
    $Y_d=\SG(n\!-\!d,2n)$ \ ; \ $\Gamma_d=\LG(d,2d)$ \ ; \
    $F_d=\Gr(n\!-\!d,n)$\\
    \botbox\hline\topbox 
    Max.\ orthogonal Grassmannian $X=\OG(n,2n)$ of type $D_n$\\
    $d_X(2)=\lfloor n/2\rfloor$
    \vspex{-3}\\
    \hspem{14.4}\pic{1}{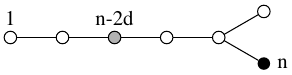}\\
    $Y_d=\OG(n\!-\!2d,2n)$ \ ; \ $\Gamma_d=\OG(2d,4d)$ \ ; \
    $F_d=\Gr(n\!-\!2d,n)$\\
    \botbox\hline\topbox 
    Even quadric $X = Q^{2n-2}$ of type $D_n$\\
    $d_X(2)=2$
    \vspex{-3}\\
    \hspem{14.4}\pic{1}{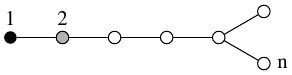}\\
    \vspex{-4.8}\\
    $Y_1=\OG(2,2n)$ \ ; \ $\Gamma_1=\bP^1$ \ ; \ $F_1=Q^{2n-4}$\\
    \botbox\hline\topbox 
    Odd quadric $X = Q^{2n-1}$ of type $B_n$\\
    $d_X(2)=2$
    \vspex{-2.9}\\
    \hspem{15}\pic{1}{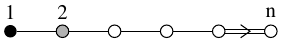}\\
    \vspex{-2.5}\\
    $Y_1=\OG(2,2n\!+\!1)$ \ ; \ $\Gamma_1=\bP^1$ \ ; \ $F_1=Q^{2n-3}$\\
    \botbox\hline\topbox 
    Cayley plane $X=E_6/P_6$\\
    $d_X(2)=2$\\
    \vspex{-.8}\\
    $Y_1=E_6/P_5$ \ ; \ $\Gamma_1=\bP^1$ \ ; \ $F_1=\OG(5,10)$
    \vspex{-8.5}\\
    \hspem{17.5}\pic{1}{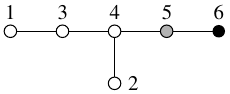}\\
    \\\\
    $Y_2=E_6/P_1$ \ ; \ $\Gamma_2=Q^8$ \ ; \ $F_2=Q^8$
    \vspex{-8.5}\\
    \hspem{17.5}\pic{1}{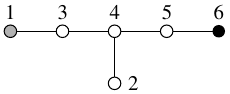}\\
    \hline\topbox 
    Freudenthal variety $X=E_7/P_7$\\
    $d_X(2)=3$\\
    \vspex{-.8}\\
    $Y_1=E_7/P_6$ \ ; \ $\Gamma_1=\bP^1$ \ ; \ $F_1=E_6/P_6$
    \vspex{-8.5}\\
    \hspem{15}\pic{1}{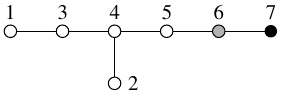}\\
    \\\\
    $Y_2=E_7/P_1$ \ ; \ $\Gamma_2=Q^{10}$ \ ; \ $F_2=E_6/P_1$
    \vspex{-8.5}\\
    \hspem{15}\pic{1}{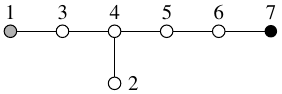}\\
    \hline 
  \end{tabular}
\end{table}

%% file: qcfibers.tex

\section{Fibers of the quantum-to-classical construction}\label{sec:qcfibers}

In this section we obtain explicit descriptions of the general fibers of several
maps related to the quantum-to-classical construction. These results are
required for determining the powers of $q$ that occur in quantum products, as
well as for our proof that the structure constants of quantum $K$-theory have
alternating signs.

\subsection{Bijections between order ideals}\label{sec:qcbiject}%

\begin{lemma}\label{lemma:weylid}%
  For $0 \leq d \leq d_X(2)$ we have the identities $\ka_d = w_{0,Y_d}^{Z_d}$,
  $z_d = w_{0,X} w_{0,Y_d}$, $z_d/\ka_d = w_{0,X}^{Z_d}$, and $w_{0,Y_d} = \ka_d
  w_{0,Z_d} = w_{0,Z_d}\ka_d$.
\end{lemma}
\begin{proof}
  Since the projection $p_d : q_d^{-1}(1.P_{Y_d}) \to X_{\ka_d}$ is an
  isomorphism, it follows that $q_d^{-1}(1.P_{Y_d}) = (Z_d)_{\ka_d}$ and
  $w_{0,Y_d}^{Z_d} = \ka_d$.  We also obtain $p_d^{-1}(1.P_X) =
  (Z_d)_{w_{0,X}^{Z_d}}$ and $Y_d(1.P_X) = (Y_d)_{w_{0,X}^{Z_d}}$, and therefore
  $Z_d(1.P_X) = (Z_d)_{w_{0,X}^{Z_d} w_{0,Y_d}^{Z_d}}$.  Since we have
  $\dist(z,1.P_X)=d$ for all points $z$ in a dense open subset of
  $\Gamma_d(1.P_X)$, it follows from \Corollary{comin_kapd} and
  \Corollary{qc_nbhd} that $p_d : Z_d(1.P_X) \to \Gamma_d(1.P_X) = X_{z_d}$ is
  birational.  We deduce that $w_{0,X}^{Z_d} w_{0,Y_d}^{Z_d} = z_d$, hence
  $w_{0,X}^{Z_d} = z_d/\ka_d$.  Finally, we have $\ka_d w_{0,Z_d} = w_{0,Y_d} =
  w_{0,Y_d}^{-1} = w_{0,Z_d} \ka_d$, and $z_d = w_{0,X}^{Z_d} \ka_d = w_{0,X}
  w_{0,Z_d} \ka_d = w_{0,X} w_{0,Y_d}$, which completes the proof.
\end{proof}

\targetsec{posetFd}{}%
The variety $F_d = p_d^{-1}(1.P_X) = P_X/P_{Z_d}$ is a product of cominuscule
varieties by \Proposition{defYd}, and the Schubert varieties in this space
(which are products of Schubert varieties in the factors of $F_d$) are indexed
by elements of the set $W^{F_d} \subset W_X$ of minimal representatives of the
cosets in $W_X/W_{Z_d}$.  The maximal element in $W^{F_d}$ is $w_{0,X}^{Z_d} =
z_d/\ka_d$, so the elements of $W^{F_d}$ correspond to order ideals in
$\cP_{F_d} = I(z_d/\ka_d)$.  This subset of $\Phi^+$ is always disjoint from
$\cP_X$ (see \Example{dissectPXa}).  Notice that if $F_d$ has more than one
cominuscule factor, then $\cP_{F_d}$ is a disjoint union of the corresponding
partially ordered sets. For $\eta \in \cP_{F_d}$ we set $\la'(\eta) = \{ \eta'
\in \cP_{F_d} \mid \eta' < \eta \}$. Then the labeling $\delta' : \cP_{F_d} \to
\Delta_X$ is given by $\delta'(\eta) = w'_{\la'(\eta)}.\eta$, where
$w'_{\la'(\eta)}$ is the product of the reflections $s_{\eta'}$ for $\eta' \in
\la'(\eta)$, in increasing order.

\begin{prop}\label{prop:biject}
  The following order isomorphisms are obtained by restricting the actions of
  Weyl group elements.
  \begin{abcenum}
  \item $\ka_d : \cP_{F_d} \to I(z_d)\ssm I(\ka_d)$ is an order isomorphism, and
    $\delta(\ka_d.\eta) = \delta'(\eta)$ for each $\eta \in
    \cP_{F_d}$.\smallskip

  \item $z_d : \cP_{F_d} \to I(\ka_d^\vee)\ssm I(z_d^\vee)$ is an order
    isomorphism, and $\delta(z_d.\eta) = w_0^X.\delta'(\eta)$ for each $\eta \in
    \cP_{F_d}$.
  \end{abcenum}
\end{prop}
\begin{proof}
  It follows from \Lemma{skew}(a) that $\ka_d : I(z_d)\ssm I(\ka_d) \to
  \cP_{F_d}$ is an order-preserving bijection, and since $\ka_d.\be > 0$ for
  each $\be \in \Delta \ssm \{\ga\}$, the inverse bijection $\ka_d^{-1}=\ka_d$
  is also order-preserving. Given $\eta \in \cP_{F_d}$, the set $\mu = I(\ka_d)
  \cup \ka_d.\la'(\eta)$ is a straight shape in $\cP_X$ such that $\ka_d.\eta$
  is a minimal box of $\cP_X \ssm \mu$, hence $\delta(\ka_d.\eta) =
  w_\mu.(\ka_d.\eta)$.  Using that $w_\mu = w'_{\la'(\eta)} \ka_d$, we obtain
  $\delta(\ka_d.\eta) = w'_{\la'(\eta)}.\eta = \delta'(\eta)$. This proves part
  (a). \Lemma{PXinvol}(a) applied to $F_d$ shows that $w_{0,Z_d} : \cP_{F_d} \to
  \cP_{F_d}$ is an order-reversing involution such that $\delta'(w_{0,Z_d}.\eta)
  = -w_{0,X}.\delta'(\eta)$.  Part (b) therefore follows from
  \Lemma{PXinvol}(a), part (a), and the identities $z_d = w_{0,X} \ka_d
  w_{0,Z_d}$ and $w_0 w_{0,X} = w_0^X$.
\end{proof}

\begin{example}\label{example:dissectPXa}%
  Consider $X = \Gr(7,17)$ and $d=4$, so that $\Gamma_d = \Gr(4,8)$ and $F_d =
  \Gr(3,7) \times \Gr(4,10)$. Let $\Phi^+ = \{e_i-e_j \mid 1 \leq i < j \leq 17
  \}$ be the set of positive roots of type $A_{16}$.  We identify each root
  $e_i-e_j$ with the box in row $17-i$ and column $j-1$ of a triangular diagram
  of boxes. \Proposition{biject} shows that $\cP_{F_d}$ can be identified
  with $\ka_d.\cP_{F_d} = I(z_d)\ssm I(\ka_d)$, and also with $z_d.\cP_{F_d} =
  I(\ka_d^\vee)\ssm I(z_d^\vee)$. This gives two dissections of $\cP_X$. Notice
  that $\cP_{\Gamma_d} = I(\ka_d)$, and $\cP_{F_d}$ can be identified with the
  disjoint union of $\cP_{\Gr(3,7)}$ and $\cP_{\Gr(4,10)}$.

  \begin{center}
  \begin{tikzpicture}[x=.4cm,y=.4cm]
    \def\zz{-- ++(0,1) -- ++(1,0)}
    \draw (16,16) -- (16,0) -- (0,0)
    \zz\zz\zz\zz\zz\zz\zz\zz\zz\zz\zz\zz\zz\zz\zz -- ++(0,1) -- cycle;
    \draw [gray!50] (6,6) -- (6,7) -- (7,7) -- (7,6) -- cycle;
    \node at (6.5,6.5) {$\ga$};
    \draw (6,0) -- (6,7) -- (16,7) -- (16,0) -- cycle;
    \node at (11,3.5) {$\cP_X$};
    \draw (2,0) -- (2,3) -- (6,3) -- (6,0) -- cycle;
    \node at (4,1.5) {$\cP_{F_d}$};
    \draw (10,7) -- (10,11) -- (16,11) -- (16,7) -- cycle;
    \node at (13,9) {$\cP_{F_d}$};

    \begin{scope}[shift={(20.5,9)}]
      \node at (-2,3.5) {$\cP_X =$};
      \draw (0,0) -- (0,7) -- (10,7) -- (10,0) -- cycle;
      \draw (0,3) -- (10,3);
      \draw (4,0) -- (4,7);
      \node at (2,5) {$\cP_{\Gamma_d}$};
      \node at (2,1.5) {$\ka_d.\cP_{F_d}$};
      \node at (7,5) {$\ka_d.\cP_{F_d}$};
      \node at (7,1.5) {$\cP_X \!\ssm\! I(z_d)$};
    \end{scope}

    \begin{scope}[shift={(20.5,0)}]
      \node at (-2,3.5) {$\cP_X =$};
      \draw (0,0) -- (0,7) -- (10,7) -- (10,0) -- cycle;
      \draw (0,4) -- (10,4);
      \draw (6,0) -- (6,7);
      \node at (3,5.5) {$I(z_d^\vee)$};
      \node at (3,2) {$z_d.\cP_{F_d}$};
      \node at (8,5.5) {$z_d.\cP_{F_d}$};
      \node at (8,2) {\tiny$\cP_X \!\ssm\! I(\ka_d^\vee)$};
    \end{scope}
  \end{tikzpicture}
  \end{center}
\end{example}

\begin{example}\label{example:dissectPXc}
  Let $X = \LG(8,16)$ and $d=5$.  Then \Proposition{biject} provides the
  following dissections of $\cP_X$.

  \begin{center}
  \begin{tikzpicture}[x=.4cm,y=.4cm]
    \def\zz{-- ++(0,-1) -- ++(1,0)}
    \begin{scope}[shift={(0,0)}]
      \node at (-2,4) {$\cP_X =$};
      \draw [gray!50] (0,8) -- (0,7) -- (1,7) -- (1,8) -- cycle;
      \node at (.5,7.5) {$\ga$};
      \draw (0,8) \zz\zz\zz\zz\zz\zz\zz\zz -- (8,8) -- cycle;
      \draw (5,8) -- (5,3) -- (8,3);
      \node at (3.5,6.2) {$\cP_{\Gamma_d}$};
      \node at (6.5,6.2) {$\ka_d.\cP_{F_d}$};
      \node at (6.5,2.5) {\tiny$\cP_X\!\!\ssm\!I(z_d)$};
    \end{scope}
    \begin{scope}[shift={(11,0)}]
      \node at (-1.2,4) {$=$};
      \draw (0,8) \zz\zz\zz\zz\zz\zz\zz\zz -- (8,8) -- cycle;
      \draw (3,8) -- (3,5) -- (8,5);
      \node at (2,6.7) {\tiny$I(z_d^\vee)$};
      \node at (5.5,6.7) {$z_d.\cP_{F_d}$};
      \node at (6,3.7) {\tiny$\cP_X\!\!\ssm\!I(\ka_d^\vee)$};
    \end{scope}
  \end{tikzpicture}
  \end{center}
\end{example}

\subsection{Fibers of the quantum-classical diagram}\label{sec:ss:qcfibers}%

\begin{defn}\label{defn:fiber}%
  Given $u, v \in W^X$ and $0 \leq d \leq d_X(2)$, define the Weyl group
  elements
  \begin{align*}
    u(d) &= (u \cap z_d^\vee) z_d \,,
    & \wh u_d &= (u \cap \ka_d^\vee)/(u \cap z_d^\vee) \,,
    & u_d &= (w_0^X)^{-1}\, \wh u_d\, w_0^X \,, \\
    v(-d) &= (v \cup z_d) / z_d \,,
    & & \text{and}
    & v^d &= (v \cap z_d)/(v \cap \ka_d) \,.
  \end{align*}
\end{defn}

It was proved in \cite{buch.chaput.ea:chevalley} that $X_{u(d)} = \Gamma_d(X_u)$
and $X^{v(-d)} = \Gamma_d(X^v)$.  We will reprove these statements in
\Corollary{fibers_geom} below, together with similar geometric interpretations
of $u_d$ and $v^d$. In particular, $u(d)$ and $v(-d)$ belong to $W^X$. The shape
$I(v(-d))$ is obtained from $I(v)$ by removing the boxes in $I(v \cap z_d)$ and
moving the remaining boxes north-west until they fit in the upper-left corner of
$\cP_X$ (see \cite[\S 3.2]{buch.chaput.ea:chevalley}). Similarly, $I(u(d))$ is
obtained by attaching the shape $I(u)$ to the south-east border of $I(z_d)$ and
discarding any boxes that do not fit within $\cP_X$. More precisely, the
following identities follow from \Proposition{firstbiject}(a) and
\Corollary{zd=z1k1d}.

\begin{lemma}\label{lemma:u(d)}%
  We have $I(u(-d)) = (z_1 s_\ga)^d.(I(u) \ssm I(z_d))$ and $I(u(d)) = I(z_d)
  \cup (z_1 s_\ga)^{-d}.(I(u) \cap I(z_d^\vee))$.
\end{lemma}

Our next result shows that $u_d$ and $v^d$ belong to $W^{F_d} = W_X \cap
W^{Z_d}$, and the shapes of these elements in $\cP_{F_d}$ are determined by
$z_d.I(u_d) = I(u) \cap z_d.\cP_{F_d}$ and $\ka_d.I(v^d) = I(v) \cap
\ka_d.\cP_{F_d}$ (see \Example{dissectPXa}, \Example{dissectPXc}, and
\Example{fiber-diagrams}). Recall from \Section{flagvar} that the parabolic
factorization of $v \in W$ with respect to $P_{Y_d}$ is denoted by $v = v^{Y_d}
v_{Y_d}$.

\begin{prop}\label{prop:fibers_combin}%
  Let $u, v \in W^X$ and $0 \leq d \leq d_X(2)$.
  \begin{abcenum}
  \item We have $v_{Y_d} = v\cap\ka_d$ and $v^{Y_d} = (v\cup\ka_d)/\ka_d$, and
    the parabolic factorization of $v^{Y_d}$ with respect to $P_X$ is $v^{Y_d} =
    v(-d)v^d$.\smallskip

  \item We have $u_d, v^d \in W^{F_d}$, with shapes given by
    $I(u_d) = z_d^{-1}.I(u) \cap \cP_{F_d}$ and $I(v^d) = \ka_d.I(v) \cap
    \cP_{F_d}$.\smallskip

  \item We have that $u^\vee(-d) = u(d)^\vee = w_0 u(d) w_{0,X}$ is dual to
    $u(d)$ in $W^X$, and $(u^\vee)^d = w_{0,X} u_d w_{0,Z_d}$ is dual to $u_d$
    in $W^{F_d}$.
  \end{abcenum}
\end{prop}
\begin{proof}
  The element $\wh u_d$ is by definition the product of the simple reflections
  $s_{\delta(\al)}$ for all boxes $\al$ in $I(u) \cap (I(\ka_d^\vee) \ssm
  I(z_d^\vee))$, in decreasing order. \Proposition{biject}(b) therefore shows
  that
  \[
    u_d = (w_0^X)^{-1} \wh u_d w_0^X
    = \prod_{\eta \in z_d^{-1}.I(u) \cap \cP_{F_d}} s_{\delta'(\eta)} \,,
  \]
  the product in decreasing order.  This shows that $u_d \in W^{F_d}$ and
  $I(u_d) = z_d^{-1}.I(u) \cap \cP_{F_d}$.
  \Proposition{biject}(a) similarly shows that $v^d \in W^{F_d}$ and $I(v^d) =
  \ka_d.I(v) \cap \cP_{F_d}$.  This proves part (b).

  Since $v\cup\ka_d \in W^X \subset W^{Z_d}$, the product of
  $(v\cup\ka_d)/\ka_d$ with $w_{0,Y_d} = \ka_d w_{0,Z_d}$ is reduced, hence
  $(v\cup\ka_d)/\ka_d \in W^{Y_d}$.  Since $v\cap\ka_d \in W_{Y_d}$, we deduce
  that the parabolic factorization $v = v^{Y_d} v_{Y_d}$ is given by $v^{Y_d} =
  (v\cup\ka_d)/\ka_d$ and $v_{Y_d} = v\cap\ka_d$. Since $u(d) \leq_L z_d^\vee
  z_d = w_0^X$, we have $u(d) \in W^X$.  The dual element is $w_0 u(d) w_{0,X} =
  w_0 (u \cap z_d^\vee) w_{0,X} z_d^{-1} = (u^\vee \cup z_d)/z_d = u^\vee(-d)$.
  This implies that $v(-d) \in W^X$, and since $v^d \in W_X$, it follows that
  $v^{Y_d} = v(-d) v^d$ is the parabolic factorization of $v^{Y_d}$ with respect
  to $P_X$.  This proves part (a).

  The element $\tau = z_d^\vee/(u\cap z_d^\vee)$ commutes with $\wh u_d$ and
  satisfies $\tau u(d) = z_d^\vee z_d = w_0^X$.  This implies $u_d =
  (w_0^X)^{-1} \wh u_d w_0^X = u(d)^{-1} \wh u_d u(d)$, hence $w_0 u^\vee(-d)
  (u^\vee)^d w_{0,Z_d} = w_0 (u^\vee \cup \ka_d) \ka_d w_{0,Z_d} = (u \cap
  \ka_d^\vee) z_d = \wh u_d u(d) = u(d) u_d$.   This shows that $u^\vee(-d)
  (u^\vee)^d$ is dual to $u(d) u_d$ in $W^{Z_d}$. Since $u^\vee(-d)$ is dual to
  $u(d)$ in $W^X$, it follows that $(u^\vee)^d$ is dual to $u_d$ in $W^{F_d}$,
  see \Remark{oppschubfib}.  This completes the proof of part (c).
\end{proof}

In examples we denote an element $u \in W^X$ by the partition $\la = (\la_1,
\la_2, \dots, \la_p)$ for which $\la_j$ is the number of boxes in the $j$-th row
of $I(u)$.

\begin{example}\label{example:fiber-diagrams}%
  Let $X = \LG(8,16)$, $u = (8,6,2) \in W^X$, $v = u^\vee = (7,5,4,3,1)$, and
  set $d=5$. The shape of $u_d = (5,4,1)$ is obtained by intersecting the shape
  of $u$ with $z_d.\cP_{F_d}$, and the shape of $v^d = (2,1,1,1)$ is obtained by
  intersecting the shape of $v$ with $\ka_d.\cP_{F_d}$, see
  \Example{dissectPXc}.\smallskip
  \begin{center}
    \begin{tikzpicture}[x=.4cm,y=.4cm]
      \def\zz{-- ++(0,-1) -- ++(1,0)}
      \begin{scope}[shift={(0,0)}]
        \draw (0,8) \zz\zz\zz\zz\zz\zz\zz\zz -- (8,8) -- cycle;
        \draw (3,8) -- (3,5) -- (8,5);
        \draw[very thick] (0,8) \zz\zz -- (2,5) -- (4,5) -- (4,6) -- (7,6)
          -- (7,7) -- (8,7) -- (8,8) -- cycle;
      \end{scope}
      \begin{scope}[shift={(11,0)}]
        \draw (0,8) \zz\zz\zz\zz\zz\zz\zz\zz -- (8,8) -- cycle;
        \draw (5,8) -- (5,3) -- (8,3);
        \draw[very thick] (0,8) \zz\zz\zz\zz\zz -- (5,4) -- (6,4) -- (6,7)
          -- (7,7) -- (7,8) -- cycle;
      \end{scope}
    \end{tikzpicture}
  \end{center}
  The elements $u_d$ and $v^d$ are dual to each other in $W^{F_d}$, but the
  images $z_d.I(u_d)$ and $\ka_d.I(v^d)$ of their shapes are represented in two
  different rectangles in $\cP_X$. The composed bijection $z_d/\ka_d :
  \ka_d.\cP_{F_d} \cong \cP_{F_d} \cong z_d.\cP_{F_d}$ is given by a
  transposition when $X$ is a Lagrangian Grassmannian. An expression of an
  element of $W^{F_d}$ as a partition therefore depends on how the rectangle
  $\cP_{F_d}$ is oriented. Opposite conventions are used in the expressions $u_d
  = (5,4,1)$ and $v^d = (2,1,1,1)$ given above. We also have $u(d) = w_0^X$ and
  $v(-d) = 1$. Other shifts of $u$ include $u(-2) = (2)$, $u(-1) = (6,2)$, $u(1)
  = (8,7,6,2)$, and $u(2) = (8,7,6,5,2)$.
\end{example}

\begin{cor}\label{cor:fibers_geom}
  Let $u, v \in W^X$ and $0 \leq d \leq d_X(2)$.
  \begin{abcenum}
  \item The general fibers of the map $q_d : p_d^{-1}(X^v) \to Y_d(X^v)$ are
  translates of $(\Gamma_d)^{v\cap\ka_d}$.\smallskip

  \item We have $Y_d(X^v) = (Y_d)^{v(-d)v^d}$, $\Gamma_d(X^v) = X^{v(-d)}$, and
    the general fibers of the map $p_d : Z_d(X^v) \to \Gamma_d(X^v)$ are
    translates of $(F_d)^{v^d}$.\smallskip

  \item We have $Z_d(X_u) = (Z_d)_{u(d)u_d}$, $\Gamma_d(X_u) = X_{u(d)}$, and
    the general fibers of the map $p_d : Z_d(X_u) \to \Gamma_d(X_u)$ are
    translates of $(F_d)_{u_d}$.\smallskip

  \item The general fibers of the map $p_d: Z_d(X_u,X^v) \to \Gamma_d(X_u,X^v)$
    are translates of $(F_d)_{u_d}^{u_d \cap v^d}$.
  \end{abcenum}
\end{cor}
\begin{proof}
  Parts (a) and (b) follow from \Theorem{schubfib}, \Corollary{qc_nbhd}, and
  \Proposition{fibers_combin}(a), and \Proposition{fibers_combin}(c) implies
  that part (c) is equivalent to part (b), see \Remark{oppschubfib}.  Finally,
  part (d) follows from \Theorem{richfib} and \Proposition{comint} together with
  parts (b) and (c).
\end{proof}

%% file: qhinterval.tex

\section{The $q$-degrees in quantum cohomology products}
\label{sec:qhinterval}

\subsection{Quantum cohomology}\label{sec:qcohom}%

\targetsec{qhprod}{}%
Let $X = G/P_X$ be a cominuscule flag variety, and let $\QH(X)$ be the (small)
quantum cohomology ring of $X$. As an additive group, this ring is defined as
$\QH(X) = H^*(X;\Z) \otimes_\Z \Z[q]$. Multiplication is defined by
\[
  [X_u] \star [X^v] = \sum_{w,d\geq 0} \gw{[X_u],[X^v],[X_w]}{d}\,
  q^d\, [X^w]
\]
for $u, v \in W^X$. Let
\[
  ([X_u] \star [X^v])_d = \sum_w \gw{[X_u], [X^v], [X_w]}{d}\, [X^w]
\]
denote the coefficient of $q^d$ in this product. The goal of this section is to
identify the degrees $d$ for which $([X_u] \star [X^v])_d \neq 0$. In
particular, we will show that these degrees form an integer interval.

\targetsec{dminmax}{}%
It follows from \cite[Thm.~9.1]{fulton.woodward:quantum} that the smallest
degree $d$ for which $([X_u] \star [X^v])_d \neq 0$ is equal to the \emph{degree
distance} between $X_u$ and $X^v$, defined as the minimal $d$ for which
$\Gamma_d(X_u,X^v) \neq \emptyset$. In particular, the quantum product of two
Schubert classes is never zero. Let $\dmin(u^\vee,v)$ and $\dmax(u^\vee,v)$
denote the minimal and maximal degrees for which $([X_u]\star[X^v])_d \neq 0$.
We let $\dmax(v)$ denote the (unique) number of occurrences of $s_\ga$ in a
reduced expression for $v$. Notice that $d = \dmax(v)$ is also determined by
$\ka_d \leq v \leq z_d$. The following result implies that $\dmax(v)$ is the
only power of $q$ that occurs in the product $[\pt] \star [X^v]$, that is
$\dmax(v) = \dmin(w_0^X,v) = \dmax(w_0^X,v)$. More generally, it was proved in
\cite{belkale:transformation, chaput.manivel.ea:affine} that $[\pt] \star [X^v]
= q^{\dmax(v)}\,[X^{w_0^X v}]$ holds in $\QH(X)$.

\begin{prop}\label{prop:qhcoef}%
  We have $([X_u]\star[X^v])_d \neq 0$ if and only if $\dmin(u^\vee,v) \leq d
  \leq \min(\dmax(u^\vee),\dmax(v))$ and $u_d \leq v^d$. In this case we have
  $([X_u]\star[X^v])_d = [\Gamma_d(X_u,X^v)]$.
\end{prop}
\begin{proof}
  Using \Corollary{qc} and the projection formula we obtain
  \[
  \gw{[X_u], [X^v], [X_w]}{d} = \int_X {p_d}_* q_d^*\big(
  {q_d}_* p_d^*[X_u] \cdot {q_d}_* p_d^*[X^v]\big) \cdot [X_w] \,,
  \]
  which implies that
  \[
  ([X_u] \star [X^v])_d = {p_d}_* q_d^*\big({q_d}_* p_d^*[X_u] \cdot
  {q_d}_* p_d^*[X^v]\big) \,.
  \]
  \Corollary{fibers_geom}(a) shows that ${q_d}_* p_d^*[X^v]$ is equal to
  $[q_d(p_d^{-1}(X^v))]$ for $d \leq \dmax(v)$ and is zero otherwise. It follows
  that $([X_u]\star[X^v])_d$ is non-zero only if $d \leq
  \min(\dmax(u^\vee),\dmax(v))$, in which case
  \[
  ([X_u]\star[X^v])_d = p_*[Z_d(X_u,X^v)] \,.
  \]
  The proposition therefore follows from \Corollary{fibers_geom}(d).
\end{proof}

The main technical result of this section is the following lemma, which we will
prove after discussing its consequences. Notice that \Proposition{biject}(a)
shows that $w \ka_d \in W^X$ and $I(w \ka_d) = I(\ka_d) \cup \ka_d.I(w)$ for
each $w \in W^{F_d}$.

\begin{lemma}\label{lemma:fiber-dtod1}%
  Let $u, v \in W^X$.
  \begin{abcenum}
  \item For $0 < d \leq \dmax(u^\vee)$ we have
    $u_{d-1} \ka_{d-1} = (u_d \ka_d)(-1)$.\smallskip
  \item For $0 < d \leq \dmax(v)$ we have
    $v^{d-1} \ka_{d-1} = (v^d \ka_d) \cap z_{d-1}$.
  \end{abcenum}
\end{lemma}

Part (b) of the following lemma will be relevant for our study of quantum
$K$-theory in \Section{quantumk}.

\begin{lemma}\label{lemma:ud<vd}%
  Let $u, v \in W^X$ and $0 < d \leq \min(\dmax(u^\vee), \dmax(v))$.
  \begin{abcenum}
  \item Assume that $u_d \leq v^d$. Then $u_{d-1} \leq v^{d-1}$.\smallskip

  \item If $I(u_d) \ssm I(v^d)$ is a non-empty rook strip in
    $\cP_{F_d}$, then $d = \dmax(u^\vee,v)+1$.
  \end{abcenum}
\end{lemma}
\begin{proof}
  The inequality $u_d \leq v^d$ holds if and only if $u_d \ka_d \leq v^d \ka_d$.
  Part (a) therefore follows from \Lemma{fiber-dtod1}, using that
  $(u_d\ka_d)(-1) \leq u_d\ka_d$ and  $(u_d\ka_d)(-1) \leq z_d(-1) \leq
  z_{d-1}$. If $u_d \not\leq v^d$, then it follows from \Proposition{qhcoef} and
  part (a) that $d > \dmax(u^\vee,v)$. If $I(u_d)\ssm I(v^d)$ is a rook strip,
  then $(u_d\ka_d)(-1) \leq v^d\ka_d$, which implies that $u_{d-1} \leq v^{d-1}$
  and $d-1 \leq \dmax(u^\vee,v)$. This proves part (b).
\end{proof}

The following result was proved in \cite{postnikov:affine} for Grassmannians of
type A.

\begin{cor}\label{cor:qhint}%
  The $q$-degrees appearing in a quantum product $[X_u] \star [X^v]$ form an
  interval, that is $([X_u] \star [X^v])_d \neq 0$ if and only if
  $\dmin(u^\vee,v) \leq d \leq \dmax(u^\vee,v)$.
\end{cor}
\begin{proof}
  This follows from \Proposition{qhcoef} and \Lemma{ud<vd}.
\end{proof}

The equivariant (small) quantum cohomology ring $\QH_T(X) = H_T^*(X;\Z)
\otimes_\Z \Z[q]$ is defined like $\QH(X)$, except that equivariant
Gromov-Witten invariants are used to define the quantum product $[X_u]_T \star
[X^v]_T$, see \cite{kim:quantum}. \Proposition{qhcoef} is true also for the
equivariant quantum product, with the same proof.

\begin{cor}
  The equivariant quantum product $[X_u]_T \star [X^v]_T$ in $\QH_T(X)$ contains
  the same powers of $q$ as the non-equivariant product $[X_u] \star [X^v]$.
\end{cor}

\begin{remark}
  It is natural to ask whether the powers of $q$ appearing in an equivariant
  quantum product of Schubert classes defined by the {\em same} Borel subgroup
  form an interval. Based on substantial computer evidence we conjecture that
  $q^d$ occurs in the product $[X^u]_T \star [X^v]_T$ in $\QH_T(X)$ if and only
  if $0 \leq d \leq \dmax(u,v)$.
\end{remark}

\subsection{The minimal and maximal degrees}\label{sec:qh-min-max}

We next give a type-uniform description of the minimal and maximal powers of $q$
in the quantum product $[X_u] \star [X^v]$ that generalizes the description for
Grassmannians of type A proved in \cite{fulton.woodward:quantum,
postnikov:affine}. The minimal degree $\dmin(u^\vee,v)$ is the smallest integer
$d$ for which $\Gamma_d(X_u,X^v) \neq \emptyset$. Since $\Gamma_d(X_u,X^v)$ is
non-empty if and only if $v \leq u(d)$, the degree $\dmin(u^\vee,v)$ can be
interpreted as the number of steps the shape $I(u)$ must be shifted in order to
contain $I(v)$. This recovers Fulton and Woodward's description in type A
\cite{fulton.woodward:quantum}. Postnikov gave a similar description
\cite{postnikov:affine} of the maximal degree $\dmax(u^\vee,v)$ in type A, as
the number of steps $I(u)$ can be shifted before it no longer fits inside a
shape. These descriptions of the powers of $q$ in a quantum product are
generalized in \Theorem{qdeg-shapes} and \Theorem{qposet}. The maximal degree
for arbitrary cominuscule varieties is given by the following formula from
\cite[Thm.~1.2]{chaput.manivel.ea:quantum}.

\begin{thm}\label{thm:dmax-strange}%
  We have $\dmax(u^\vee,v) = \dmax(v) -  \dmin((w_0^X v)^\vee, u)$.
\end{thm}

\targetsec{qbruhat}{}%
Let $\cB = \{q^d [X^u] : u \in W^X, d \in \Z\}$ denote the natural $\Z$-basis of
the localized quantum cohomology ring $\QH(X)_q = \QH(X) \otimes_{\Z[q]}
\Z[q,q^{-1}]$. We define a partial order on $\cB$ by
\[
  q^e [X^v] \leq q^d [X^u] \ \ \Longleftrightarrow \ \
  \Gamma_{d-e}(X_u,X^v) \neq \emptyset \ \ \Longleftrightarrow \ \
  X_v \subset \Gamma_{d-e}(X_u) \,.
\]
Here the sets $\Gamma_{d-e}(X_u,X^v)$ and $\Gamma_{d-e}(X_u)$ can be non-empty
only if $d \geq e$. Notice that if $q^f [X^w] \leq q^e [X^v] \leq q^d [X^u]$,
then $X_w \subset \Gamma_{e-f}(X_v) \subset \Gamma_{e-f}(\Gamma_{d-e}(X_u))
\subset \Gamma_{d-f}(X_u)$ shows that $q^f [X^w] \leq q^d [X^u]$. The order on
$\cB$ extends the Bruhat order on $W^X$, and
\cite[Thm.~9.1]{fulton.woodward:quantum} implies that $q^e [X^v] \leq q^d [X^u]$
holds if and only if $q^d [X^u]$ occurs with non-zero coefficient in the
expansion of $q^e[X^v] \star q^{d'}[X^w]$ in $\QH(X)_q$, for some $w \in W^X$
and $d' \geq 0$ \footnote{Our construction defines a partial order on $\cB = \{
q^d [M^u] : u \in W^M, d \in H_2(M,\Z) \}$ for any flag variety $M$, with the
same interpretation in terms of quantum multiplication.}. Notice that $[\pt]
\star [X^v] = q^{\dmax(v)} [X^{w_0^X v}]$ is an element of $\cB$, as proved in
\cite{belkale:transformation, chaput.manivel.ea:affine}.

\begin{thm}\label{thm:qdeg-shapes}%
  Let $u, v \in W^X$ and $d \in \Z$. The power $q^d$ occurs in $[X_u] \star
  [X^v]$ if and only if $[X^v] \leq q^d [X^u] \leq [\pt] \star [X^v]$.
\end{thm}
\begin{proof}
  The smallest degree $d$ for which $[X^v] \leq q^d [X^u]$ is $\dmin(u^\vee,v)$
  by definition of the partial order on $\cB$, and the largest degree $d$ for
  which $q^d[X^u] \leq q^{\dmax(v)}[X^{w_0^X v}]$ is $\dmax(u^\vee,v)$ by
  \Theorem{dmax-strange}.
\end{proof}

\subsection{Distributive lattice}\label{sec:dist-lattice}%

Our next goal is to show that $\cB$ is a distributive lattice. In the remainder
of this section we extend earlier definitions by setting $u(d) = z_d = w_0^X$
and $u(-d) = 1$ for $d \geq d_X(2)$ and $u \in W^X$.

\begin{lemma}\label{lemma:shift}%
  Let $u, v \in W^X$, $d \in \Z$, and $e \in \N$. The following identities hold.
  \begin{align*}
    (u \cup v)(d) &= u(d) \cup v(d) &
    (u \cap v)(d) &= u(d) \cap v(d) \\
    (u \cup v(d))(e) &= u(e) \cup v(d+e) &
    (u(d) \cap v)(-e) &= u(d-e) \cap v(-e)
  \end{align*}
\end{lemma}
\begin{proof}
  The identities $(u \cup v)(d) = u(d) \cup v(d)$ and $(u \cap v)(d) = u(d) \cap
  v(d)$ follow from \Lemma{u(d)}, which also implies that $u(-e)(e) = u \cup
  z_e$. We claim that $u(d)(e) = u(d+e) \cup z_e$. For $d \geq 0$ it follows
  from \Theorem{dist} that $\Gamma_e(\Gamma_d(X_u)) = \Gamma_{d+e}(X_u)$, which
  implies $u(d)(e) = u(d+e) = u(d+e) \cup z_e$ by \Corollary{fibers_geom}(c). A
  dual argument shows that $u(-d)(-e) = u(-d-e)$ for $d \geq 0$. For $e' \geq 0$
  we obtain $u(-e-e')(e) = u(-e')(-e)(e) = u(-e') \cup z_e$ and $u(-e)(e+e') =
  u(-e)(e)(e') = (u \cup z_e)(e') = u(e') \cup z_{e+e'}$. This proves all cases
  of $u(d)(e) = u(d+e) \cup z_e$. Using that $z_e \leq u(e)$, we obtain $u(e)
  \cup v(d+e) = u(e) \cup z_e \cup v(d+e) = u(e) \cup v(d)(e) = (u \cup
  v(d))(e)$. The last identity of the lemma follows from this, using that
  $u(d)^\vee = u^\vee(-d)$ by \Proposition{fibers_combin}(c).
\end{proof}

\begin{prop}\label{prop:qposet-dist}%
  The partially ordered set $\cB = \{ q^d [X^u] : u \in W^X, d \in \Z\}$ is a
  distributive lattice with meet and join operations given by
  \[
    q^d [X^u] \cap q^e [X^v] = q^e [X^{u(d-e) \cap v}]
    \text{ \ \ and \ \ }
    q^d [X^u] \cup q^e [X^v] = q^d [X^{v(e-d) \cup u}]
  \]
  for $u,v \in W^X$ and $e \leq d$.
\end{prop}
\begin{proof}
  In this proof we denote $q^d [X^u]$ by $[d,u]$ for brevity. Let $u, v, w \in
  W^X$ and $d, e, f \in \Z$. The partial order on $\cB = \Z \times W^X$ is
  defined by
  \[
    [e,v] \leq [d,u]
    \ \Leftrightarrow \
    \big( e \leq d \text{ and } v \leq u(d-e) \big)
    \ \Leftrightarrow \
    \big( e \leq d \text{ and } v(e-d) \leq u \big) \,.
  \]
  We first show that $[e, u(d-e) \cap v]$ is the greatest lower bound of $[d,u]$
  and $[e,v]$ when $e \leq d$. The relations $[e, u(d-e) \cap v] \leq [d,u]$ and
  $[e, u(d-e) \cap v] \leq [e,v]$ follow from the definition. If $[f,w] \leq
  [d,u]$ and $[f,w] \leq [e,v]$, then
  \[
    w \leq u(d-f) \cap v(e-f) = (u(d-e) \cap v)(e-f) \,,
  \]
  hence $[f,w] \leq [e, u(d-e) \cap v]$. This proves $[d,u] \cap [e,v] = [e,
  u(d-e) \cap v]$. Noting that the map $[d,u] \mapsto [-d,u^\vee]$ is an
  order-reversing involution of $\cB$, the expression $[d,u] \cup [e,v] = [d, u
  \cup v(e-d)]$ for the least upper bound is equivalent to the expression for
  the greatest lower bound. This shows that $\cB$ is a lattice. To prove
  distributivity, assume again that $e \leq d$, and set $n = \max(e,f)$ and $m =
  \max(d,f)$. Notice that $ u(d-m)(m-n) = u(d-n) $, as we have either $m=d$ or
  $m=f=n$. Using \Lemma{shift} we obtain
  \[
    \def\myeq{\\ & \ \ = \ }
    \begin{split}
      &
      ([d,u] \cup [f,w]) \cap ([e,v] \cup [f,w])
      \myeq
      [m,u(d-m) \cup w(f-m)] \cap [n,v(e-n) \cup w(f-n)]
      \myeq
      [n,(u(d-m) \cup w(f-m))(m-n) \cap (v(e-n) \cup w(f-n))]
      \myeq
      [n,(u(d-m)(m-n) \cup w(f-n)) \cap (v(e-n) \cup w(f-n))]
      \myeq
      [n,(u(d-n) \cup w(f-n)) \cap (v(e-n) \cup w(f-n))]
      \myeq
      [n,(u(d-n) \cap v(e-n)) \cup w(f-n)]
      \myeq
      [n,(u(d-e) \cap v)(e-n) \cup w(f-n)]
      \myeq
      [e,u(d-e) \cap v] \cup [f,w]
      \myeq
      ([d,u] \cap [e,v]) \cup [f,w] \,.
    \end{split}
  \]
  Since this identity is formally equivalent to
  \[
    ([d,u] \cup [e,v]) \cap [f,w] = ([d,u] \cap [f,w]) \cup ([e,v] \cap [f,w])
    \,,
  \]
  this completes the proof.
\end{proof}

\begin{defn}\label{defn:quantum-poset}%
  An element $\wh\al \in \cB$ is called \emph{join-irreducible} if $\wh\al =
  \wh\al_1 \cup \wh\al_2$ implies $\wh\al = \wh\al_1$ or $\wh\al = \wh\al_2$,
  for $\wh\al_1, \wh\al_2 \in \cB$. Let $\wh\cP_X \subset \cB$ denote the subset
  of join-irreducible elements. Given $q^d [X^u] \in \cB$, set
  \[
    I(q^d [X^u]) = \{ \wh\al \in \wh\cP_X \mid \wh\al \leq q^d [X^u] \} \,.
  \]
  For $\al \in \cP_X$, define $\partial(\al) \in \N$, $\xi(\al) \in W^X$, and
  $\tau(\al) \in \cB$ by
  \[
    \begin{split}
      \partial(\al) &=
      \min\, \{ d \geq 0 \mid (z_1 s_\ga)^{-d}.\al \in \cP_X \ssm I(z_1^\vee) \}
      \,, \\
      I(\xi(\al)) &= \{ \al' \in \cP_X \mid \al' \leq \al \}
      \,\text{, and} \\
      \tau(\al) &= q^{-\partial(\al)} [X^{\xi(\be)}] \ , \text{\ \ where \ \ }
      \be = (z_1 s_\ga)^{-\partial(\al)}.\al \,.
    \end{split}
  \]
\end{defn}

The integer $\partial(\al)$ exists by \Proposition{z1k1}(a). For example, we
have
\begin{equation}\label{eqn:qposet-minmax}%
  \tau(\ga) = q^{1-d_X(2)} [X^{\ka_{d_X(2)}}]
  \text{ \ \ \ \ and \ \ \ \ }
  \tau(\rho) = [\pt] \,,
\end{equation}
where $\rho \in \Phi^+$ denotes the highest root. For any box $\al \in \cP_X$,
we will show in \Theorem{qposet} that $[X^{\xi(\al)}] = \tau(\al) \cup 1$ holds
in $\cB$, and $[X^{\xi(\al)}] = \tau(\al)$ is join-irreducible if and only if
$\al \in \cP_X \ssm I(z_1^\vee)$. This motivates the definition of $\tau(\al)$.

\begin{lemma}\label{lemma:tau-leq}%
  Let $\al \in \cP_X$ and $q^d [X^u] \in \cB$. We have $\tau(\al) \leq q^d
  [X^u]$ in $\cB$ if and only if $\al \in I(u(d))$.
\end{lemma}
\begin{proof}
  Set $e = \partial(\al)$ and $\be = (z_1 s_\ga)^{-e}.\al$, so that $\tau(\al) =
  q^{-e}[X^{\xi(\be)}]$. We then have $\tau(\al) \leq q^d [X^u]$ if and only if
  $-e \leq d$ and $\xi(\be) \leq u(d+e)$, or equivalently, $d+e \geq 0$ and $\be
  \in I(u(d+e)))$. Since $\be \in \cP_X \ssm I(z_1^\vee)$, the condition $d+e
  \geq 0$ follows from $\be \in I(u(d+e))$. Using that
  \[
    I(u(d+e)) \cup I(z_e) \, = \, I(u(d)(e)) \, = \,
    (z_1 s_\ga)^{-e}\big(I(u(d)) \cap I(z_e^\vee)\big) \cup I(z_e)
  \]
  by \Lemma{shift} and \Lemma{u(d)}, and that $\be \notin I(z_e)$ and $\al \in
  I(z_e^\vee)$, we deduce that $\be \in I(u(d+e))$ is equivalent to $\al \in
  I(u(d))$, as required.
\end{proof}

Part (d) of the following result is essentially a consequence of Birkhoff's
representation theorem \cite{birkhoff:rings} together with
\Proposition{qposet-dist}; we supply a proof since $\wh\cP_X$ is an infinite
set.

\begin{thm}\label{thm:qposet}%
  {\rm(a)} We have $[X^{\xi(\al)}] = \tau(\al) \cup 1$ in $\cB$ for each $\al
  \in \cP_X$.\smallskip

  \noin{\rm(b)} We have $\wh\cP_X = \{ q^d [X^{\xi(\al)}] : \al \in \cP_X \ssm
  I(z_1^\vee),\, d \in \Z \} \cup \{ q^d : d \in \Z \}$.\smallskip

  \noin{\rm(c)} The map $\tau : \cP_X \to \tau(\cP_X)$ is an order isomorphism
  onto an interval in $\wh\cP_X$.\smallskip

  \noin{\rm(d)} The map $q^d [X^u] \mapsto I(q^d [X^u])$ is an order isomorphism
  of $\cB$ with the set of non-empty, proper, lower order ideals in $\wh\cP_X$,
  ordered by inclusion.
\end{thm}
\begin{proof}
  If $q^d = q^e [X^v] \cup q^d [X^u]$, then $e \leq d$ and $v(e-d) \cup u = 1$,
  hence $q^d[X^u] = q^d$. This shows that $q^d$ is join-irreducible. Let $q^d
  [X^u] \in \cB$ satisfy $u \neq 1$. If $q^d [X^u]$ is join-irreducible, then
  $u$ is join-irreducible in $W^X$, so $u = \xi(\al)$ for some $\al \in \cP_X$.
  If $\al \in I(z_1^\vee)$, then $\be = (z_1 s_\ga)^{-1}.\al \in \cP_X$ by
  \Proposition{z1k1}(a), and it follows from \Lemma{u(d)} that $u =
  \xi(\be)(-1)$. But then $q^d [X^u] = q^d \cup q^{d-1} [X^{\xi(\be)}]$ is not
  join-irreducible, a contradiction. On the other hand, assume that $u =
  \xi(\al)$ where $\al \in \cP_X \ssm I(z_1^\vee)$. If $q^d [X^u] = q^e [X^v]
  \cup q^d [X^{u'}]$ with $e < d$, then $u = v(e-d) \cup u'$. Since $\al \notin
  I(v(e-d))$, we have $\al \in I(u')$, so $u' = u$. This proves part (b).

  Let $\al', \al \in \cP_X$, and set $u = \xi((z_1 s_\ga)^{-\partial(\al)}.\al)$.
  Using that $u(-\partial(\al)) = \xi(\al)$ by \Lemma{u(d)}, it follows from
  \Lemma{tau-leq} that $\tau(\al') \leq \tau(\al)$ holds if and only if $\al'
  \in I(\xi(\al))$, which is equivalent to $\al' \leq \al$. This shows that $\tau
  : \cP_X \to \tau(\cP_X)$ is an order isomorphism. To see that $\tau(\cP_X)$ is
  an interval in $\wh\cP_X$, assume that $\tau(\ga) \leq q^d [X^{\xi(\al)}] \leq
  \tau(\rho)$, where $\al \in \cP_X \ssm I(z_1^\vee)$ and $d \in \Z$. By
  \Lemma{tau-leq} and \eqn{qposet-minmax}, this is equivalent to $1 - d_X(2)
  \leq d \leq 0$ and $I(\xi(\al)(d)) \neq \emptyset$. We deduce that $(z_1
  s_\ga)^e.\al \in \cP_X \ssm I(z_1)$ for $0 \leq e < -d$, and $ q^d
  [X^{\xi(\al)}] = \tau((z_1 s_\ga)^{-d}.\al) \in \tau(\cP_X)$. This proves part
  (c).

  Given $\al \in \cP_X$, \Proposition{qposet-dist} implies that $\tau(\al) \cup
  1 = [X^u]$ for some $u \in W^X$. Since $\tau(\al') \not\leq 1$ for each $\al'
  \in \cP_X$ by \Lemma{tau-leq}, another application of \Lemma{tau-leq} shows
  that $\al' \in I(u)$ holds if and only if $\tau(\al') \leq \tau(\al)$, so it
  follows from part (c) that $u = \xi(\al)$. This proves part (a).

  Let $I \subset \wh\cP_X$ be any non-empty, proper, lower order ideal. Since
  $I$ has finitely many maximal elements, say $\wh\al_1, \dots, \wh\al_\ell$, it
  follows from \Proposition{qposet-dist} that $I$ has a well-defined least upper
  bound $q^d [X^u] = \wh\al_1 \cup \dots \cup \wh\al_\ell$ in $\cB$. For any
  element $\wh\be \in I(q^d [X^u])$, we obtain
  \[
    \wh\be = \wh\be \cap q^d [X^u] = (\wh\be \cap \wh\al_1) \cup \dots
    \cup (\wh\be \cap \wh\al_\ell) \,.
  \]
  Since $\wh\be$ is join-irreducible, this implies $\wh\be = \wh\be \cap
  \wh\al_i$ for some $i$, so $\wh\be \in I$. We deduce that $I = I(q^d [X^u])$.
  Part (d) follows from this, noting that any element $q^d [X^u] \in \cB$ is the
  least upper bound of the finite set $\{ q^d \} \cup \{ q^d \tau(\al) \mid \al
  \in I(u) \}$ by part (a).
\end{proof}

The following result elaborates on the relationship between $\cP_X$ and
$\wh\cP_X$.

\begin{prop}\label{prop:qposet-elaborate}%
  Let $\wh\al \in \wh\cP_X$ be any element and let $\rho \in \Phi^+$ be the
  highest root.\smallskip

  \noin{\rm(a)} We have $[\pt] \star 1 = \tau(\rho)$ and $[\pt] \star \tau(\ga)
  = q$.\smallskip

  \noin{\rm(b)} We have $\tau(\ga) \leq \wh\al$ if and only if $\wh\al \not\leq
  1$.\smallskip

  \noin{\rm(c)} We have $\wh\al \leq \tau(\rho)$ if and only if $q \not\leq
  \wh\al$.\smallskip

  \noin{\rm(d)} For $w \in W^X$ with $w \leq z_1^\vee$, we have $I(q[X^w]) =
  I([X^{w(1)}]) \cup \{q\}$.
\end{prop}
\begin{proof}
  By \eqn{qposet-minmax} we have $\tau(\rho) = [\pt]$ and $\tau(\ga) = q^{1-d}
  [X^{\ka_d}]$, where $d = d_X(2)$. Using \cite{belkale:transformation,
  chaput.manivel.ea:affine} we obtain $[\pt] \star [X^{\ka_d}] = q^d [X^{w_0^X
  \ka_d}] = q^d$. Part (a) follows from this. Part (c) follows from the
  definition of the partial order on $\cB$, and part (b) follows from (a) and
  (c), noting that quantum multiplication by the point class defines an order
  automorphism of $\cB$. Assume that $w \in W^X$ satisfies $w \leq z_1^\vee =
  w_0^X(-1)$. It follows from \Lemma{shift} that $w(1-e) = w(1)(-e)$ holds for
  all $e \in \Z$. Indeed, for $e \leq 0$ we have $w(1)(-e) = (1 \cup w(1))(-e) =
  1(-e) \cup w(1-e) = w(1-e)$, and for $e > 0$ we obtain
  \[
    w(1)(-e) = (w(1) \cap w_0^X)(-e) = w(1-e) \cap w_0^X(-e) = w(1-e) \,,
  \]
  noting that $w(1-e) \leq w_0^X(-1)(1-e) = w_0^X(-e)$. \Lemma{tau-leq} then
  implies that $q^e \tau(\al) \leq q[X^w]$ holds if and only if $q^e\tau(\al)
  \leq [X^{w(1)}]$, for any $\al \in \cP_X$ and $e \in \Z$, which proves part
  (d).
\end{proof}

\subsection{Quantum shapes}\label{sec:qshapes}%

Define a \emph{quantum shape} in $\wh\cP_X$ to be any non-empty lower order
ideal $\la \subset \wh\cP_X$. Given a quantum shape $\nu$ containing $\la$, we
write $\nu/\la$ for the \emph{skew shape} of boxes of $\nu$ that are not in
$\la$. By \Theorem{qposet}(c) we may identify $\cP_X$ with the interval
$\tau(\cP_X)$ of $\wh\cP_X$. \Proposition{qposet-elaborate} implies that $\cP_X$
is the skew shape $\cP_X = I([\pt]) / I(1)$. \Lemma{tau-leq} then shows that
\begin{equation}\label{eqn:shape2q}%
  I([X^u]) = I(u) \cup I(1)
\end{equation}
holds for all $u \in W^X$, and further shows that the shift operations on $W^X$
can be expressed as
\[
  I(u(d)) \,=\, I(q^d [X^u]) \cap \cP_X \,=\, q^d I([X^u]) \cap \cP_X
\]
for all $d \in \Z$. \Theorem{qdeg-shapes} and \Theorem{qposet}(d) show that
$q^d$ occurs in $[X_u] \star [X^v]$ if and only if the order ideal of
$q^d[X^u]$ lies between the order ideals of $[X^v]$ and $[\pt] \star [X^v]$.

When $X$ is a Grassmannian of type A, the set $\wh\cP_X$ is Postnikov's
\emph{cylinder} from \cite[\S3]{postnikov:affine}. The analogue of Postnikov's
\emph{torus} is the quotient of $\wh\cP_X$ by the group $\{[\pt]^d \mid d \in
\Z\}$ of powers of a point. Pictures of $\wh\cP_X$ for cominuscule varieties of
types other than A can be found in \Example{qposet-LG} and \Figure{qposet}.

The partially ordered sets $\wh\cP_X$ are isomorphic to certain \emph{full
heaps} of affine Dynkin diagrams that were defined in
\cite[Ch.~6]{green:combinatorics*1} based on a type-by-type construction, and
used to study minuscule representations. Postnikov's cylinder was constructed in
\cite[\S8]{hagiwara:minuscule*2} from a similar viewpoint.

\begin{example}\label{example:qposet-LG}%
  Let $X = \LG(4,8)$ be the Lagrangian Grassmannian of maximal isotropic
  subspaces in an 8-dimensional symplectic vector space, and define $u, v \in
  W^X$ by the shapes $I(u) = \tableau{4}{{}&{}&{}}$ and $I(v) =
  \tableau{4}{{}&{}&{}&{}\\&{}&{}&{}\\&&{}}$. We have
  \[
    [X_u] \star [X^v] = q^2 [X^{\tableau{4}{{}&{}&{}\\&{}&{}}}]
    + q^2 [X^{\tableau{4}{{}&{}&{}&{}\\&{}}}] + q^3
  \]
  in $\QH(X)$, so $\dmin(u^\vee,v) = 2$ and $\dmax(u^\vee,v) = 3$. The following
  picture shows a section of the partially ordered set $\wh\cP_X$, with the
  boxes of $\cP_X$ colored gray.\medskip
  \begin{center}
    \begin{tikzpicture}[x=.4cm,y=.4cm]
      \def\zz{-- ++(0,-1) -- ++(1,0)}
      \def\ZZ{-- ++(4,0) -- ++(0,-4)}
      \draw [fill=gray!30] (2,7) \zz\zz\zz\zz -- ++(0,4) -- cycle;
      \draw (0,9) \ZZ\ZZ -- ++(4,0);
      \draw (1,9) -- ++(0,-1) \ZZ -- ++(4,0) -- ++(0,-3);
      \draw (2,9) -- ++(0,-2) \ZZ -- ++(4,0) -- ++(0,-2);
      \draw (3,9) -- ++(0,-3) \ZZ -- ++(4,0) -- ++(0,-1);
      \draw (2,7) -- (6,7) -- (6,3);
      \draw (0,9) \zz\zz\zz\zz\zz\zz\zz\zz;
      \draw (4,9) -- (5,9) \zz\zz\zz\zz\zz\zz\zz -- ++(0,-1);
      \draw [line width=.7mm,black] (5,4) -- (5,5) -- (6,5) -- (6,7);
      \draw [line width=.7mm,black] (6,3) -- (8,3) -- (8,4) -- (9,4);
      \draw [line width=.7mm,red] (3,6) -- (5,6) -- (5,7) -- (6,7);
    \end{tikzpicture}
  \end{center}
  \medskip
  The south-east borders of the order ideals of $[X^v]$ and $[\pt] \star [X^v]$
  are colored black, and the south-east border of the order ideal of $[X^u]$ is
  colored red. The order ideal of $[\pt] \star [X^v]$ is obtained by reflecting
  $I(v)$ in a diagonal line and attaching the result to the right side of
  $\cP_X$; this follows from \cite[Lemma~2.9]{buch.samuel:k-theory}, by
  observing that multiplication by a point preserves the partial order on $\cB$.
  Notice that the red border will fit between the two black borders if it is
  shifted south-east by 2 or 3 steps.
\end{example}

\input{qposet}

\begin{remark}
  Let $\al_0$ be the added simple root of the affine root system corresponding
  to $G$. For any affine root $\theta = n_0 \al_0 + \sum_{\be \in \Delta} n_\be
  \be$, let $\la(\theta)=n_\ga - n_0$. Since $\ga$ is cominuscule, we have
  $\la(\theta) \in \{-1,0,1\}$. We plan to prove in a follow-up paper with
  Nicolas Ressayre that $\wh\cP_X$ is isomorphic to the partially ordered set of
  affine roots $\theta$ for which $\la(\theta)=1$, where the order on this set
  is defined by the covering relation $\theta_1 \lessdot \theta_2$ if and only
  if $\theta_2 - \theta_1$ is a positive affine root.
 \end{remark}

\subsection{Proof of the main lemma}\label{sec:qcintproof}%

\targetsec{Wcomin}{}%
Our proof of \Lemma{fiber-dtod1} utilizes a relationship between all the
cominuscule flag varieties $F = G/P_F$ of the same group $G$. Let $W^\comin
\subset W$ denote the set of representatives of single points in these
varieties, together with the identity element:
\[
  W^\comin \ = \ \{ w_0^F \mid F \text{ is a cominuscule flag variety of $G$ }\}
  \cup \{1\} \,.
\]
For each cominuscule root $\ga \in \Delta$ we let $F_\ga = G/P_\ga$ denote the
corresponding cominuscule flag variety. The following result was used to
determine the Seidel representation on the quantum cohomology ring of any flag
variety in \cite{chaput.manivel.ea:affine}, see also
\cite[Prop.~VI.2.6]{bourbaki:elements*78}.

\begin{prop}\label{prop:Wcomin}%
  The set $W^\comin$ is a subgroup of $W$ isomorphic to the coweight lattice of
  $\Phi$ modulo the coroot lattice. The isomorphism maps $w_0^{F_\ga}$ to the
  class of the fundamental coweight $\om^\vee_\ga$ corresponding to $\ga$.
\end{prop}

We mostly need this result when $G$ has Lie type A. Let $w_0^{\Gr(d,n)} \in S_n$
denote the permutation representing the point class on $\Gr(d,n)$. This
permutation is determined by $w_0^{\Gr(d,n)}(p) \equiv p-d$ (mod $n$) for $p \in
[1,n]$. The following consequence of \Proposition{Wcomin} is also immediate from
this description.

\begin{cor}\label{cor:WcominA}%
  The assignment $d \mapsto w_0^{\Gr(d,n)}$ defines an isomorphism of groups
  $\Z/n\Z \to S_n^\comin$.
\end{cor}

\begin{cor}\label{cor:zdkd}%
  For $0 \leq d \leq d_X(2)$ we have $z_d/\ka_d = (z_1/\ka_1)^d = (z_1
  s_\ga)^d$.
\end{cor}
\begin{proof}
  Since $z_d/\ka_d = w_{0,X}^{Z_d}$ represents a point in $F_d = P_X/P_{Z_d}$ by
  \Lemma{weylid}, we can prove the identity by applying \Proposition{Wcomin} to
  the Weyl group $W_X$ of $P_X$. If $X$ is a Grassmannian of type A, a
  Lagrangian Grassmannian, or a maximal orthogonal Grassmannian, then the Levi
  subgroup of $P_X$ is a group of type A (or a product of two such groups), and
  the identity follows from \Corollary{WcominA} and \Table{incidence}. If $X$ is
  a quadric hypersurface, then $F_1$ is also a quadric, $F_2$ is a point, and
  the identity follows from \Proposition{primitive}(4) because $F_1$ is
  primitive. Finally, if $X$ is the Cayley plane $E_6/P_6$ or the Freudenthal
  variety $E_7/P_7$, the identity follows from \Table{incidence} together with
  the isomorphisms $W_{D_5}^\comin \cong \Z/4\Z$ and $W_{E_6}^\comin \cong
  \Z/3\Z$.
\end{proof}

\begin{lemma}\label{lemma:z1-kd}%
  For $1 \leq d \leq d_X(2)$ we have $(z_1 s_\ga)^{d-1}.(I(z_d) \ssm I(z_{d-1}))
  = I(z_1 \cap z_{d-1}^\vee)$ and $I(\ka_d) \cup w_{0,X} z_1 s_\ga.I(z_1 \cap
  z_{d-1}^\vee) = I(z_1 \cup \ka_d)$.
\end{lemma}
\begin{proof}
  Recall the definition of the set $\cS_d$ before \Proposition{Sd}. Noting that
  $I(z_d) \ssm I(z_{d-1}) = \cS_d \cap (\cP_X \ssm I(z_{d-1}))$ and $I(z_1 \cap
  z_{d-1}^\vee) = \cS_1 \cap I(z_{d-1}^\vee)$, the first identity of the lemma
  follows from \Proposition{Sd}, \Proposition{firstbiject}(a), and
  \Corollary{zd=z1k1d}. \Proposition{z1k1} implies that
  \[
    \begin{split}
      w_{0,X} z_1 s_\ga.I(z_1 \cap z_{d-1}^\vee)
      &\ = \
      w_{0,X} z_1 s_\ga.I(z_1) \cap (z_1 s_\ga)^{-1} w_{0,X}.I(z_{d-1}^\vee) \\
      &\ = \
      I(z_1) \ssm (z_1 s_\ga)^{-1}.I(z_{d-1}) \,,
    \end{split}
  \]
  so the second
  identity is equivalent to
  \[
    I(z_1) \cap (z_1 s_\ga)^{-1}.I(z_{d-1}) \subset I(\ka_d) \,.
  \]
  Since $(z_1 s_\ga)^{-1}.I(z_{d-1}) = \bigcup_{e=2}^d \cS_e$ by
  \Proposition{Sd}, it suffices to show that $I(z_1) \cap \cS_d \subset
  I(\ka_d)$ for $d \geq 2$. This follows from the definition of $\cS_d$, as
  $I(z_1) \cap (I(z_d)\ssm I(z_{d-1})) = \emptyset$.
\end{proof}

\begin{proof}[Proof of \Lemma{fiber-dtod1}]
  The definition of $v^d$ is equivalent to $v^d \ka_d = (v \cap z_d) \cup
  \ka_d$, which specializes to $v^d \ka_d = v \cap z_d$ for $d \leq \dmax(v)$.
  Part (b) follows from this. For part (a), let $v = u^\vee$ be dual to $u$ in
  $W^X$. Then \Proposition{fibers_combin}(c) shows that $u_d$ is dual to $v^d$
  in $W^{F_d}$, and $u_{d-1}$ is dual to $v^{d-1}$ in $W^{F_{d-1}}$. Since
  \Corollary{zdkd} and \Lemma{weylid} show that $w_{0,X} (z_1 s_\ga)^d =
  w_{0,Z_d} = \ka_d w_{0,Z_d} \ka_d$, it follows from \Proposition{biject}(a)
  that
  \[
    \begin{split}
      w_{0,X}(z_1 s_\ga)^d.(I(z_d) \ssm I(v^d \ka_d))
      &= \ka_d w_{0,Z_d} \ka_d.(I(z_d) \ssm I(v^d \ka_d)) \\
      &= \ka_d w_{0,Z_d}.(\cP_{F_d} \ssm I(v^d))
      = \ka_d.I(u_d) \\
      &= I(u_d \ka_d) \ssm I(\ka_d) \,,
    \end{split}
  \]
  and part (b) implies that
  \[
    \begin{split}
      I(z_{d-1}) \ssm I(v^{d-1} \ka_{d-1})
      &= (I(z_d) \ssm I(v^d \ka_d)) \cap I(z_{d-1}) \\
      &= (I(z_d) \ssm I(v^d \ka_d)) \ssm (I(z_d) \ssm I(z_{d-1})) \,.
    \end{split}
  \]
  By combining these identities and using \Lemma{z1-kd} and \Lemma{u(d)}, we
  obtain
  \[
    \begin{split}
      I(u_{d-1} \ka_{d-1}) \ssm I(\ka_{d-1})
      &= w_{0,X} (z_1 s_\ga)^{d-1}.(I(z_{d-1}) \ssm I(v^{d-1} \ka_{d-1})) \\
      &= z_1 s_\ga w_{0,X} (z_1 s_\ga)^{d}.(
        (I(z_d) \ssm I(v^d \ka_d)) \ssm (I(z_d) \ssm I(z_{d-1})) ) \\
      &= z_1 s_\ga.\left( (I(u_d \ka_d) \ssm I(\ka_d)) \ssm
        w_{0,X} z_1 s_\ga.I(z_1 \cap z_{d-1}^\vee) \right) \\
      &= z_1 s_\ga.\left( I(u_d \ka_d) \ssm I(z_1 \cup \ka_d) \right) \\
      &= z_1 s_\ga.\left( I(u_d \ka_d) \ssm I(z_1) \right)
        \ssm z_1 s_\ga.\left( I(\ka_d) \ssm I(z_1) \right) \\
      &= I((u_d \ka_d)(-1)) \ssm I(\ka_{d-1}) \,.
    \end{split}
  \]
  This proves part (a).
\end{proof}

\begin{example}\label{example:rot2c}%
  Let $X = \LG(8,16)$, $u = (8,6,2) \in W^X$, and set $d = 5$. Then $u_d =
  (5,4,1)$ is obtained by intersecting $I(u)$ with the rectangle
  $z_d.\cP_{F_d}$, see \Example{dissectPXc} and \Example{fiber-diagrams}. The
  following pictures illustrate how the skew shape $I(u_{d-1} \ka_{d-1}) \ssm
  I(\ka_{d-1})$ is obtained from $I(u_d \ka_d) \ssm I(\ka_d)$. Notice that,
  since the bijection $z_d.\cP_{F_d} \cong \ka_d.\cP_{F_d}$ is given by a
  transposition, the partition of $u_d$ is conjugated when we form $I(u_d \ka_d)
  \ssm I(\ka_d)$.

  \begin{center}
  \begin{tikzpicture}[x=.4cm,y=.4cm]
    \def\zz{-- ++(0,1) -- ++(-1,0)}
    \tikzset{dot/.pic={\node at (.5,-.5) {$\bullet$};}}
    \begin{scope}[shift={(0,12)}]
      \node at (5,1) {$I(u_d \ka_d) \ssm I(\ka_d)$};
      \draw (0,0) -- (8,0) -- (8,-8) -- ++(-1,0) \zz\zz\zz\zz\zz\zz\zz -- cycle;
      \draw[very thick] (5,0) -- (8,0) -- (8,-5) -- (5,-5) -- cycle;
      \draw[pattern=north west lines] (5,0) -- (8,0) -- ++(0,-1) -- ++(-1,0)
      -- ++(0,-3) -- ++(-1,0) -- ++(0,-1) -- ++(-1,0) -- cycle;
      \pic foreach \py in {-3,...,0} foreach \px in {5,6} at (\px,\py) {dot};
      \pic at (7,0) {dot};
      \pic at (5,-4) {dot};
      \draw [very thick, ->] (9.5,-4) -- (12.5,-4);
    \end{scope}
    \begin{scope}[shift={(14,12)}]
      \node at (5,1) {$I(z_d) \ssm I(v^d \ka_d)$};
      \draw (0,0) -- (8,0) -- (8,-8) -- ++(-1,0) \zz\zz\zz\zz\zz\zz\zz -- cycle;
      \draw[very thick] (5,0) -- (8,0) -- (8,-5) -- (5,-5) -- cycle;
      \draw[pattern=north west lines] (8,0) -- (7,0) -- ++(0,-1) -- ++(-1,0)
      -- ++(0,-3) -- ++(-1,0) -- ++(0,-1) -- ++(3,0) -- cycle;
      \pic foreach \py in {-4,...,-1} foreach \px in {6,7} at (\px,\py) {dot};
      \pic at (5,-4) {dot};
      \pic at (7,0) {dot};
      \draw [very thick, ->] (4,-7) -- (4,-9.5);
    \end{scope}
    \begin{scope}[shift={(14,0)}]
      \node at (4,1) {$I(z_{d-1}) \ssm I(v^{d-1} \ka_{d-1})$};
      \draw (0,0) -- (8,0) -- (8,-8) -- ++(-1,0) \zz\zz\zz\zz\zz\zz\zz -- cycle;
      \draw [very thick] (4,0) -- (8,0) -- (8,-4) -- (4,-4) -- cycle;
      \draw[pattern=north west lines] (8,0) -- (7,0) -- ++(0,-1) -- ++(-1,0)
      -- ++(0,-3) -- ++(2,0) -- cycle;
      \pic foreach \py in {-4,...,-1} foreach \px in {6,7} at (\px,\py) {dot};
      \pic at (5,-4) {dot};
      \pic at (7,0) {dot};
      \draw [very thick, ->] (-1.5,-4) -- (-4.5,-4);
    \end{scope}
    \begin{scope}[shift={(0,0)}]
      \node at (4,1.5) {$I(u_{d-1} \ka_{d-1}) \ssm I(\ka_{d-1})$};
      \draw (0,0) -- (8,0) -- (8,-8) -- ++(-1,0) \zz\zz\zz\zz\zz\zz\zz -- cycle;
      \draw [very thick] (4,0) -- (8,0) -- (8,-4) -- (4,-4) -- cycle;
      \draw[pattern=north west lines] (4,0) -- (6,0) -- ++(0,-3) -- ++(-1,0)
      -- ++(0,-1) -- ++(-1,0) -- cycle;
      \pic foreach \py in {-2,...,1} foreach \px in {4,5} at (\px,\py) {dot};
      \pic at (6,1) {dot};
      \pic at (4,-3) {dot};
    \end{scope}
  \end{tikzpicture}
  \end{center}
\end{example}

\begin{example}\label{example:rot2a}%
  Let $X = \Gr(7,17)$, $u = (10,8,5,5,4,1,0) \in W^X$, and $d=4$. Then $F_d =
  \Gr(3,7) \times \Gr(4,10)$, so elements of $W^{F_d}$ can be represented by
  pairs of partitions. We find $u_d = ((4,2,0),(5,4,1,0))$ by intersecting
  $I(u)$ with $z_d.\cP_{F_d}$, see \Example{dissectPXa}. The skew shape
  $I(u_{d-1} \ka_{d-1}) \ssm I(\ka_{d-1})$ is obtained from $I(u_d \ka_d) \ssm
  I(\ka_d)$ with the following steps.

  \begin{center}
  \begin{tikzpicture}[x=.4cm,y=.4cm]
    \tikzset{dot/.pic={\node at (.5,-.5) {$\bullet$};}}
    \begin{scope}[shift={(0,11)}]
      \draw (0,0) -- (0,7) -- (10,7) -- (10,0) -- cycle;
      \draw [very thick] (4,3) -- (10,3) -- (10,7) -- (4,7) -- cycle;
      \draw [very thick] (0,0) -- (0,3) -- (4,3) -- (4,0) -- cycle;
      \draw [pattern=north west lines] (0,1) -- (2,1) -- (2,2)
      -- (4,2) -- (4,3) -- (0,3) -- cycle;
      \draw [pattern=north west lines]  (5,4) -- (5,5) -- (8,5) -- (8,6) -- (9,6) -- (9,7)
      -- (4,7) -- (4,4) -- cycle;
      \pic at (3,3) {dot};
      \pic foreach \py in {6,7} foreach \px in {4,5,6,7} at (\px,\py) {dot};
      \pic at (4,5) {dot};
      \pic at (8,7) {dot};
      \pic foreach \px in {0,1,2} at (\px,3) {dot};
      \pic at (0,2) {dot};
      \pic at (1,2) {dot};
      \draw [very thick, ->] (11.5,3.5) -- (14.5,3.5);
    \end{scope}
    \begin{scope}[shift={(16,11)}]
      \draw (0,0) -- (0,7) -- (10,7) -- (10,0) -- cycle;
      \draw [very thick] (4,3) -- (10,3) -- (10,7) -- (4,7) -- cycle;
      \draw [very thick] (0,0) -- (0,3) -- (4,3) -- (4,0) -- cycle;
      \draw [pattern=north west lines] (0,1) -- (2,1) -- (2,2) -- (4,2) -- (4,0)
      -- (0,0) -- cycle;
      \draw [pattern=north west lines] (5,3) -- (5,4) -- (6,4) -- (6,5) -- (9,5) -- (9,6) -- (10,6) -- (10,3)
      -- cycle;
      \pic at (0,1) {dot};
      \pic foreach \py in {4,5} foreach \px in {6,7,8,9} at (\px,\py) {dot};
      \pic at (5,4) {dot};
      \pic at (9,6) {dot};
      \pic foreach \px in {1,2,3} at (\px,1) {dot};
      \pic at (3,2) {dot};
      \pic at (2,2) {dot};
      \draw [very thick, ->] (5,-1) -- (5,-3);
    \end{scope}
    \begin{scope}[shift={(16,0)}]
      \draw (0,0) -- (0,7) -- (10,7) -- (10,0) -- cycle;
      \draw [very thick] (3,4) -- (10,4) -- (10,7) -- (3,7) -- cycle;
      \draw [very thick] (0,0) -- (0,4) -- (3,4) -- (3,0) -- cycle;
      \draw [pattern=north west lines] (0,1) -- (2,1) -- (2,2) -- (3,2) -- (3,0)
      -- (0,0) -- cycle;
      \draw [pattern=north west lines] (6,4) -- (6,5) -- (9,5) -- (9,6) -- (10,6) -- (10,4) -- cycle;
      \pic at (0,1) {dot};
      \pic foreach \py in {4,5} foreach \px in {6,7,8,9} at (\px,\py) {dot};
      \pic at (5,4) {dot};
      \pic at (9,6) {dot};
      \pic foreach \px in {1,2,3} at (\px,1) {dot};
      \pic at (3,2) {dot};
      \pic at (2,2) {dot};
      \draw [very thick, ->] (-1.5,3.5) -- (-4.5,3.5);
    \end{scope}
    \begin{scope}[shift={(0,0)}]
      \draw (0,0) -- (0,7) -- (10,7) -- (10,0) -- cycle;
      \draw [very thick] (3,4) -- (10,4) -- (10,7) -- (3,7) -- cycle;
      \draw [very thick] (0,0) -- (0,4) -- (3,4) -- (3,0) -- cycle;
      \draw [pattern=north west lines] (0,2) -- (1,2) -- (1,3) -- (3,3)
      -- (3,4) -- (0,4) -- cycle;
      \draw [pattern=north west lines] (3,5) -- (4,5) -- (4,6) -- (7,6) -- (7,7) -- (3,7) -- cycle;
      \pic at (2,4) {dot};
      \pic foreach \py in {7,8} foreach \px in {3,4,5,6} at (\px,\py) {dot};
      \pic at (3,6) {dot};
      \pic at (7,8) {dot};
      \pic foreach \px in {-1,0,1} at (\px,4) {dot};
      \pic at (-1,3) {dot};
      \pic at (0,3) {dot};
    \end{scope}
  \end{tikzpicture}
  \end{center}
\end{example}

%% file: qposet.tex

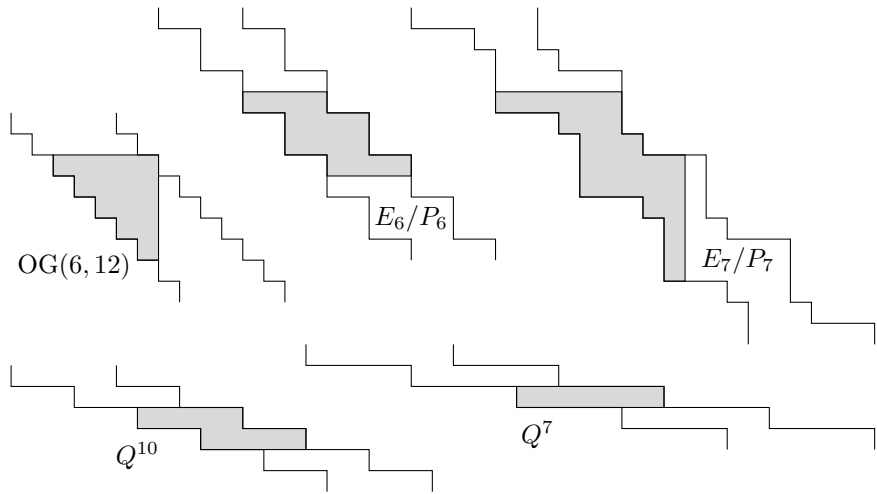
\begin{figure}[h]
    \caption{The partially ordered sets $\wh\cP_X$ for a collection of
      cominuscule flag varieties, with the boxes of $\cP_X$ colored gray.}
    \label{fig:qposet}
    \bigskip
    \begin{tikzpicture}[x=2.8mm,y=2.8mm]
    \begin{scope}[shift={(0,-3)}] 
        \def\zz{-- ++(0,-1) -- ++(1,0)}
        \draw [fill=gray!30] (0,0) \zz\zz\zz\zz\zz -- (5,0) -- cycle;
        \draw (-2,2) \zz\zz\zz\zz\zz\zz\zz\zz -- ++(0,-1);
        \draw (3,2) \zz\zz\zz\zz\zz\zz\zz\zz -- ++(0,-1);
        \node at (1,-5.3) {$\OG(6,12)$};
    \end{scope}
    \begin{scope}[shift={(9,0)}] 
        \def\zz{-- ++(0,-2) -- ++(2,0)}
        \draw [fill=gray!30] (0,0) -- (4,0) -- ++(0,-1) -- ++(2,0) \zz
        -- ++(0,-1) -- ++(-4,0) -- ++(0,1) -- ++(-2,0) -- ++(0,2)
        -- ++(-2,0) -- cycle;
        \draw (0,4) -- ++(0,-1) -- ++(2,0) \zz\zz\zz\zz\zz -- ++(0,-1);
        \draw (-4,4) -- ++(0,-1) -- ++(2,0) \zz\zz\zz\zz\zz -- ++(0,-1);
        \node at (8,-6) {$E_6/P_6$};
    \end{scope}
    \begin{scope}[shift={(21,0)}] 
        \def\zz{-- ++(0,-3) -- ++(1,0) -- ++(0,-1) -- ++(3,0)}
        \def\uu{-- ++(0,3) -- ++(-1,0) -- ++(0,1) -- ++(-3,0)}
        \draw [fill=gray!30] (0,0) -- ++(6,0) -- ++(0,-2) -- ++(1,0) -- ++(0,-1)
        -- ++(2,0) -- ++(0,-6) -- ++(-1,0) \uu \uu -- cycle;
        \draw (2,4) -- ++(0,-2) -- ++(1,0) -- ++(0,-1) -- ++(3,0)\zz\zz\zz
        -- ++(0,-1);
        \draw (12,-12) -- ++(0,2) -- ++(-1,0) -- ++(0,1) -- ++(-3,0) \uu\uu\uu
        -- ++(0,1);
        \node at (11.5,-8) {$E_7/P_7$};
    \end{scope}
    \begin{scope}[shift={(4,-15)}] 
        \def\zz{-- ++(0,-1) -- ++(3,0)}
        \draw [fill=gray!30] (0,0) -- (5,0) \zz -- ++(0,-1) -- ++(-5,0)
        -- ++(0,1) -- ++ (-3,0) -- cycle;
        \draw (-1,2) \zz\zz\zz\zz\zz -- ++(0,-1);
        \draw (-6,2) \zz\zz\zz\zz\zz -- ++(0,-1);
        \node at (0,-2.3) {$Q^{10}$};
    \end{scope}
    \begin{scope}[shift={(22,-14)}] 
        \def\zz{-- ++(0,-1) -- ++(5,0)}
        \draw [fill=gray!30] (0,0) -- (7,0) -- (7,-1) -- (0,-1) -- cycle;
        \draw (-3,2) \zz\zz\zz\zz -- ++(0,-1);
        \draw (-10,2) \zz\zz\zz\zz -- ++(0,-1);
        \node at (1,-2.3) {$Q^7$};
    \end{scope}
    \end{tikzpicture}
  \end{figure}

%% file: quantumk.tex

\section{Results about quantum $K$-theory}\label{sec:quantumk}

\subsection{The small quantum $K$-theory ring}\label{sec:qkring}%

\targetsec{qkring}{}%
Let $X = G/P_X$ be a cominuscule flag variety. The (small) quantum $K$-theory
ring $\QK(X)$ of Givental and Lee \cite{givental:wdvv, lee:quantum} is an
algebra over the ring $\Z\llbracket q \rrbracket$ of formal power series in a
single variable $q$ called the \emph{deformation parameter}. As a $\Z\llbracket
q \rrbracket$-module we have $\QK(X) = K(X) \otimes \Z\llbracket q\rrbracket$.
The associative product $\star$ of $\QK(X)$ is defined in terms of $K$-theoretic
Gromov-Witten invariants. We recall a construction of this product from
\cite{buch.chaput.ea:chevalley}.

\targetsec{qkprod}{}%
Let $\Psi : K(X) \to K(X)$ be the linear map defined by $\Psi(\cO^w) =
\cO^{w(-1)}$. This map sends the class of any Schubert variety $X^w$ to the
class of its line neighborhood $\Gamma_1(X^w)$ by \Corollary{fibers_geom}(b), so
it can also be defined by $\Psi = (\ev_2)_* (\ev_1)^*$, where $\ev_1$ and
$\ev_2$ are the evaluation maps from $\Mb_{0,2}(X,1)$. \Corollary{qc_nbhd}
implies that $\Psi = (p_1)_* (q_1)^* (q_1)_* (p_1)^*$. Given $u, v \in W^X$ and
$d \geq 1$, we define the class
\begin{equation}\label{eqn:qk-prod-d}%
  (\cO_u \star \cO^v)_d \ = \ [\cO_{\Gamma_d(X_u,X^v)}] -
  \Psi([\cO_{\Gamma_{d-1}(X_u,X^v)}])
\end{equation}
in $K(X)$. Let $(\cO_u \star \cO^v)_0 = \cO_u \cdot \cO^v$ be the product in the
$K$-theory ring. It then follows from \cite[Prop.~3.2]{buch.chaput.ea:chevalley}
that Givental's product in $\QK(X)$ is given by
\begin{equation}\label{eqn:qk-prod}%
  \cO_u \star \cO^v \ = \ \sum_{d \geq 0} (\cO_u \star \cO^v)_d\, q^d \,.
\end{equation}
The proof in \cite{buch.chaput.ea:chevalley} showing that \eqn{qk-prod} agrees
with Givental's definition relies on a version of the quantum-to-classical
principle for large degrees that was established in \cite{buch.mihalcea:quantum,
chaput.perrin:rationality, buch.chaput.ea:projected}.

The definition \eqn{qk-prod-d} implies that $(\cO_u \star \cO^v)_d = 0$ for all
sufficiently large degrees $d$, since eventually we have
$\Gamma_{d-1}(X_u,X^v)=X$. As a consequence, the product $\cO_u \star \cO^v$
contains only finitely many non-zero terms. A similar finiteness result is known
for the quantum $K$-theory of arbitrary flag varieties
\cite{buch.chaput.ea:finiteness, buch.chaput.ea:rational, kato:quantum,
anderson.chen.ea:finiteness}. In the cominuscule case we have $(\cO_u \star
\cO^v)_d = 0$ whenever $d > d_X(2)$ by \cite[Thm.~1]{buch.chaput.ea:finiteness}.
Using this explicit bound, we can focus on the terms $(\cO_u \star \cO^v)_d$ of
small degrees, which will be studied using the tools developed in the previous
sections.

\targetsec{qkconst}{}%
The Schubert structure constants of $\QK(X)$ are the integers $N_{u,v}^{w,d}$
defined by
\begin{equation}\label{eqn:qk-const}%
  \cO^u \star \cO^v \ = \ \sum_{w,d\geq 0} N_{u,v}^{w,d}\, q^d\, \cO^w \,.
\end{equation}
Equivalently, we have $(\cO_u \star \cO^v)_d = \sum_w N^{w,d}_{u^\vee,v} \cO^w$
for each degree $d$. These structure constants are expected to have alternating
signs in the following sense \cite{lenart.maeno:quantum,
buch.mihalcea:quantum}.

\begin{conj}\label{conj:qkpos}%
  We have $(-1)^{\ell(uvw) + \deg(q^d)} N_{u,v}^{w,d} \geq 0$.
\end{conj}

Here $\deg(q^d) = d \int_{X_{s_\ga}} c_1(T_X) = d\,(\ell(z_1) + 1)$ denotes the
degree of $q^d$ in the quantum cohomology ring $\QH(X)$. This conjecture
generalizes the fact that the structure constants $N_{u,v}^{w,0}$ of the
ordinary $K$-theory ring of $X$ have alternating signs
\cite{buch:littlewood-richardson*1, brion:positivity}. A more general version of
\Conjecture{qkpos} for the equivariant quantum $K$-theory of arbitrary flag
varieties is discussed in \cite[\S2.4]{buch.chaput.ea:chevalley}.

Recall from \Corollary{fibers_geom}(d) that the general fibers of the map $p_d :
Z_d(X_u,X^v) \to \Gamma_d(X_u,X^v)$ are translates of $(F_d)^{u_d\cap
v^d}_{u_d}$. The quotient $(u_d \cup v^d)/v^d$ of Weyl group elements (see
\Sec{comin-schubert}) will be called a \emph{short rook strip} if it is a
product of commuting reflections associated to short simple roots of $\Phi$.
Equivalently, the skew shape $I(u_d \cup v^d) \ssm I(v^d)$ in $\cP_{F_d}$
consists of pairwise incomparable short boxes. We emphasize that the lengths of
the simple roots should be measured relative to the root system $\Phi$ of $G$,
as opposed to the root system $\Phi_X$ of the group $P_X$ acting on $F_d$. For
example, when $X = \LG(n,2n)$ is a Lagrangian Grassmannian, $\Phi_X$ is a root
system of type A, but its roots are considered short since they are short in
$\Phi$.

\begin{defn}\label{defn:exceptional}%
  Let $u,v \in W^X$. An integer $d$ is an \emph{exceptional degree} of the
  product $\cO_u \star \cO^v$ if $d \leq \min(\dmax(u^\vee), \dmax(v))$ and
  $(u_d \cup v^d)/v^d$ is a non-empty short rook strip.
\end{defn}

Notice that exceptional degrees do not occur when $X$ is minuscule, as all roots
of a simply laced root system are considered long. In addition, most products
$\cO_u \star \cO^v$ on odd quadrics $Q^{2n-1}$ and Lagrangian Grassmannians
$\LG(n,2n)$ have no exceptional degrees, see \Example{quadric-exceptional} and
\Table{exceptional-counts}. Notice also that, if $d$ is an exceptional degree of
$\cO_u \star \cO^v$, then we must have $d = \dmax(u^\vee,v) + 1$ by
\Lemma{ud<vd}(b). We proceed to state our main results about quantum $K$-theory
of cominuscule flag varieties.

\begin{thm}\label{thm:qkdegrees}%
  Let $u, v \in W^X$. The quantum $K$-theory product $\cO_u \star \cO^v$ in
  $\QK(X)$ contains the same powers of $q$ as the quantum cohomology product
  $[X_u] \star [X^v]$ in $\QH(X)$, with the exception that $q^d$ may also occur
  in $\cO_u \star \cO^v$ if $d = \dmax(u^\vee,v)+1$ is an exceptional degree.
  In particular, the powers $q^d$ occurring in $\cO_u \star \cO^v$ form an
  integer interval.
\end{thm}

We conjecture that $(\cO_u \star \cO^v)_d \neq 0$ whenever $d$ is an exceptional
degree. This is true for quadrics and has been verified for Lagrangian
Grassmannians $\LG(n,2n)$ with $n \leq 6$. See \Conjecture{exceptional} for a
more detailed statement.

\begin{thm}\label{thm:qkpos}%
  \Conjecture{qkpos} is true whenever $X$ is minuscule or any quadric
  hypersurface. It is also true whenever $d$ is not an exceptional degree of the
  product $\cO_{u^\vee} \star \cO^v$.
\end{thm}

\begin{table}[h]
  \caption{The number of products $\cO_u \star \cO^v$ with exceptional degrees
  on Lagrangian Grassmannians $\LG(n,2n)$.}
  \label{tab:exceptional-counts}
  \begin{tabular}{|c|c|c|}
    \hline
    $n$ & Total products & Exceptional degrees \\
    \hline
    2  & 10 & 1 (10\%) \\
    3  & 36 & 3 (8.3\%) \\
    4  & 136 & 17 (12.5\%) \\
    5  & 528 & 70 (13.3\%) \\
    6  & 2080 & 313 (15.0\%) \\
    7  & 8256 & 1317 (16.0\%) \\
    8  & 32896 & 5590 (17.0\%) \\
    9  & 131328 & 23310 (17.7\%) \\
    10 & 524800 & 96932 (18.5\%) \\
    \hline
  \end{tabular}
\end{table}

\subsection{Dual structure constants}\label{sec:qkdual}%

Before discussing our proofs of \Theorem{qkdegrees} and \Theorem{qkpos}, we
apply these results to obtain a similar positivity statement for the structure
constants of $\QK(X)$ with respect to its dual basis $\{\cI_q^v\}$. We also
apply our results from \Section{qhinterval} to give a new proof of a formula
from \cite{summers:dual} for the dual basis elements of any cominuscule flag
variety.

\targetsec{qktring}{}%
For any flag
variety $M = G/P_M$ we let $\QK_T(M)$ denote the $T$-equivariant quantum
$K$-theory ring. This ring is an algebra over $\cK\llbracket q \rrbracket$,
where $\cK = K_T(\pt)$ and $q$ is a set of variables indexed by the Schubert
curves in $M$. We have $\QK_T(M) = K_T(M) \otimes_\cK \cK\llbracket q\rrbracket$
as a $\cK\llbracket q\rrbracket$-module; we refer to e.g.\
\cite[\S2]{buch.chaput.ea:chevalley} or \cite[\S3]{buch.chung.ea:euler} for
details, including the definition of the product $\star$ of $\QK_T(M)$.

\targetsec{metric}{}%
Givental's \emph{quantum $K$-metric} is the $\cK\llbracket q
\rrbracket$-bilinear paring on $\QK_T(M)$ given by
\[
  (\!(\cF_1, \cF_2)\!) = \sum_{d \geq 0} I_d(\cF_1,\cF_2)\, q^d
\]
for $\cF_1, \cF_2 \in K_T(M)$, where $I_d(\cF_1,\cF_2) \in \cK$ denotes the
$T$-equivariant Gromov-Witten invariant. For $v \in W^M$ we let $\cI_q^v \in
\QK_T(M)$ denote class dual to $\cO_{M_v}$ under this pairing, that is, we have
$(\!(\cO_{M_u},\cI_q^v)\!) = \delta_{u,v}$ for all $u,v \in W^M$. Recall the
partial order $\leq$ on the set $\cB = \{ q^d [M^w] : w \in W^M, d \in H_2(M,\Z)
\}$ from \Section{qh-min-max}.

\begin{thm}\label{thm:schub2dual}%
  Let $v \in W^M$. The expansion of $\cO_{M^v}$ in the dual basis of $\QK_T(M)$
  is given by
  \[
    \cO_{M^v} = \sum_{q^d[M^w] \,\geq\, [M^v]} q^d\, \cI_q^w \,.
  \]
\end{thm}
\begin{proof}
  It suffices to prove the equality
  \[
    \qpair{\cO_{M^v},\cO_{M_u}} \,=
    \sum_{q^d[M^w] \,\geq\, [M^v]} q^d\, \qpair{\cI_q^w,\cO_{M_u}} \,.
  \]
  It follows from \cite[Cor.~3.3]{buch.chaput.ea:finiteness} that the left side
  is the sum of $q^d$ for all degrees $d \in H_2(M,\Z)$ for which
  $\Gamma_d(M_u,M^v) \neq \emptyset$, and the right side is the sum of $q^d$ for
  all $d$ satisfying $q^d[M^u] \geq [M^v]$. The result follows because these
  conditions on $d$ are equivalent.
\end{proof}

\targetsec{dual2schub}{}%
It follows from \Theorem{schub2dual} that each dual Schubert class $\cI_q^w$ can
be written as a (possibly infinite) linear combination
\begin{equation}\label{eqn:dual2schub}%
  \cI_q^w \,=\, \sum_{q^d[M^u] \,\geq\, [M^w]} C_w^{u,d}\, q^d\, \cO_{M^u}
\end{equation}
with integer coefficients $C_w^{u,d} \in \Z$. Set $\deg(q^d) = \int_d c_1(T_M)$.
The following conjecture follows from the main result of \cite{summers:dual}
when $M$ is cominuscule.

\begin{conj}
  The dual Schubert class $\cI_q^w$ is a finite alternating $\Z$-linear
  combination of the $\Z$-basis $\{q^d \cO_{M^u}\}$, that is, the sum
  \eqn{dual2schub} contains finitely many non-zero terms, and we have
  $(-1)^{\ell(w u)+\deg(q^d)} C_w^{u,d} \geq 0$.
\end{conj}

\targetsec{dualconst}{}%
Define the \emph{dual structure constants} $A_{u,v}^{w,d} \in \cK$
of $\QK_T(M)$ by
\begin{equation}\label{eqn:dualconst}%
  \cI_q^u \star \cI_q^v \,=\, \sum_{w,d} A_{u,v}^{w,d}\, q^d\, \cI_q^w \,.
\end{equation}
Based on computations with cominuscule flag varieties, we conjecture that the
dual structure constants satisfy both finiteness and positivity in the following
sense.

\begin{conj}\label{conj:qkdual}%
  Let $u, v, w \in W^M$. The constant $A_{u,v}^{w,d}$ is non-zero for only
  finitely many degrees $d \in H_2(M,\Z)$. In addition, we have
  \[
    (-1)^{\ell(uvw)+\deg(q^d)} A_{u,v}^{w,d} \,\in\,
    \N\big[\, [\C_{-\be}]-1 : \be \in \Delta \,\big] \,.
  \]
\end{conj}

The positivity statement for $d=0$ is equivalent to
\cite[Cor.~5.2]{anderson.griffeth.ea:positivity}. We will show that when $M$ is
cominuscule, the finiteness statement follows from the finiteness of quantum
multiplication in the structure sheaf basis \cite{buch.chaput.ea:finiteness}
together with the formula for $\cI_q^v$ proved in \cite{summers:dual}. We will
also prove the non-equivariant case of the positivity statement when $M$ is
minuscule or an odd quadric.

\targetsec{qposet-basis}{}%
Assume that $M=X$ is cominuscule. In this case the equivariant quantum
$K$-theory ring $\QK_T(X)$ is defined by \eqn{qk-prod-d} and \eqn{qk-prod},
except that all structure sheaves are endowed with their natural $T$-equivariant
structure, see \cite[\S3.3]{buch.chaput.ea:chevalley}. Given any quantum shape
$\la \subset \wh\cP_X$, we define $\cO^\la = q^d \cO^w$ and $\cI_q^\la = q^d
\cI_q^w$, where $w \in W^X$ and $d \in \Z$ are determined by $I(q^d[X^w]) =
\la$, see \Theorem{qposet}(d).

\targetsec{jxi}{}%
Let $J = 1 - \cO^{s_\ga} \in K_T(X)$ be the ideal sheaf of the Schubert divisor
$X^{s_\ga}$. We define three $\cK\llbracket q\rrbracket$-linear endomorphisms
$\cJ$, $\zeta$, $\theta$ of $\QK_T(X)$ as follows. For $u \in W^X$ we set
$\cJ(\cO^u) = J_u \cO^u$, where $J_u \in \cK$ denotes the restriction of $J$ to
the point $u.P_X \in X^T$. The map $\zeta$ is defined by
\[
  \zeta(\cO^u) \,= \sum_{w/u \text{ short rook strip}}
  (-1)^{\ell(w/u)}\, [\C_{-\delta(w/u)}]\, \cO^w \,,
\]
where the sum is over all $w \in W^X$ such that $u \leq w$ and $w/u$ is a short
rook strip, and $\delta(w/u)$ denotes the sum of the (short) simple roots $\be
\in \Delta$ for which $s_\be$ appears in a reduced expression of $w/u$. This map
$\zeta$ was denoted by $\phi$ in \cite[\S3.4]{buch.chaput.ea:chevalley}; it is
the identity map when $X$ is minuscule. Finally, for any quantum shape $\la
\subset \wh\cP_X$ we set
\[
  \theta(\cO^\la) \,= \sum_{\nu/\la \text{ rook strip}}
  (-1)^{|\nu/\la|}\, \cO^\nu
\]
where the sum is over all quantum shapes $\nu \subset \wh\cP_X$ containing
$\la$, such that $\nu/\la$ is a rook strip in $\wh\cP_X$, that is, any two
distinct boxes of $\nu/\la$ are incomparable. This map is $\cK\llbracket
q\rrbracket$-linear because multiplication by $q$ defines an order automorphism
of $\wh\cP_X$. The Chevalley formula from \cite{buch.chaput.ea:chevalley} can be
stated as follows.

\begin{thm}\label{thm:qktchev}%
  For any class $\cF \in \QK_T(X)$ we have $J \star \cF =
  \theta(\zeta(\cJ(\cF)))$.
\end{thm}
\begin{proof}
  We may assume that $\cF = \cO^u$ is the Schubert class defined by $u \in W^X$.
  Let $\la = I([X^u])$ be the associated quantum shape, and let $\wh\cR(\la)$ be
  the set of quantum shapes $\nu \subset \wh\cP_X$ containing $\la$ such that
  $\nu/\la$ is a rook strip. We partition this set into the subsets
  $\wh\cR_0(\la) = \{\nu \in \wh\cR(\la) \mid q \notin \nu\}$ and $\wh\cR_1(\la)
  = \{\nu \in \wh\cR(\la) \mid q \in \nu\}$. The first subset $\wh\cR_0(\la)$ is
  in bijective correspondence with the set
  \[
    \cR_0(u) = \{ w \in W^X \mid
    u \leq w \text{ and } w/u \text{ is a rook strip} \} \,,
  \]
  where the bijection $\cR_0(u)
  \to \wh\cR_0(\la)$ is given by $w \mapsto I([X^w])$; this follows from
  \eqn{shape2q} since $w/u$ is a rook strip if and only if $I(w)/I(u)$ is a rook
  strip in $\cP_X$. \Proposition{qposet-elaborate}(d) shows that $I(q) =
  I([X^{z_1}]) \cup \{q\}$. It follows that $z_1 \leq u$ holds if and only if
  $q$ is a minimal box of $\wh\cP_X \ssm \la$, which in turn is equivalent to
  $\wh\cR_1(\la) \neq \emptyset$. In this case, consider the set
  \[
    \cR_1(u) = \{ w \in W^X \mid u(-1) \leq w \leq z_1^\vee \text{ and }
    w/u(-1) \text{ is a rook strip} \} \,.
  \]
  We claim that the map $w \mapsto I(q[X^w])$ is a bijection $\cR_1(u) \to
  \wh\cR_1(\la)$. Indeed, \Proposition{z1k1}(a) shows that $w \mapsto w(1)$ is a
  bijection $\cR_1(u) \to \cR_0(u)$, and $\nu \mapsto \nu\cup\{q\}$ is a
  bijection $\wh\cR_0(\la) \to \wh\cR_1(\la)$, so the claim follows from
  \Proposition{qposet-elaborate}(d). Using the notation of
  \cite[\S3]{buch.chaput.ea:chevalley} we obtain $\theta(\cO^u) =
  \theta_0(\cO^u) + q\,\theta_1(\cO^u)$, after which the claimed identity is
  equivalent to \cite[Thm.~3.9]{buch.chaput.ea:chevalley}.
\end{proof}

\begin{remark}\label{remark:zeta-qshapes}%
  By \Theorem{qposet}(c) we may identify $\cP_X$ with a subset of $\wh\cP_X$, in
  which case any box in $\wh\cP_X$ has the form $q^d$ or $q^d\al$, with $\al \in
  \cP_X$ and $d \in \Z$. If we define $q^d$ to be long, define $q^d\al$ to have
  the same length as $\al$, and set $\delta(q^d \al) = \delta(\al)$, then the
  map $\zeta$ can be defined by $\zeta(\cO^\la) = \sum_\nu (-1)^{|\nu/\la|}
  [\C_{-\delta(\nu/\la)}] \cO^\nu$, where the sum is over all quantum shapes
  $\nu$ containing $\la$ such that $\nu/\la$ is a short rook strip, and
  $\delta(\nu/\la) = \sum_{\wh\al \in \nu/\la} \delta(\wh\al)$.
\end{remark}

The first equality of the following result is equivalent to the main result of
\cite{summers:dual} when $\la \subset \cP_X$ is a classical shape. We give a new
proof based on quantum shapes.

\begin{cor}\label{cor:comin-dual}%
  The dual Schubert class $\cI_q^\la \in \QK_T(X)$ is given by
  \[
    \cI_q^\la \,=\, J \star \cJ^{-1} \zeta^{-1}(\cO^\la) \,=\, \theta(\cO^\la) \,.
  \]
\end{cor}
\begin{proof}
  The second equality $J \star \cJ^{-1}\zeta^{-1}(\cO^\la) = \theta(\cO^\la)$
  follows from \Theorem{qktchev}. \Theorem{schub2dual} can be restated as
  $\cO^\la = \sum_{\nu \supset \la} \cI_q^\nu$, where the sum is over all
  quantum shapes $\nu$ containing $\la$. Using this to expand $\theta(\cO^\la)$
  in the dual basis $\{\cI_q^\nu\}$, the coefficient of $\cI_q^\nu$ is zero if
  $\nu$ does not contain $\la$, and otherwise is the sum of $(-1)^{|S|}$ over
  all subsets $S$ of the minimal boxes of $\nu/\la$. This proves
  $\theta(\cO^\la) = \cI_q^\la$.
\end{proof}

By substituting the first identity of \Corollary{comin-dual} into
\eqn{dualconst}, dividing by $J$, and applying $\zeta\cJ$ to both sides, we
obtain the following expression for the dual structure constants of $\QK_T(X)$:
\begin{equation}\label{eqn:dualcomp}%
  \sum_{w,d} A_{u,v}^{w,d}\, q^d\, \cO^w \,=\,
  \zeta\cJ(J \star \cJ^{-1}\zeta^{-1}(\cO^u) \star \cJ^{-1}\zeta^{-1}(\cO^v))
  \,.
\end{equation}
The finiteness part of \Conjecture{qkdual} for cominuscule flag varieties
follows from this identity together with the main result of
\cite{buch.chaput.ea:finiteness}. We have also used \eqn{dualcomp} to verify the
positivity claim of \Conjecture{qkdual} for $\Gr(2,8)$, $\Gr(3,7)$, $\OG(6,12)$,
$\LG(5,10)$, $Q^{10}$, $Q^{11}$, as well as all cominuscule subvarieties of
these spaces.

\targetsec{qkdual}{}%
Let $\ov A_{u,v}^{w,d} \in \Z$ denote the dual structure constants of the
non-equivariant ring $\QK(X)$, that is, we have $\cI_q^u \star \cI_q^v = \sum
\ov A_{u,v}^{w,d}\, q^d\, \cI_q^w$ in $\QK(X)$. These constants are determined
by the non-equivariant specialization of \eqn{dualcomp}, where $\cJ$ is the
identity map and the factors $[\C_{-\delta(w/u)}]$ are ignored in the definition
of $\zeta$. The special case $d=0$ of the next result follows from
\cite[Thm.~1]{brion:positivity} and \cite[Remark~3.7]{graham.kumar:positivity}.

\begin{cor}\label{cor:dualsigns}%
  Assume that $X$ is a minuscule flag variety or a quadric hypersurface. Then
  the dual structure constants of $\QK(X)$ satisfy
  \[
    (-1)^{\ell(uvw)+\deg(q^d)}\, \ov A_{u,v}^{w,d} \,\geq\, 0 \,.
  \]
\end{cor}
\begin{proof}
  When $X$ is minuscule, $\zeta$ is the identity map, so \eqn{dualcomp} simplifies
  to
  \[
    \sum_{w,d} \ov A_{u,v}^{w,d}\, q^d\, \cO^w \,=\,
    J \star \cO^u \star \cO^v \,.
  \]
  The result follows from this because the product $J \star \cO^u \star \cO^v$
  has alternating signs by \Theorem{qkpos}, that is, it is an alternating linear
  combination of the $\Z$-basis $\{q^d\cO^w\}$. The claim for odd quadrics is
  verified in the following example.
\end{proof}

\begin{example}\label{example:quadric-exceptional}%
  Let $X = Q^{2n-1} = \OG(1,2n+1)$ be a quadric hypersurface of type $B_n$. We
  assume $n \geq 3$ for simplicity. We have $d_X(2) = 2$, $\deg(q) = 2n-1$, and
  \[
  \cP_X = \tableau{20}{1&2&{\cdots}&n\!\!-\!\!1&n&n\!\!-\!\!1&{\cdots}&2&1} \,.
  \]
  All boxes of $\cP_X$ are long except the middle box. Notice that
  $\ka_1.\cP_{F_1} = z_1.\cP_{F_1}$ consists of the middle $2n-3$ boxes of
  $\cP_X$. It follows that $\cO_u \star \cO^v$ has an exceptional degree if and
  only if $I(u) \ssm I(v)$ consists of the middle box. In other words, the only
  exceptional product is $\cO^{n-1} \star \cO^{n-1} = \cO_n \star \cO^{n-1}$;
  here we denote each element $u \in W^X$ by its length $\ell(u)$. By the
  Chevalley formula from \cite{buch.chaput.ea:chevalley} we have
  \[
    \cO^1 \star \cO^u = \begin{cases}
      2 \cO^n - \cO^{n+1} & \text{if $u = n-1$,} \\
      \cO^{2n-1} + q - q \cO^1 & \text{if $u = 2n-2$,} \\
      q \cO^1 & \text{if $u = 2n-1$, and} \\
      \cO^{u+1} & \text{otherwise.}
    \end{cases}
  \]
  Using this and the associativity of the quantum $K$-theory product, we obtain
  \[
  \cO^{n-1} \star \cO^{n-1} = (\cO^1)^{n-1} \star \cO^{n-1} =
  2 \cO^{2n-2} - \cO^{2n-1} - q + q \cO^1 \,.
  \]
  This product has alternating signs and exceptional degree 1. The corresponding
  product in $\QH(X)$ is $[X^{n-1}] \star [X^{n-1}] = 2[X^{2n-2}]$.

  To prove that the dual structure constants $\ov A_{u,v}^{w,d}$ also have
  alternating signs, notice that $(2-\cO^1)\star\cO^n = \cO^1\star\cO^{n-1}$,
  and therefore
  \[
    (2-\cO^1)\star\cO^n\star\cO^n = (\cO^1)^n\star\cO^n
    = 2q\cO^1 - q\cO^2 = (2-\cO^1)\star q\cO^1 \,.
  \]
  Since $2-\cO^1$ is a non-zero divisor in $\QK(X)$, we obtain $\cO^n \star
  \cO^n = q \cO^1$. The (non-equivariant) maps $\zeta$ and $\zeta^{-1}$ are given by
  \[
    \zeta(\cO^u) = \begin{cases}
      \cO^u & \text{\!\!if $u \neq n\!-\!1$,}\\
      \cO^{n-1}\!-\!\cO^n & \text{\!\!if $u = n\!-\!1$}
    \end{cases}
    \ \ \ \text{and} \ \ \
    \zeta^{-1}(\cO^u) = \begin{cases}
      \cO^u & \text{\!\!if $u \neq n\!-\!1$,}\\
      \cO^{n-1}\!+\!\cO^n & \text{\!\!if $u = n\!-\!1$.}
    \end{cases}
  \]
  Using the above identities, we compute the products
  \[
    J \star \zeta^{-1}(\cO^{n-1}) \,=\, \cO^{n-1} - \cO^n
  \]
  and
  \[
    J \star \zeta^{-1}(\cO^{n-1}) \star \zeta^{-1}(\cO^{n-1}) \,=\,
    (\cO^{n-1} - \cO^n) \star (\cO^{n-1} + \cO^n) \,=\,
    2\cO^{2n-2} - \cO^{2n-1} - q \,.
  \]
  Since both of these products have alternating signs, we deduce from
  \Theorem{qkpos} that the product $J \star \zeta^{-1}(\cO^u) \star
  \zeta^{-1}(\cO^v)$ has alternating signs for all $u,v \in W^X$. Noting that the
  map $\zeta$ preserves alternating signs, this implies that the right side of
  \eqn{dualcomp} has alternating signs, which completes the proof of
  \Corollary{dualsigns}.
\end{example}

\subsection{A geometric construction of the quantum product}
\label{sec:geomprod}%

We give a geometric construction of the classes $(\cO_u \star \cO^v)_d \in K(X)$
that is better suited for determining the signs of the structure constants of
$\QK(X)$. Recall the notation from \Section{prim-nbhd}.

\begin{lemma}\label{lemma:trans-Gamma}%
  Let $1 \leq d \leq d_X(2)$. The diagonal action of $G$ on the set
  $\{(\eta,\om) \in Y_{d-1}\times Y_d \mid \Gamma_\eta \subset \Gamma_\om
  \}$ is transitive.
\end{lemma}
\begin{proof}
  Let $(\eta,\om) \in Y_{d-1} \times Y_d$ be such that $\Gamma_\eta \subset
  \Gamma_\om$. We must show that $(\eta,\om)$ is in the orbit
  $G.(1.P_{Y_{d-1}},1.P_{Y_d})$. Since $G$ acts transitively on $Y_d$, we may
  assume that $\om = 1.P_{Y_d}$. Choose $x,y \in \Gamma_\eta$ such that
  $\dist(x,y)=d-1$. Then $\Gamma_{d-1}(x,y) = \Gamma_\eta$ by
  \Corollary{comin_kapd}, and \Lemma{orbits} applied to $\Gamma_\om$ shows that
  we can find $g \in P_{Y_d}$ such that $g.(x,y) = (1.P_X,\kappa_{d-1}.P_X)$. It
  follows that $g.\Gamma_\eta = X_{\ka_{d-1}}$, as required.
\end{proof}

\targetsec{Ydd1}{}%
For $1 \leq d \leq d_X(2)$ we set $Y_{d-1,d} = G/(P_{Y_{d-1}} \cap P_{Y_d})$. By
\Lemma{trans-Gamma} we can make the identification
\[
  Y_{d-1,d} \,=\,
  \{(\eta,\om) \in Y_{d-1}\times Y_d \mid \Gamma_\eta \subset \Gamma_\om \} \,.
\]
Let $\psi_d : Y_{d-1,d} \to Y_{d-1}$ and $\phi_d : Y_{d-1,d} \to Y_d$ be the
projections. Given $u, v \in W^X$ we define the varieties
\[
  \begin{split}
    Y_{d-1,1}(X_u,X^v) &\,=\, \phi_d(\psi_d^{-1}(Y_{d-1}(X_u,X^v))) \\
    &\,=\, \{ \om \in Y_d \mid \exists\, \eta \in Y_{d-1}(X_u,X^v) :
    \Gamma_\eta \subset \Gamma_\om \} \,,\\
    Z_{d-1,1}(X_u,X^v) &\,=\, q_d^{-1}(Y_{d-1,1}(X_u,X^v)) \,\text{, and}\\
    \Gamma_{d-1,1}(X_u,X^v) &\,=\, p_d(Z_{d-1,1}(X_u,X^v)) \,.
  \end{split}
\]
Notice that $\psi_d^{-1}(Y_{d-1}(X_u,X^v))$ is a Richardson variety and
$Y_{d-1,1}(X_u,X^v)$ is a projected Richardson variety, so \Theorem{projrich}
implies that $Y_{d-1,1}(X_u,X^v)$ has rational singularities and
$[\cO_{Y_{d-1,1}(X_u,X^v)}] = (\phi_d)_* (\psi_d)^*
[\cO_{Y_{d-1}(X_u,X^v)}]$. Since the map $q_d : Z_{d-1,1}(X_u,X^v) \to
Y_{d-1,1}(X_u,X^v)$ is a locally trivial fibration with non-singular fibers, it
follows that $Z_{d-1,1}(X_u,X^v)$ has rational singularities as well. On the
other hand, $\Gamma_{d-1,1}(X_u,X^v)$ is not in general a (projected) Richardson
variety. It would be interesting to understand the singularities of this
variety. In \Example{notprojrich} we give an example where
$\Gamma_{d-1,1}(X_u,X^v)$ has rational singularities and fails to be a projected
Richardson variety.

\begin{question}
  Does $\Gamma_{d-1,1}(X_u,X^v)$ always have rational singularities?
\end{question}

The following lemma applied to $Y_{d-1}(X_u,X^v)$ shows that
$\Gamma_{d-1,1}(X_u,X^v) = \Gamma_1(\Gamma_{d-1}(X_u,X^v))$, that is,
$\Gamma_{d-1,1}(X_u,X^v)$ is the set of all points in $X$ that are connected by
a line to a stable curve of degree $d-1$ meeting $X_u$ and $X^v$. The statement
uses the maps of the following diagram.
\[
  \xymatrix{& Z_d \ar[r]^{p_d\,} \ar[d]^{q_d} & X
    & Z_1 \ar[l]_{\ p_1} \ar[d]^{q_1} \\
    Y_{d-1,d} \ar[r]^{\ \,\phi_d} \ar[d]_{\psi_d} & Y_d && Y_1 \\
    Y_{d-1} & & Z_{d-1} \ar[ll]_{q_{d-1}} \ar[uu]_{p_{d-1}}}
\]

\begin{lemma}\label{lemma:pusharound}
  For any subset $\Omega \subset Y_{d-1}$ we have
  \begin{equation}\label{eqn:nbhdid}
    p_d\, q_d^{-1}\, \phi_d\, \psi_d^{-1}(\Omega) =
    p_1\, q_1^{-1}\, q_1\, p_1^{-1}\, p_{d-1}\, q_{d-1}^{-1}(\Omega) \,.
  \end{equation}
\end{lemma}
\begin{proof}
  We may assume that $\Omega = \{\eta\}$ contains a single point $\eta \in
  Y_{d-1}$, in which case the right hand side of \eqn{nbhdid} is equal to
  $\Gamma_1(\Gamma_\eta)$ by \Corollary{qc_nbhd}. A point $z \in X$ belongs to
  the left hand side of \eqn{nbhdid} if and only if there exists $\om \in Y_d$
  such that $z \in \Gamma_\om$ and $\Gamma_\eta \subset \Gamma_\om$. Since
  $\Gamma_1(z)\cap\Gamma_\om$ and $\Gamma_\eta$ represent dual classes in
  $H^*(\Gamma_\om;\Z)$ by \Lemma{prim_dual}, this implies $\Gamma_1(z) \cap
  \Gamma_\eta \neq \emptyset$, which is equivalent to $z \in
  \Gamma_1(\Gamma_\eta)$. On the other hand, given $z \in
  \Gamma_1(\Gamma_\eta)$, we can choose $x \in \Gamma_1(z) \cap \Gamma_\eta$,
  and then choose $y \in \Gamma_\eta$ such that $\dist(x,y) = d-1$. Since there
  exists a (possibly reducible) rational curve of degree $d$ through $x$, $y$,
  and $z$, it follows from \Proposition{qc_birat} that we may choose $\om \in
  Y_d$ such that $x,y,z \in \Gamma_\om$. Since $x,y \in \Gamma_\om$ and
  $\dist(x,y) = d-1$, it follows from \Corollary{comin_kapd} that $\Gamma_\eta =
  \Gamma_{d-1}(x,y) \subset \Gamma_\om$. This shows that $z$ belongs to the
  left hand side of \eqn{nbhdid}.
\end{proof}

\begin{thm}\label{thm:qkprod}
  We have $(\cO_u \star \cO^v)_d = [\cO_{\Gamma_d(X_u,X^v)}] - (p_d)_*
  [\cO_{Z_{d-1,1}(X_u,X^v)}]$.
\end{thm}
\begin{proof}
  By equation \eqn{qk-prod-d} and \Corollary{qc_nbhd} we have
  \[
  (\cO_u \star \cO^v)_d = [\cO_{\Gamma_d(X_u,X^v)}] -
  (p_1)_* (q_1)^* (q_1)_* (p_1)^* [\cO_{\Gamma_{d-1}(X_u,X^v)}] \,.
  \]
  It is therefore enough to show that
  \begin{multline*}
    (p_d)_* (q_d)^* (\phi_d)_* (\psi_d)^* [\cO_{Y_{d-1}(X_u,X^v)}] \ = \\
    (p_1)_* (q_1)^* (q_1)_* (p_1)^* (p_{d-1})_* (q_{d-1})^*
    [\cO_{Y_{d-1}(X_u,X^v)}] \,.
  \end{multline*}
  More generally, the linear operators
  \[
    (p_d)_* (q_d)^* (\phi_d)_* (\psi_d)^*
    \ \ \ \ \ \text{and} \ \ \ \ \
    (p_1)_* (q_1)^* (q_1)_* (p_1)^* (p_{d-1})_* (q_{d-1})^*
  \]
  define the same map $K(Y_{d-1}) \to K(X)$. In fact, using that $K(Y_{d-1})$
  has a basis of Schubert classes $[\cO_\Omega]$, this follows from
  \Lemma{pusharound}.
\end{proof}

\subsection{Proof of our main theorems}
\label{sec:strategy}

In \cite{buch.chaput.ea:projected} we proved that $\Gamma_d(X_u,X^v)$ has
rational singularities and that
\[
  (p_d)_*[\cO_{Z_d(X_u,X^v)}] = [\cO_{\Gamma_d(X_u,X^v)}] \,.
\]
In fact, this follows from \Corollary{qc_nbhd} and \Theorem{projrich}.

\targetsec{resolclass}{}%
Let $\Omega$ be an irreducible variety defined over $\C$ and let $\rho:
\wt\Omega \to \Omega$ be a resolution of singularities. Define the
\emph{resolution class} of $\Omega$ to be the image $\rho_*[\cO_{\wt\Omega}] =
\sum_{i \geq 0} (-1)^i [R^i \rho_* \cO_{\wt\Omega}]$ in the Grothendieck group
$K(\Omega)$ of coherent sheaves on $\Omega$. This class is independent of the
chosen desingularization and will be denoted simply by $[\cO_{\wt\Omega}]$. When
$\Omega \subset X$ is a closed subvariety, we also write $[\cO_{\wt\Omega}]$ for
the image of the resolution class in $K(X)$. If $\Omega$ has rational
singularities, then $[\cO_{\wt\Omega}] = [\cO_\Omega]$. We need the following
result \cite[\S4, Remark]{brion:positivity}.

\begin{thm}[Brion]\label{thm:brion}%
  Let $M = G/P_M$ be a flag variety over $\C$ and let $\Omega \subset M$ be an
  irreducible closed subvariety. Then the resolution class $[\cO_{\wt\Omega}]$
  is an alternating linear combination of Schubert classes, that is, we have
  \[
    [\cO_{\wt\Omega}] = \sum_{w \in W^M} c_w(\Omega)\, \cO^w
  \]
  in $K(M)$, where $(-1)^{\ell(w)-\codim(\Omega,M)}\, c_w(\Omega) \geq 0$ for
  all $w \in W^M$.
\end{thm}

\begin{thm}\label{thm:kollar}%
  Let $f : \Omega' \to \Omega$ be a surjective morphism between complex
  projective varieties with rational singularities. Then $f$ is cohomologically
  trivial if and only if the general fibers of $f$ are cohomologically trivial.
\end{thm}
\begin{proof}
  The implication \emph{`if'} follows from \cite[Thm~7.1]{kollar:higher} (see
  the proof of \cite[Thm.~3.2]{buch.mihalcea:quantum}), and \emph{`only if'}
  follows from \cite[III.12.8 and III.12.9]{hartshorne:algebraic*1}.
\end{proof}

We will use the following consequence of \Theorem{kollar} when condition (b) is
satisfied. The condition that $\Omega'$ has rational singularities is necessary
in this case, see \Example{push}.

\begin{cor}\label{cor:push}%
  Let $f : \Omega' \to \Omega$ be a surjective morphism of irreducible
  projective varieties over $\C$. Assume that either {\rm(a)} the general fibers
  of $f$ are rationally connected, or {\rm(b)} $\Omega'$ has rational
  singularities and the general fibers of $f$ are cohomologically trivial. Then,
  $f_*[\cO_{\wt\Omega'}] = [\cO_{\wt\Omega}]$.
\end{cor}
\begin{proof}
  Let $\wt\Omega$ be a desingularization of $\Omega$, and let $\wt\Omega'$ be a
  desingularization of the unique irreducible component of $\Omega'
  \times_\Omega \wt\Omega$ that maps birationally onto $\Omega'$. We obtain a
  commutative diagram where the vertical maps are resolutions of singularities.
  \[
    \xymatrix{
      \wt\Omega' \ar[d]_{\pi'} \ar[r]^{\wt f} & \wt\Omega \ar[d]^\pi \\
      \Omega' \ar[r]^f & \Omega
    }
  \]
  Let $U' \subset \Omega'$ be a dense open subset such that $\pi' :
  {\pi'}^{-1}(U') \to U'$ is an isomorphism, and set $Z = \wt\Omega' \ssm
  {\pi'}^{-1}(U')$. For $x \in \Omega$ we let $\Omega'_x \subset \Omega'$,
  $\wt\Omega'_x \subset \wt\Omega'$, and $Z_x \subset Z$ denote the fibers over
  $x$. Set $r = \dim(\Omega') - \dim(\Omega)$. Choose a dense open subset $U
  \subset \Omega$ such that $f\pi' : (f\pi')^{-1}(U) \to U$ is smooth,
  $\dim(Z_x) < r$ for all $x \in U$, and $\Omega'_x$ is rationally connected for
  $x \in U$ in case (a), or cohomologically trivial with rational singularities
  in case (b). Here we use that the general fibers of $f$ have rational
  singularities when $\Omega'$ has rational singularities by
  \cite[Lemma~3]{brion:positivity}.

  Let $x \in U$. Then $\wt\Omega'_x$ is a disjoint union of non-singular
  varieties of dimension $r$, and $\Omega'_x$ is irreducible. Since
  $\wt\Omega'_x \cap {\pi'}^{-1}(U') \subset \wt\Omega'_x$ is a dense open
  subset isomorphic to $\Omega'_x \cap U'$, it follows that $\wt\Omega'_x$ is
  birational to $\Omega'_x$. We deduce that $\wt\Omega'_x$ is cohomologically
  trivial; this follows from \cite[Cor.~4.18(a)]{debarre:higher-dimensional} if
  $\Omega'_x$ is rationally connected, and from the Leray spectral sequence if
  $\Omega'_x$ is cohomologically trivial with rational singularities.
  \Theorem{kollar} now shows that $\wt f$ is cohomologically trivial, which
  completes the proof.
\end{proof}

The alternating signs conjecture for $\QK(X)$ would be a consequence of the
following two statements.

\begin{enumerate}
\item[(I)] The general fibers of the map $p_d : Z_{d-1,1}(X_u,X^v) \to
  \Gamma_{d-1,1}(X_u,X^v)$ are cohomologically trivial.\smallskip

\item[(II)] The variety $\Gamma_{d-1,1}(X_u,X^v)$ is either equal to
  $\Gamma_d(X_u,X^v)$ or a divisor in $\Gamma_d(X_u,X^v)$.
\end{enumerate}

In fact, the class $[\cO_{\Gamma_d(X_u,X^v)}]$ has alternating signs by
\Theorem{brion}, and these signs are compatible with \Conjecture{qkpos} for
$\dmin(u^\vee,v) \leq d \leq \dmax(u^\vee,v)$, as \Proposition{qhcoef} shows
that $\codim(\Gamma_d(X_u,X^v), X) = \ell(u^\vee) + \ell(v) - \deg(q^d)$.
Property (I) implies that $(p_d)_*[\cO_{Z_{d-1,1}(X_u,X^v)}]$ is the resolution
class of $\Gamma_{d-1,1}(X_u,X^v)$ by \Corollary{push}, which also has
alternating signs by \Theorem{brion}. The point of (II) is that, if
$\Gamma_{d-1,1}(X_u,X^v)$ is a divisor in $\Gamma_d(X_u,X^v)$, then the
alternating signs of the two terms in \Theorem{qkprod} enhance each other to
yield the alternating signs of $(\cO_u \star \cO^v)_d$. Properties (I) and (II)
also imply that $(\cO_u \star \cO^v)_d$ is non-zero if and only if
$\Gamma_{d-1,1}(X_u,X^v) \neq \Gamma_d(X_u,X^v)$. This determines whether the
power $q^d$ occurs in $\cO_u \star \cO^v$. We will show that (II) is true in all
cases, whereas (I) holds if and only if $d$ is not an exceptional degree of
$\cO_u \star \cO^v$. These results are sufficient to establish
\Theorem{qkdegrees} and \Theorem{qkpos}. \Table{exceptional-counts} illustrates
that most products $\cO_u \star \cO^v$ on Lagrangian Grassmannians are fully
described by these results.

\begin{proof}[Proofs of \Theorem{qkdegrees} and \Theorem{qkpos}]
  \Table{divisors} shows the range of degrees where $Y_{d-1,1}(X_u,X^v)$ is a
  divisor in $Y_d(X_u,X^v)$, and where $\Gamma_{d-1,1}(X_u,X^v)$ is a divisor in
  $\Gamma_d(X_u,X^v)$. The codimension of $Y_{d-1,1}(X_u,X^v)$ is determined by
  \Proposition{yd1projrich} and \Proposition{yd=yd1}, after which the
  codimension of $\Gamma_{d-1,1}(X_u,X^v)$ is determined by
  \Proposition{zd1birat} and \Corollary{gd=gd1}. The results now follow from
  \Corollary{zd1-fibers} using the strategy discussed above.
\end{proof}

\begin{table}
  \caption{Divisors in $Y_d(X_u,X^v)$ and $\Gamma_d(X_u,X^v)$.}
  \label{tab:divisors}
  \def\arraystretch{1.3}
  \begin{tabular}{|c|c|c|}
    \hline
    Range of degrees & $Y_{d-1,1}(X_u,X^v)$ & $\Gamma_{d-1,1}(X_u,X^v)$ \\
    \hline
    $1 \leq d \leq \dmin(u^\vee,v)$ &
    $= \emptyset$ & $= \emptyset$
    \\
    $\dmin(u^\vee,v) < d \leq \dmax(u^\vee,v)$ &
    $\subsetneq \ Y_d(X_u,X^v)$ & $\subsetneq \ \Gamma_d(X_u,X^v)$
    \\
    $\dmax(u^\vee,v) < d \leq \min(\dmax(u^\vee),\dmax(v))$ &
    $\subsetneq \ Y_d(X_u,X^v)$ & $= \ \Gamma_d(X_u,X^v)$
    \\
    $\min(\dmax(u^\vee),\dmax(v)) < d \leq d_X(2)$ &
    $= \ Y_d(X_u,X^v)$ & $= \ \Gamma_d(X_u,X^v)$
    \\
    \hline
  \end{tabular}
\end{table}

\begin{example}\label{example:push}%
  Let $E \subset \bP^2$ be an elliptic curve, let $\Omega' \subset \bP^3$ be the
  cone over $E$, set $\Omega = \Spec(\C)$, and let $f : \Omega' \to \Omega$ be
  the structure morphism. Using the exact sequence $0 \to \cO_{\bP^3}(-3) \to
  \cO_{\bP^3} \to \cO_{\Omega'} \to 0$, it follows that the fibers of $f$ are
  cohomologically trivial. Let $\wt\Omega'$ be the blow-up of $\Omega'$ at its
  vertex, and let $g : \wt\Omega' \to E$ be the map induced by the projection
  $\Omega' \dashrightarrow E$. Since the fibers of $g$ are projective lines, it
  follows from \Corollary{push}(a) that $g_*[\cO_{\wt\Omega'}] = [\cO_E]$. Since
  $\chi(E,\cO_E) = 0$, we deduce that $f_*[\cO_{\wt\Omega'}] = 0 \neq
  [\cO_{\Omega}]$. This shows that \Corollary{push}(b) may fail without the
  assumption that $\Omega'$ has rational singularities.
\end{example}

\begin{remark}\label{remark:qkt-pos}%
  Since \Corollary{push} remains true in equivariant $K$-theory, our results
  imply that the identity
  \[
    (\cO_u \star \cO^v)_d =
    [\cO_{\Gamma_d(X_u,X^v)}] - [\cO_{\wt{\Gamma_{d-1,1}}(X_u,X^v)}]
  \]
  holds in the equivariant quantum $K$-theory ring $\QK_T(X)$ whenever $d$ is
  not an exceptional degree. In particular, \Theorem{qkdegrees} holds for
  $\QK_T(X)$.

  Let $\wt N^{w,d}_{u,v} \in K_T(\pt)$, for $u,v,w \in W^X$, denote the
  structure constants describing the action of the $B^-$-stable Schubert basis
  $\{\cO^v\}$ on the $B$-stable basis $\{\cO_{u^\vee}\}$ of $\QK_T(X)$:
  \[
    \cO_{u^\vee} \star \cO^v = \sum_{w,d \geq 0} \wt N^{w,d}_{u,v}\, q^d\,
    \cO_{w^\vee}
  \]
  If \Theorem{brion} is upgraded to the equivariant setting of
  \cite[Thm.~4.1]{anderson.griffeth.ea:positivity}, then these constants would
  satisfy the positivity property
  \[
    (-1)^{\ell(uvw)+\deg(q^d)}\, \wt N^{w,d}_{u,v}
    \ \in \ \N\big[ [\C_{-\be}]-1 : \be \in \Delta \big] \,.
  \]
  We thank David~Anderson for sending us an outline of a proof of the
  equivariant version of \Theorem{brion}, with some details left to check. We
  hope to address this elsewhere, and possibly prove a slight generalization of
  \Theorem{brion}.

  The classes $\cO_{u^\vee}$ and $\cO^u$ are distinct in $K_T(X)$, so it is not
  clear how to apply our results to products $\cO^u \star \cO^v$ of two
  $B^-$-stable Schubert classes in $\QK_T(X)$. A positivity conjecture for the
  structure constants of such products is discussed in
  \cite[Conj.~2.2]{buch.chaput.ea:chevalley}, generalizing
  \cite[Conj.~5.1]{griffeth.ram:affine} and
  \cite[Cor.~5.3]{anderson.griffeth.ea:positivity}.
\end{remark}

\begin{conj}
  The power $q^d$ occurs in the equivariant quantum product $\cO^u \star \cO^v
  \in \QK_T(X)$ if and only if $0 \leq d \leq \dmax(u,v)$ or $d = \dmax(u,v)+1$
  is an exceptional degree of $\cO_{u^\vee} \star \cO^v$.
\end{conj}

\subsection{Proofs and counterexamples to (I) and (II)}\label{sec:pf-cex}%

Fix elements $u, v \in W^X$ and a degree $1 \leq d \leq d_X(2)$. We proceed to
establish the required properties of the map $p_d : Z_{d-1,1}(X_u,X^v) \to
\Gamma_{d-1,1}(X_u,X^v)$. Notice that $\Gamma_d(X_u,X^v)$ is empty for $d <
\dmin(u^\vee,v)$, and $\Gamma_{d-1,1}(X_u,X^v)$ is empty for $d \leq
\dmin(u^\vee,v)$.

\begin{lemma}\label{lemma:swtald}%
  We have $w_{0,Y_d}^{Y_{d-1,d}} = s_{\tal_d}$.
\end{lemma}
\begin{proof}
  The element $w_{0,Y_d}^{Y_{d-1,d}}$ describes the fiber of the map $\phi_d :
  Y_{d-1,d} \to Y_d$ over $1.P_{Y_d}$, that is, $\phi_d^{-1}(1.P_{Y_d}) =
  (Y_{d-1,d})_{w_{0,Y_d}^{Y_{d\!-\!1,d}}}$. Since this fiber does not change if
  the cominuscule variety $X$ is replaced with $X_{\ka_d}$, we may assume that
  $X$ is a primitive cominuscule variety of diameter $d$. In this case $Y_d$ is
  a point and $Y_{d-1,d} = Y_{d-1}$, so \Lemma{weylid} and \Lemma{prim_dual}
  imply that $w_{0,Y_d}^{Y_{d-1,d}} = w_0 w_{0,Y_{d-1}} = w_0 w_{0,X} z_{d-1} =
  w_0^X z_{d-1} = s_{\tal_d}$, as required.
\end{proof}

We first consider degrees in the range $\dmin(u^\vee,v) < d \leq
\min(\dmax(u^\vee),\dmax(v))$. In this case the maps $q_d : p_d^{-1}(X_u) \to
Y_d(X_u)$ and $q_d : p_d^{-1}(X^v) \to Y_d(X^v)$ are birational by
\Corollary{fibers_geom}(a). Since the fibers of $p_d : Z_d \to X$ and $q_d : Z_d
\to Y_d$ are described by $w_{0,X}^{Z_d} = z_d/\ka_d$ and $w_{0,Y_d}^{Z_d} =
\ka_d$ by \Lemma{weylid}, we deduce that $u z_d \ka_d = u(z_d/\ka_d)$ and
$v\ka_d = v/\ka_d$ belong to $W^{Y_d}$, and $u z_d \in W^{Z_d}$. With the
notation for projected Richardson varieties from \Section{projrich}, we obtain
$Y_d(X_u,X^v) = (Y_d)_{u z_d \ka_d}^{v\ka_d} = \Pi_{u z_d \ka_d}^{v\ka_d}(Y_d)$
and $Z_d(X_u,X^v) = (Z_d)_{u z_d}^{v\ka_d} = \Pi_{u z_d}^{v\ka_d}(Z_d)$.

\begin{prop}\label{prop:yd1projrich}%
  Assume that $\dmin(u^\vee,v) < d \leq \min(\dmax(u^\vee),\dmax(v))$. Then
  $v\ka_d s_\ga \leq_{Y_d} u z_d \ka_d$, and $Y_{d-1,1}(X_u,X^v) = \Pi_{u z_d
  \ka_d}^{v\ka_d s_\ga}(Y_d)$ is a divisor in $Y_d(X_u,X^v)$.
\end{prop}
\begin{proof}
  Define $\eta \in W_{Z_d}$ by $s_\ga \eta = \ka_d/\ka_{d-1}$. We have
  $Y_{d-1}(X^v) = (Y_{d-1})^{v/\ka_{d-1}} = (Y_{d-1})^{v \ka_d s_\ga \eta}$, and
  \Proposition{fibers_combin}(a) shows that $v \ka_d s_\ga \eta = v/\ka_{d-1}
  \in W^{Y_{d-1}}$. We also have $\ka_d/\ka_{d-1} \in W_{Y_d} \cap W^{Y_{d-1}}$,
  hence $u z_d \ka_d \eta \in W^{Y_{d-1,d}}$. Define
  \[
    Z_{d-1,d} = G/(P_{Y_{d-1}} \cap P_{Y_d} \cap P_X) =
    \{(\eta,\om,x) \in Y_{d-1}\times Y_d\times X \mid x \in
    \Gamma_\eta \subset \Gamma_\om \} \,,
  \]
  with projections $p : Z_{d-1,d} \to Z_d$ and $q : Z_{d-1,d} \to Y_{d-1,d}$.
  \[
    \xymatrix{
      X & Z_{d-1} \ar[l]_{p_{d-1}} \ar[d]_{q_{d-1}} &
      Z_{d-1,d} \ar[l] \ar[r]^p \ar[d]^q & Z_d \ar[r]^{p_d} \ar[d]^{q_d} & X \\
      & Y_{d-1} & Y_{d-1,d} \ar[l]_{\psi_d} \ar[r]^{\phi_d} & Y_d
    }
  \]
  Using that $\ka_d w_{0,Z_d}^{Z_{d-1,d}} = w_{0,Y_d}^{Y_{d-1,d}}
  w_{0,Y_{d-1,d}}^{Z_{d-1,d}} = s_{\tal_d} \ka_{d-1}$ by \Lemma{weylid} and
  \Lemma{swtald}, we obtain $w_{0,Z_d}^{Z_{d-1,d}} = \ka_d s_{\tal_d} \ka_{d-1}
  = \eta$. This implies
  \[
    \psi_d^{-1}(Y_{d-1}(X_u)) = q p^{-1} p_d^{-1}(X_u)
    = (Y_{d-1,d})_{u z_d \ka_d \eta} \,.
  \]
  We obtain
  \[
    \begin{split}
      Y_{d-1,1}(X_u,X^v) &= \phi_d(\psi_d^{-1}(Y_{d-1}(X_u,X^v))) \\
      &= \phi_d \left( \psi_d^{-1}(Y_{d-1}(X_u)) \cap
        \psi_d^{-1}(Y_{d-1}(X^v)) \right) \\
      &= \phi_d \left( (Y_{d-1,d})_{u z_d \ka_d \eta}^{v\ka_d
        s_\ga \eta}\right)
       = \Pi_{u z_d \ka_d \eta}^{v\ka_d s_\ga \eta}(Y_d) =
      \Pi_{u z_d \ka_d}^{v\ka_d s_\ga}(Y_d) \,,
    \end{split}
    \]
    where the last two equalities follow from \Corollary{rich2rich} and
    \Theorem{simX}(b), or \cite[Prop.~3.3]{knutson.lam.ea:projections}. The
    inequality $v\ka_d s_\ga \leq_{Y_d} u z_d \ka_d$ holds because
    $Y_{d-1}(X_u,X^v) \neq \emptyset$ and $u z_d \ka_d \in W^{Y_d}$. Finally, it
    follows from \Proposition{rich_iso} that $Y_{d-1,1}(X_u,X^v)$ is a divisor
    in $Y_d(X_u,X^v) = \Pi_{u z_d \ka_d}^{v \ka_d}(Y_d)$.
\end{proof}

\begin{lemma}\label{lemma:zd1open}%
  Assume that $\dmin(u^\vee,v) < d \leq \min(\dmax(u^\vee),\dmax(v))$. Then
  $Z_{d-1,1}(X_u,X^v) \cap \oPi_{u z_d}^{v\ka_d}(Z_d)$ is a dense open subset of
  $Z_{d-1,1}(X_u,X^v)$.
\end{lemma}
\begin{proof}
  By \Proposition{yd1projrich}, $Z_{d-1,1}(X_u,X^v)$ is a divisor in $Z_d(X_u,X^v) =
  \Pi_{u z_d}^{v\ka_d}(Z_d)$. If the claim is false, then \Theorem{richmodel}
  implies that $Z_{d-1,1}(X_u,X^v) = \Pi_a^b(Z_d)$, where $v\ka_d \leq b
  \leq_{Z_d} a \leq u z_d$. Since $\Pi_a^b(Y_d) = \Pi_{u z_d \ka_d}^{v\ka_d
  s_\ga}(Y_d)$ by \Proposition{yd1projrich} and $u z_d \ka_d \in W^{Y_d}$, it
  follows from \Theorem{simX} that $a \geq u z_d \ka_d$ and $b \geq v\ka_d
  s_\ga$. Since we also have $\ell(a)-\ell(b) = \ell(u z_d) - \ell(v\ka_d
  s_\ga)$, it follows that $a = u z_d$ and $b = v\ka_d s_\ga$. But
  \Theorem{simX} also implies that $\Pi_{u z_d}^{v\ka_d s_\ga}(Y_d) = \Pi_{u z_d
  \ka_d}^{v\ka_d}(Y_d)$, a contradiction.
\end{proof}

\begin{prop}\label{prop:zd1birat}%
  Assume that $\dmin(u^\vee,v) < d \leq \dmax(u^\vee,v)$. Then the map $p_d :
  Z_{d-1,1}(X_u,X^v) \to \Gamma_{d-1,1}(X_u,X^v)$ is birational.
\end{prop}
\begin{proof}
  Since $p_d : Z_d(X_u,X^v) \to \Gamma_d(X_u,X^v)$ is birational by
  \Proposition{qhcoef} and \Corollary{fibers_geom}(d), it follows from
  \Proposition{rich_iso} that the restriction of $p_d$ to the open projected
  Richardson variety $\oPi_{u z_d}^{v\ka_d}(Z_d)$ is injective. The result
  therefore follows from \Lemma{zd1open}.
\end{proof}

Our next result shows that $Z_{d-1,1}(X_u,X^v) = Z_d(X_u,X^v)$ and
$\Gamma_{d-1,1}(X_u,X^v) = \Gamma_d(X_u,X^v)$ whenever $d >
\min(\dmax(u^\vee),\dmax(v))$.

\begin{prop}\label{prop:yd=yd1}%
  Assume $d > \min(\dmax(u^\vee),\dmax(v))$. Then, $Y_{d-1,1}(X_u,X^v) =
  Y_d(X_u,X^v)$.
\end{prop}
\begin{proof}
  Let $\om \in Y_d(X_u,X^v)$ and assume that $d > \dmax(v)$. Then
  $\Gamma_\om \cap X^v$ has positive dimension by \Corollary{fibers_geom}(a).
  Choose any point $x \in \Gamma_\om \cap X_u$. Then $\Gamma_{d-1}(x) \cap
  \Gamma_\om$ is a divisor in $\Gamma_\om$ by \Lemma{prim_dual}. It
  follows that $\Gamma_{d-1}(x) \cap \Gamma_\om \cap X^v \neq \emptyset$.
  Choose any point $y \in \Gamma_{d-1}(x) \cap \Gamma_\om \cap X^v$. Since
  $\dist(x,y) \leq d-1$, there exists $\eta \in Y_{d-1}$ such that $x,y \in
  \Gamma_\eta \subset \Gamma_\om$ by \Corollary{comin_kapd}(b) applied to
  $\Gamma_\om$. This proves that $\om \in Y_{d-1,1}(X_u,X^v)$. A symmetric
  argument works if $d > \dmax(u^\vee)$.
\end{proof}

\noin
We finally discuss the remaining range $\dmax(u^\vee,v) < d \leq
\min(\dmax(u^\vee),\dmax(v))$. For degrees in this range, the map $p_d :
Z_d(X_u,X^v) \to \Gamma_d(X_u,X^v)$ has fibers of positive dimension by
\Proposition{qhcoef} and \Corollary{fibers_geom}(d).

\begin{notation}\label{notation:epsilon}%
  Given a fixed cominuscule variety $X$, we let $\epsilon$ denote the constant
  defined by $\epsilon=1$ if $X$ is minuscule or an odd quadric of dimension at
  least five, while $\epsilon=2$ if $X$ is a Lagrangian Grassmannian. The
  three-dimensional quadric $Q^3 = \LG(2,4)$ is considered a Lagrangian
  Grassmannian.
\end{notation}

\targetsec{partialYd}{}%
The proof of the following result is postponed to \Section{divisors}, where we
also justify the definition of $\epsilon$. Let $[\partial Y_d] = \sum_{\be \in
\Delta\ssm\Delta_{Y_d}} [Y_d^{s_\be}]$ denote the (ample) sum of the Schubert
divisors in $Y_d$.

\begin{prop}\label{prop:zd1-fibers}%
  Let $\dmax(u^\vee,v) < d \leq \min(\dmax(u^\vee),\dmax(v))$. For all points
  $z$ in a dense open subset of $\Gamma_d(X_u,X^v)$, the fiber $D = p_d^{-1}(z)
  \cap Z_{d-1,1}(X_u,X^v)$ is a Cartier divisor of class $\epsilon\,
  q_d^*[\partial Y_d]$ in the Richardson variety $R = p_d^{-1}(z) \cap
  Z_d(X_u,X^v)$.
\end{prop}

\begin{cor}\label{cor:gd=gd1}%
  For $d > \dmax(u^\vee,v)$ we have $\Gamma_{d-1,1}(X_u,X^v) =
  \Gamma_d(X_u,X^v)$.
\end{cor}
\begin{proof}
  This follows from \Proposition{yd=yd1} and \Proposition{zd1-fibers}.
\end{proof}

\begin{cor}\label{cor:zd1-fibers}%
  Let $\dmin(u^\vee,v) < d \leq d_X(2)$. The general fibers of the map $p_d :
  Z_{d-1,1}(X_u,X^v) \to \Gamma_{d-1,1}(X_u,X^v)$ are cohomologically trivial if
  and only if $d$ is not an exceptional degree of $\cO_u \star \cO^v$.
\end{cor}
\begin{proof}
  This follows from \Proposition{zd1birat} if $d \leq \dmax(u^\vee,v)$, and it
  follows from \Proposition{yd=yd1} and \Corollary{richfib} for $d >
  \min(\dmax(u^\vee), \dmax(v))$. Assume that $\dmax(u^\vee,v) < d \leq
  \min(\dmax(u^\vee),\dmax(v))$, and let $R$ and $D$ be as in
  \Proposition{zd1-fibers}. Then $R$ is a translate of the Richardson variety
  $(F_d)_{u_d}^{u_d \cap v^d}$, and we have $[D] = \epsilon [\partial F_d]$ in
  $\Pic(R)$. Notice that $R$ has positive dimension by \Proposition{qhcoef}. We
  use \Corollary{trivial} to argue that $D$ is cohomologically trivial if and
  only if $(u_d \cup v^d)/v^d$ is not a short rook strip. If $X$ is minuscule,
  then $F_d$ is a product of minuscule varieties, $\epsilon = 1$, and $D$ is
  cohomologically trivial because there are no decreasing primed tableaux of
  shape $I(u_d) \ssm I(v^d)$ with integer labels from the interval
  $[\frac{1}{2},1)$. If $X = Q^{2n-1}$ is an odd quadric with $n \geq 3$, then
  $\epsilon=1$, and the assumptions imply that $d=1$, hence $F_d = Q^{2n-3}$.
  This time $D$ is cohomologically trivial if and only if there are no
  decreasing primed tableaux of shape $I(u_d) \ssm I(v^d)$ using only the label
  $\frac{1}{2}$, that is, $(u_d \cup v^d)/v_d$ is not a short rook strip.
  Finally, if $X = \LG(n,2n)$ is a Lagrangian Grassmannian, then $\epsilon = 2$,
  $F_d = \Gr(n-d,n)$ is a Grassmannian of type A, and $D$ is cohomologically
  trivial if and only if there are no decreasing primed tableaux of shape
  $I(u_d) \ssm I(v^d)$ using only the label $1$, that is, $(u_d \cup v^d)/v^d$
  is not a rook strip. In this case a rook strip is the same as a short rook
  strip, since all boxes of $\cP_{F_d}$ are short by convention. The result
  follows from these observations.
\end{proof}

Since the general fibers of $p_d : Z_{d-1,1}(X_u,X^v) \to
\Gamma_{d-1,1}(X_u,X^v)$ have rational singularities by
\cite[Lemma~3]{brion:positivity}, it follows from \Corollary{zd1-fibers} that
these fibers are irreducible projective varieties of arithmetic genus zero for
any non-exceptional degree in the range $\dmin(u^\vee,v) < d \leq d_X(2)$. The
following result describes the fibers for exceptional degrees.

\begin{thm}\label{thm:exceptional}%
  Let $d = \dmax(u^\vee,v)+1$ be an exceptional degree of $\cO_u \star \cO^v$.
  Then the general fibers of $p_d : Z_{d-1,1}(X_u,X^v) \to \Gamma_d(X_u,X^v)$
  have rational singularities and arithmetic genus one. They are irreducible
  projective varieties if they have positive dimension.
\end{thm}
\begin{proof}
  The general fibers  $D = Z_{d-1,1}(X_u,X^v) \cap p_d^{-1}(z)$ have rational
  singularities by \cite[Lemma~3]{brion:positivity}. Since there is exactly one
  decreasing primed tableau of shape $I(u_d)\ssm I(v^d)$ with labels in
  $[\frac{1}{2},\epsilon)$ by the proof of \Corollary{zd1-fibers}, it follows
  from \Theorem{Jcohom} and the proof of \Corollary{trivial} that $D$ has
  arithmetic genus 1. In the positive dimensional case, the general fibers are
  connected by \Proposition{zd1-fibers} and the Fulton-Hansen connectedness
  theorem \cite{fulton.hansen:connectedness}.
\end{proof}

\begin{remark}
  When $d = \dmax(u^\vee,v)+1$ is an exceptional degree of $\cO_u \star \cO^v$,
  the general fibers of $p_d : Z_{d-1,1}(X_u,X^v) \to \Gamma_d(X_u,X^v)$ can be
  described more explicitly as follows. Since $(u_d \cup v^d)/v^d$ is a rook
  strip, it follows from \Corollary{fibers_geom}(d) and
  \cite[Lemma~3.2(b)]{buch.ravikumar:pieri} that the Richardson variety $R =
  Z_d(X_u,X^v) \cap p_d^{-1}(z)$ is a product of projective lines for general $z
  \in \Gamma_d(X_u,X^v)$. \Proposition{zd1-fibers} shows that $D =
  Z_{d-1,1}(X_u,X^v) \cap p_d^{-1}(z)$ has multidegree $(2,2,\dots,2)$ in $R$.
  The arithmetic genus of $D$ can also be computed from this description.
\end{remark}

\targetsec{altsign}{}%
Given a non-zero $K$-theory class $\cF \in K(X)$, the \emph{initial term}
$\lead(\cF)$ is defined as the homogeneous component of lowest degree in the
Chern character $\ch(\cF) \in H^*(X,\Q)$. Equivalently, $\lead(\cF)$ is the
leading term of $\cF$ modulo the topological filtration of $K(X)$ (see
\cite[Ex.~15.2.16]{fulton:intersection}). Let $\codim(\cF)$ denote the complex
degree of $\lead(\cF)$, so that $\lead(\cF) \in H^{2\codim(\cF)}(X,\Z)$, and let
\[
  \cF = \sum_{w \in W^X} c_w(\cF)\, \cO^w
\]
be the expansion of $\cF$ in the Schubert basis of $K(X)$. Then $\codim(\cF)$ is
the minimal length $\ell(w)$ for which $c_w(\cF) \neq 0$. The class $\cF$ has
\emph{alternating signs} if $(-1)^{\ell(w) - \codim(\cF)} c_w(\cF) \geq 0$ holds
for all $w \in W^X$.

Part (a) of the following conjecture might point towards a generalization of
Brion's positivity theorem. Parts (b) and (c) imply that $(\cO_u \star \cO^v)_d
\neq 0$ whenever $d$ is an exceptional degree.

\begin{conj}\label{conj:exceptional}%
  Assume that $d = \dmax(u^\vee,v)+1$ is an exceptional degree of $\cO_u \star
  \cO^v$.
  \begin{abcenum}
  \item The class $(p_d)_* [\cO_{Z_{d-1,1}}(X_u,X^v)] \in K(X)$ has alternating
    signs.\smallskip

  \item If $\dim \Gamma_d(X_u,X^v) \not\equiv \dim Z_d(X_u,X^v)$ (mod 2), then
    the initial term of $(p_d)_* [\cO_{Z_{d-1,1}}(X_u,X^v)]$ is equal to $2\,
    [\Gamma_d(X_u,X^v)]$.\smallskip

  \item If $\dim \Gamma_d(X_u,X^v) \equiv \dim Z_d(X_u,X^v)$ (mod 2), then the
    initial term of $(p_d)_* [\cO_{Z_{d-1,1}}(X_u,X^v)]$ has complex degree
    $\codim(\Gamma_d(X_u,X^v),X) + 1$.
  \end{abcenum}
\end{conj}

\begin{example}
  Let $X = Q^{2n-1}$ be a quadric of odd dimension. By
  \Example{quadric-exceptional}, the only exceptional product in $\QK(X)$ is
  $\cO_n \star \cO^{n-1}$, with corresponding exceptional degree $d=1$. Since
  $(\cO_n \star \cO^{n-1})_1 = -1 + \cO^1$, we obtain from \Theorem{qkprod} that
  $\Gamma_1(X_n, X^{n-1}) = X$ and
  \[
    (p_1)_*[\cO_{Z_{0,1}(X_n,X^{n-1})}] =
    [\cO_{\Gamma_1(X_n,X^{n-1})}] - (\cO_n \star \cO^{n-1})_1 = 2 - \cO^1 \,.
  \]
  The general fibers of $p_1 : Z_1(X_n,X^{n-1}) \to X$ are projective lines by
  \Corollary{fibers_geom}(d), and the general fibers of $p_1 :
  Z_{0,1}(X_n,X^{n-1}) \to X$ consist of two reduced points by
  \Proposition{zd1-fibers}. This proves \Conjecture{exceptional} for odd quadrics.
\end{example}

\begin{example}\label{example:lg-exceptional}%
  Let $X = \LG(4,8)$ and define $u,v \in W^X$ by
  $I(u) = \tableau{6}{{}&{}&{}&{}\\&{}&{}}$ and
  $I(v) = \tableau{6}{{}&{}&{}\\&{}}$. The corresponding products in $\QH(X)$
  and $\QK(X)$ are given by
  \[
    [X_u] \star [X^v]
    = 4 [X^{(4,3,1)}] + 4q [X^{(3)}] + 2q [X^{(2,1)}]
  \]
  and
  \[
    \begin{split}
      \cO_u \star \cO^v &=
      4 \cO^{(4,3,1)}
      -4 \cO^{(4,3,2)}
      +\cO^{(4,3,2,1)} \\
      & +4 q \cO^{(3)}
      +2 q \cO^{(2,1)}
      -4 q \cO^{(4)}
      -11 q \cO^{(3,1)}
      +7 q \cO^{(3,2)}
      +7 q \cO^{(4,1)} \\
      & -5 q \cO^{(4,2)}
      -2 q \cO^{(3,2,1)}
      +q \cO^{(4,3)}
      +2 q \cO^{(4,2,1)}
      -q \cO^{(4,3,1)} \\
      & +q^2
      -2 q^2 \cO^{(1)}
      +2 q^2 \cO^{(2)}
      -q^2 \cO^{(3)}
      -q^2 \cO^{(2,1)}
      +q^2 \cO^{(3,1)} \,.
    \end{split}
  \]
  The product $\cO_u \star \cO^v$ has exceptional degree $d = 2$, and we have
  $\Gamma_d(X_u,X^v) = X$, $F_d = \Gr(2,4)$, $u_d = \tableau{6}{{}&{}\\{}}$, and
  $v^d = \tableau{6}{{}}$. The general fibers of $p_d : Z_{d-1,1}(X_u,X^v) \to
  X$ are elliptic curves by \Theorem{exceptional}. The identity
  \[
    (p_d)_*[\cO_{Z_{d-1,1}(X_u,X^v)}] = 1 - (\cO_u \star \cO^v)_d =
    2 \cO^{(1)} - 2 \cO^{(2)} + \cO^{(3)} + \cO^{(2,1)} - \cO^{(3,1)}
  \]
  shows that \Conjecture{exceptional} holds for the product $\cO_u \star \cO^v$.
\end{example}

In \cite[Ex.~5.4]{buch.chaput.ea:projected} we gave an example of a projected
Richardson variety in the Grassmannian $\Gr(2,6)$ that is not of the form
$\Gamma_d(X_u,X^v)$. On the other hand, the following example shows that not all
varieties of the form $\Gamma_{d-1,1}(X_u,X^v)$ are projected Richardson
varieties. The studied variety $\Gamma_{d-1,1}(X_u,X^v)$ has rational
singularities and satisfies $(p_d)_* [\cO_{Z_{d-1,1}(X_u,X^v)}] =
[\cO_{\Gamma_{d-1,1}(X_u,X^v)}]$.

\begin{example}\label{example:notprojrich}%
  Let $X = \Gr(3,6)$ be the Grassmannian of 3-planes in $\C^6$ and set $v =
  s_2s_4s_3$ and $u = v^\vee$. Then $v$ corresponds to the partition $I(v) =
  (2,1) = \tableau{8}{{}&{}\\{}}$. A calculation in $\QH(X)$ gives $([X_u]\star
  [X^v])_1 = 1$, so we have $\Gamma_1(X_u,X^v) = X$, and it follows from
  \Proposition{zd1birat} that $\Gamma_{0,1}(X_u,X^v)$ is a divisor in $X$. Let
  $\{e_1,e_2,e_3,e_4,e_5,e_6\}$ be the standard basis of $\C^6$ and set $A_1 =
  \Span\{e_1,e_2\}$, $A_2 = \Span\{e_3,e_4\}$, and $A_3 = \Span\{e_5,e_6\}$. The
  Richardson variety $X_u^v$ is isomorphic to $\P^1 \times \P^1 \times \P^1$ and
  consists of all 3-planes $V = \Span\{a_1,a_2,a_3\}$ for which $a_i \in A_i$.
  The variety $\Gamma_{0,1}(X_u,X^v) = \Gamma_1(X_u^v)$ is the union of all
  lines through $X_u^v$. For any point $V' \in X$ we have $V' \in
  \Gamma_1(X_u^v)$ if and only if there exists a point $V \in X_u^v$ such that
  $\dim(V + V') \leq 4$. Consider the open affine subset $U \subset X$
  corresponding to matrices of the form:
  \begin{equation}\label{eqn:matrix1}
    \begin{bmatrix}
      1 & x_{11} & 0 & x_{12} & 0 & x_{13} \\
      0 & x_{21} & 1 & x_{22} & 0 & x_{23} \\
      0 & x_{31} & 0 & x_{32} & 1 & x_{33}
    \end{bmatrix}
  \end{equation}
  The row space of such a matrix belongs to $X_u^v$ if and only if has the form:
  \begin{equation}\label{eqn:matrix2}
    \begin{bmatrix}
      1 & t_1 & 0 & 0   & 0 & 0 \\
      0 & 0   & 1 & t_2 & 0 & 0 \\
      0 & 0   & 0 & 0   & 1 & t_3
    \end{bmatrix}
  \end{equation}
  The span of the 6 row vectors in \eqn{matrix1} and \eqn{matrix2} has rank 4 or
  less if and only if the matrix
  \[
    \begin{bmatrix}
      x_{11}-t_1 &   x_{12}       &    x_{13}  \\
      x_{21}     &   x_{22}-t_2   &    x_{23}  \\
      x_{31}     &   x_{32}       &    x_{33}-t_3
    \end{bmatrix}
  \]
  has rank at most one, which implies that $x_{12} x_{23} x_{31} = x_{32} x_{21}
  x_{13}$. Since the divisor defined by this equation is irreducible, it
  coincides with $\Gamma_1(X_u^v) \cap U$.

  Let $p_{ijk}$ for $1 \leq i < j < k \leq 6$ denote the Plücker coordinates on
  $X$. Then $\Gamma_1(X_u^v)$ is defined by the equation $p_{123}\, p_{456} =
  p_{124}\, p_{356}$. It follows that $\Gamma_1(X_u^v)$ is a divisor of degree
  2 in $X$, so it is not a projected Richardson variety. In fact, it follows
  from \Theorem{simX} that there are 6 projected Richardson divisors in $X$,
  namely $\Pi_{s_3 w_0^X}^{1}(X)$ and $\Pi_{w_0^X}^{s_k}(X)$ for $1 \leq k \leq
  5$, and since their union is anticanonical by \cite[Lemma
  5.4]{knutson.lam.ea:projections}, each of these divisors has degree 1.
  Moreover, we obtain
  \[
    [\cO_{\Gamma_1(X_u^v)}] = 2 \cO^{(1)} - \cO^{(1)}\cdot \cO^{(1)} = 2
    \cO^{(1)} - \cO^{(2)} - \cO^{(1,1)} + \cO^{(2,1)} \,.
  \]
  Using the Pieri formula \cite[Thm.~5.4]{buch.mihalcea:quantum} we obtain
  \[
    (\cO_u \star \cO^v)_1 =
    1 - 2 \cO^{(1)} + \cO^{(2)} + \cO^{(1,1)} - \cO^{(2,1)}
  \]
  Using \Theorem{qkprod}, we obtain
  \[
    (p_1)_* [\cO_{Z_{0,1}(X_u,X^v)}] = [\cO_{\Gamma_1(X_u,X^v)}] -
    (\cO_u \star \cO^v)_1 = [\cO_{\Gamma_1(X_u^v)}] \,.
  \]
  This identity also follows from \Proposition{zd1birat}, granted that
  $\Gamma_1(X_u^v)$ has rational singularities. In fact, Chenyang~Xu has shown
  us a proof that the local equation $x_{12} x_{23} x_{31} = x_{32} x_{21}
  x_{13}$ is a canonical singularity, which implies that $\Gamma_1(X_u^v) \cap
  U$ has rational singularities. One can check that the local neighborhood of
  $\Gamma_1(X_u^v)$ defined by the non-vanishing of any Plücker coordinate
  $p_{ijk}$ is a deformation of $\Gamma_1(X_u^v) \cap U$. It therefore follows
  from \cite{kawamata:deformations*1} that $\Gamma_1(X_u^v)$ has canonical
  singularities globally, or from \cite{elkik:singularites} that
  $\Gamma_1(X_u^v)$ has rational singularities globally. As mentioned earlier,
  it would be interesting to know if all varieties of the form
  $\Gamma_{d-1,1}(X_u,X^v)$ have rational singularities.
\end{example}

%% file: divisor.tex

\section{Divisors of the quantum-to-classical construction}\label{sec:divisors}

\targetsec{Zd2}{}%
Let $X = G/P_X$ be cominuscule and fix a degree $1 \leq d \leq d_X(2)$. Define
the variety
\[
  \Zd2 \,=\, Z_d \times_{Y_d} Z_d \,=\, \{ (\om,x,y) \in Y_d \times X^2
  \mid x,y \in \Gamma_\om \} \,,
\]
with projections $e_i : \Zd2 \to Z_d$ for $i=1,2$. Recall from
\Notation{epsilon} that we set $\epsilon=2$ if $X$ is a Lagrangian Grassmannian
and $\epsilon=1$ otherwise. This means that the roots of
$\Delta\ssm\Delta_{Y_d}$ are long if $\epsilon=1$ and short if $\epsilon=2$. In
particular, we have $(\al^\vee,\om_\ga)=\epsilon$ for any $\al \in \cP_X$
satisfying $\delta(\al) \in \Delta\ssm\Delta_{Y_d}$.

\begin{prop}\label{prop:D-class}%
  The set $\cD = \{ (\om,x,y) \in \Zd2 \mid \dist(x,y) \leq d-1 \}$ is a divisor
  in $\Zd2$ with rational singularities. The class of $\cD$ in $\Pic \Zd2$ is
  given by
  \[
    [\cD] \,=\, (p_d e_1)^*[X^{s_\ga}] + (p_d e_2)^*[X^{s_\ga}] -
    \epsilon\, (q_d e_1)^* [\partial Y_d] \,.
  \]
\end{prop}
\begin{proof}
  The projection $e_2 : \cD \to Z_d$ is $G$-equivariant, and therefore a locally
  trivial fibration with fibers given by $\cD \cap e_2^{-1}(\om,x) \cong
  \Gamma_{d-1}(x) \cap \Gamma_\om$. \Lemma{prim_dual} and \Lemma{ka_d_inv} imply
  that $\cD$ is a divisor in $\Zd2$ with rational singularities.

  The group $H^2(\Zd2;\Z)$ is a free abelian group generated by the basis
  elements $(p_d e_i)^*[X^{s_\ga}]$ for $i=1,2$ and $(q_d e_1)^*[Y_d^{s_\be}]$
  for $\be \in \Delta\ssm\Delta_{Y_d}$. Set $\om_0 = 1.P_{Y_d} \in Y_d$ and
  $x_0 = 1.P_X \in X$, and define an embedding $\zeta : \Gamma_{\om_0} \to
  \Zd2$ by $\zeta(x) = (\om_0,x,x_0)$. Since $\zeta(\Gamma_{\om_0}) =
  e_2^{-1}(\om_0,x_0)$, it follows from the local triviality of $\cD$ that
  $\zeta^{-1}(\cD)$ is reduced. The identity $\zeta^*[\cD] = [\Gamma_{\om_0}
  \cap \Gamma_{d-1}(x_0)] = \zeta^* (p_d e_1)^* [X^{s_\ga}]$ then implies that
  the coefficient of $(p_d e_1)^*[X^{s_\ga}]$ in $[\cD]$ is one. A symmetric
  argument shows that the coefficient of $(p_d e_2)^*[X^{s_\ga}]$ in $[\cD]$ is
  one.

  Given $\be \in \Delta\ssm\Delta_{Y_d}$, let $\al \in \cP_X$ be the minimal
  root for which $\delta(\al) = \be$. This root $\al$ can be constructed as the
  sum of all simple roots in the interval $[\ga,\be]$ from $\ga$ to $\be$ in the
  Dynkin diagram. Then $I(\kappa_d) \cup \{\al\}$ is a straight shape in
  $\cP_X$, and $\be = \delta(\al) = \kappa_d.\al$. Let $C \subset Z_d$ be the
  $T$-stable curve through the points $\kappa_d.(\om_0,x_0)$ and $\kappa_d
  s_\al.(\om_0,x_0)$. Since $\ka_d^{-1} = \ka_d$, we obtain
  $s_\al\kappa_d.x_0 = \kappa_d s_\be.x_0 = \kappa_d.x_0 \in \Gamma_{\om_0}$.
  It follows that $x_0 \in \Gamma_\om$ for each $\om \in q_d(C)$, so $C' =
  \{(\om,x,x_0) \mid (\om,x) \in C\}$ is a curve in $\Zd2$. Since
  $\kappa_d.x_0$ and $\kappa_d s_\al.x_0$ are points in $\Gamma_d(x_0) \ssm
  \Gamma_{d-1}(x_0)$ by \Theorem{dist}, we obtain $\dist(x,x_0) = d$ for all $x
  \in p_d(C)$, hence $C' \cap \cD = \emptyset$ and $\int_{C'} [\cD] = 0$.
  Finally, since $\int_{C'} (p_d e_1)^*[X^{s_\ga}] = (\al^\vee,\om_\ga) =
  \epsilon$ and $\int_{C'} (q_d e_1)^*[Y_d^{s_\be}] = (\al^\vee,\om_\be) =
  1$, we deduce that the coefficient of $(q_d e_1)^*[Y_d^{s_\be}]$ in $[\cD]$ is
  $-\epsilon$, as required.
\end{proof}

\targetsec{exactnbhd}{}%
Given any closed subset $\Omega \subset X$ we set $\oGamma_d(\Omega) =
\Gamma_d(\Omega) \ssm \Gamma_{d-1}(\Omega)$. We have
\[
  Y_d(\Omega, \oGamma_d(\Omega))
  \,=\, \{ \om \in Y_d(\Omega) \mid
  \Gamma_\om \cap \oGamma_d(\Omega) \neq \emptyset \} \,.
\]
For any point $\om \in Y_d(\Omega)$ we have $\Gamma_\om \cap \Omega \neq
\emptyset$, and since $\Gamma_\om$ has diameter $d$, we obtain $\Gamma_\om
\subset \Gamma_d(\Omega)$. It follows that $\YdoGam{\Omega} = Y_d(\Omega, X \ssm
\Gamma_{d-1}(\Omega))$. Since $q_d$ is an open map, this shows that
$\YdoGam{\Omega}$ is a relatively open subset of $Y_d(\Omega)$. Notice also that
for $v \in W^X$, we have $Y_d(X^v, \oGamma_d(X^v)) \neq \emptyset$ if and only
if $d \leq \dmax(v)$.

\begin{prop}\label{prop:section}%
  Assume that $\Omega \subset X$ is a Schubert variety.
  \begin{abcenum}
  \item For each $\om \in \YdoGam{\Omega}$, $\Gamma_\om \cap \Omega$ is a
    (reduced) single point.\smallskip

  \item The map $\sigma : \YdoGam{\Omega} \to \Omega$ defined by
    $\{\sigma(\om)\} = \Gamma_\om \cap \Omega$ is a morphism of varieties.
  \end{abcenum}
\end{prop}
\begin{proof}
  Given any point $\om \in \YdoGam{\Omega}$, the intersection $\Gamma_\om \cap
  \Omega$ is a Schubert variety in $\Gamma_\om$ by \Theorem{schubfib}. If it has
  positive dimension, then it meets the Schubert divisor $\Gamma_\om \cap
  \Gamma_{d-1}(z)$ for every point $z \in \Gamma_\om$ by \Lemma{prim_dual}. This
  implies that $\Gamma_\om \subset \Gamma_{d-1}(\Omega)$, a contradiction. This
  proves part (a).

  Since $q_d : p_d^{-1}(\Omega) \to Y_d(\Omega)$ is a projective morphism, so is
  the restriction $q_d : p_d^{-1}(\Omega) \cap Z_d(\oGamma_d(\Omega)) \to
  \YdoGam{\Omega}$, and part (a) implies that this restricted map is bijective.
  Since the target is normal, the map is an isomorphism by Zariski's main
  theorem. Part (b) follows from this because $\sigma$ is the composition of the
  inverse map with $p_d$.
\end{proof}

\targetsec{oYd}{}%
Given $u, v \in W^X$ we define the varieties
\[
  \begin{split}
    \oY_d(X_u,X^v) &\,=\, Y_d(X_u,\oGamma_d(X_u)) \cap Y_d(X^v,\oGamma_d(X^v))
      \,\text{, and}\\
    \oY_{d-1,1}(X_u,X^v) &\,=\, \oY_d(X_u,X^v) \cap Y_{d-1,1}(X_u,X^v) \,.
  \end{split}
\]
It follows from Kleiman's transversality theorem \cite{kleiman:transversality}
that $\oY_d(X_u,X^v)$ is a dense open subset of $Y_d(X_u,X^v)$ whenever $d \leq
\min(\dmax(u^\vee),\dmax(v))$. By \Proposition{section} there are morphisms
$\sigma_1 : \oY_d(X_u,X^v) \to X_u$ and $\sigma_2 : \oY_d(X_u,X^v) \to X^v$
defined by $\{\sigma_1(\om)\} = \Gamma_\om \cap X_u$ and $\{\sigma_2(\om)\} =
\Gamma_\om \cap X^v$. By \Corollary{comin_kapd}(b) we have
\begin{equation}\label{eqn:oYd1}
  \oY_{d-1,1}(X_u,X^v) \,=\, \big\{ \om \in \oY_d(X_u,X^v) \mid
  \dist(\sigma_1(\om), \sigma_2(\om)) \leq d-1 \big\} \,.
\end{equation}

\begin{prop}\label{prop:Yd1cartier}%
  Assume $1 \leq d \leq \min(\dmax(u^\vee),\dmax(v))$. Then
  $\oY_{d-1,1}(X_u,X^v)$ is a Cartier divisor in $\oY_d(X_u,X^v)$, with class in
  $\Pic \oY_d(X_u,X^v)$ given by
  \[
    [\oY_{d-1,1}(X_u,X^v)] \,=\, \sigma_1^*[X^{s_\ga}] + \sigma_2^*[X^{s_\ga}]
    - \epsilon\, [\partial Y_d] \,.
  \]
\end{prop}
\begin{proof}
  Define the variety
  \[
    \oZd2(X_u,X^v) = e_1^{-1}(\Omega_1) \cap e_2^{-1}(\Omega_2) \,,
  \]
  where $\Omega_1 = p_d^{-1}(X_u) \cap Z_d(\oGamma_d(X_u))$ and $\Omega_2 =
  p_d^{-1}(X^v) \cap Z_d(\oGamma_d(X^v))$. Since $\Omega_1 \subset
  p_d^{-1}(X_u)$ and $\Omega_2 \subset p_d^{-1}(X^v)$ are open subsets of
  opposite Schubert varieties in $Z_d$, it follows from Kleiman's transversality
  theorem \cite{kleiman:transversality} that $\oZd2(X_u,X^v)$ and $\cD \cap
  \oZd2(X_u,X^v)$ are reduced, where $\cD$ is the divisor of
  \Proposition{D-class}. \Proposition{section} shows that the map $\varphi:
  \oY_d(X_u,X^v) \to \oZd2(X_u,X^v)$ defined by $\varphi(\om) = (\om,
  \sigma_1(\om), \sigma_2(\om))$ is an isomorphism with inverse morphism $q_d
  e_1$, and \eqn{oYd1} shows that $\oY_{d-1,1}(X_u,X^v) = \varphi^{-1}(\cD)$
  holds as (reduced) subschemes of $\oY_d(X_u,X^v)$. The result therefore
  follows from \Proposition{D-class}.
\end{proof}

\begin{prop}\label{prop:div_pull}%
  Let $u \in W^X$, $1 \leq d \leq \dmax(u^\vee)$, and $z \in \oX_{u(d)}$. Then
  the morphism $\sigma : Y_d(X_u,z) \to X_u$ defined by $\{\sigma(\om)\} =
  \Gamma_\om \cap X_u$ is injective, and we have $\sigma^*[X^{s_\ga}] =
  \epsilon [\partial Y_d]$ in $\Pic Y_d(X_u,z)$.
\end{prop}
\begin{proof}
  The assumptions imply that $z \in \oGamma_d(X_u)$, so we must have
  $\dist(\sigma(\om),z) \geq d$ for any point $\om \in Y_d(X_u,z)$. We deduce
  from \Corollary{comin_kapd}(b) that $\Gamma_\om = \Gamma_d(\sigma(\om),z)$.
  This shows that $\sigma$ is injective.

  The projection $q_d : p_d^{-1}(z) \to Y_d(z)$ is an isomorphism, and the
  inverse image of $Y_d(X_u,z)$ is $p_d^{-1}(z) \cap Z_d(X_u)$, which is a
  translate of the Schubert variety $(F_d)_{u_d}$ by \Corollary{fibers_geom}(c).
  This shows that the restriction map $\Pic Y_d(z) \to \Pic Y_d(X_u,z)$ is
  surjective. Since the restriction map $\Pic Y_d \to \Pic Y_d(z)$ is also
  surjective, it follows that $\Pic Y_d(X_u,z)$ is generated by (the
  restrictions of) the divisors $[Y_d^{s_\be}]$ for $\be \in \Delta \ssm
  \Delta_{Y_d}$. The class $[Y_d^{s_\be}]$ is non-zero if and only if $\be \in
  I(u_d)$, which by \Proposition{fibers_combin}(b) is equivalent to $z_d.\be \in
  I(u)$. Notice also that $z_d.\be$ is a minimal box of $I(\ka_d^\vee) \ssm
  I(z_d^\vee)$ by \Proposition{biject}(b).

  To compute $\sigma^*[X^{s_\ga}]$, we may assume that $z = u(d).P_X$, since the
  maps $p_d$ and $q_d$ are equivariant. Let $\be \in \Delta\ssm\Delta_{Y_d}$ and
  assume that $\al = z_d.\be \in I(u)$. Set $\ov{u} = u \cap z_d^\vee$. Then
  $u(d) = \ov{u} z_d$, $\ov{u} s_\al \in W^X$, and $I(\ov{u} s_\al) = I(\ov{u})
  \cup \{\al\}$. We claim that the points $u(d).P_{Y_d}$ and $u(d)s_\be.P_{Y_d}$
  belong to $Y_d(X_u,z)$. Indeed these points are in $Y_d(z)$, since they are
  the images of $u(d).P_{Z_d}$ and $u(d)s_\be.P_{Z_d}$. Since $\kappa_d.P_X \in
  X_{\kappa_d} = p_d q_d^{-1}(1.P_{Y_d})$, we have $1.P_{Y_d} \in
  Y_d(\kappa_d.P_X)$, hence $u(d).P_{Y_d} \in Y_d(\ov{u}z_d \ka_d.P_X) =
  Y_d(\ov{u}.P_X)$ and $u(d)s_\be.P_{Y_d} = \ov{u}s_\al z_d.P_{Y_d} \in
  Y_d(\ov{u}s_\al.P_X)$. This proves the claim, and also shows that
  $\sigma(u(d).P_{Y_d}) = \ov{u}.P_X$ and $\sigma(u(d)s_\be.P_{Y_d}) =
  \ov{u}s_\al.P_X$. We deduce that $Y_d(X_u,z)$ contains the $T$-stable curve $C
  \subset Y_d$ through $u(d).P_{Y_d}$ and $u(d)s_\be.P_{Y_d}$, and that
  $\sigma(C) \subset X_u$ is the $T$-stable curve through $\ov{u}.P_X$ and
  $\ov{u}s_\al.P_X$. This implies that $\sigma_*[(Y_d)_{s_\be}] = \sigma_*[C] =
  [\sigma(C)] = (\al^\vee,\om_\ga)[X_{s_\ga}]$, so it follows from Poincar\'e
  duality that the coefficient of $[Y_d^{s_\be}]$ in $\sigma^*[X^{s_\ga}]$ is
  equal to $\epsilon = (\al^\vee,\om_\ga)$, as required.
\end{proof}

\begin{proof}[Proof of \Proposition{zd1-fibers}]
  Let $z \in \Gamma_d(X_u,X^v)$ be a general point, and set $R = p_d^{-1}(z)
  \cap Z_d(X_u,X^v)$ and $D = p_d^{-1}(z) \cap Z_{d-1,1}(X_u,X^v)$. Then $R$ is
  a Richardson variety by \Theorem{richfib}, and $q_d$ restricts to an
  isomorphism of $R$ onto $R' = q_d p_d^{-1}(z) \cap Y_d(X_u,X^v) = Y_d(X_u,z)
  \cap Y_d(X^v,z)$, under which $D$ is pulled back from $D' = R' \cap
  Y_{d-1,1}(X_u,X^v)$. By the choice of $z$ and the bound $d \leq
  \min(\dmax(u^\vee),\dmax(v))$, we may assume that $z \in \oX_{u(d)} \cap
  \oX^{v(-d)} \subset \oGamma(X_u) \cap \oGamma(X^v)$. This implies that $R'$ is
  contained in $\oY_d(X_u,X^v)$, so it follows from \Proposition{Yd1cartier} and
  \Proposition{div_pull} that $D'$ is a Cartier divisor in $R'$ of class $[D'] =
  \sigma_1^*[X^{s_\ga}] + \sigma_2^*[X^{s_\ga}] - \epsilon [\partial Y_d] =
  \epsilon [\partial Y_d]$. The result follows from this.
\end{proof}

\begin{remark}
  We demonstrate in \Example{div_pull}  that the identity $\sigma^*[X^{s_\ga}] =
  \epsilon [\partial Y_d]$ may fail to hold in $\Pic Y_d(X_u,\oGamma_d(X_u))$,
  with $\sigma$ as in \Proposition{section}. However, the proof of
  \Proposition{div_pull} shows that this identity holds whenever $\Delta \ssm
  \Delta_{Y_d} \subset I(u_d)$, as in this case we have $\Pic Y_d(X_u,z) = \Pic
  Y_d(X_u, \oGamma_d(X_u)) = \Pic Y_d$.
\end{remark}

\begin{example}\label{example:div_pull}%
  Let $X = \Gr(m,n)$ be a Grassmannian of diameter $d_X(2) \geq 3$, and set
  $d=2$ and $u = s_\ga$. Let $E_k = \langle e_1, e_2, \dots, e_k \rangle \subset
  \C^n$ be the subspace spanned by the first $k$ basis vectors, for $0 \leq k
  \leq n$. Then $X_u = \bP(E_{m+1}/E_{m-1}) = \{ V \in X \mid E_{m-1} \subset V
  \subset E_{m+1} \}$. Set $N_0 = \langle e_{m+2}, e_{m+3} \rangle$, $S_0 = E_m
  \oplus N_0$, and let $C \subset Y_d = \Fl(m-2,m+2;n)$ be the curve given by $C
  = \{(K,S_0) \mid K \in \bP(E_{m-1}/E_{m-3})\}$. Since $E_m \in \Gamma_\om \cap
  X_u$ for each $\om \in C$, we have $C \subset Y_d(X_u)$. Define $z : C \to X$
  by $z((K,S_0)) = K \oplus N_0$. For $V \in X_u$ and $(K,S_0) \in C$ we have $V
  \cap (K \oplus N_0) = K$. This implies that $\dist(V,z(\om)) = 2$ for each $V
  \in X_u$ and $\om \in C$, so $z(\om) \in \oGamma_d(X_u) \cap \Gamma_\om$. In
  particular, we have $\om \in Y_d(X_u,z(\om))$ and $C \subset
  Y_d(X_u,\oGamma_d(X_u))$. However, since the restriction $\sigma : C \to X_u$
  of the morphism of \Proposition{section} is the constant function $\sigma(\om)
  = E_m$, we obtain $\int_C \sigma^*[X^{s_\ga}] = 0 \neq 1 = \int_C [\partial
  Y_d]$. More generally, our construction shows that $\sigma^*[X^{s_\ga}] = 0
  \in \Pic Y_d(X_u,\oGamma_d(X_u)) = \Pic Y_d$.
\end{example}

\Proposition{zd1-fibers} shows that the restriction of the divisor
$Z_{d-1,1}(X_u,X^v)$ to $R = Z_d(X_u,X^v) \cap p_d^{-1}(z)$ is a Cartier divisor
that can be pulled back from $Z_d$. The following example shows that
$Z_{d-1,1}(X_u,X^v)$ may not itself be a Cartier divisor pulled back from $Z_d$.

\begin{example}
  Let $X = \LG(3,6)$ and define $u,v \in W^X$ by $I(u)=(3,2)$ and $I(v) =
  (2,1)$. The corresponding products in $\QH(X)$ and $\QK(X)$ are given by
  \[
    \begin{split}
      [X_u] \star [X^v] &=
      2 [X^{(3,1)}] \ \ \ \text{and} \\
      \cO_u \star \cO^v &=
      2 \cO^{(3,1)} - \cO^{(3,2)} + q \cO^{(1)} - q \cO^{(2)} \,.
    \end{split}
  \]
  Let $d=1$. We have $Y_d = \IG(2,6) = C_3/P_2$ and $Y_d(X_u,X^v) = q_d
  p_d^{-1}(X^{s_3 s_2 s_3}) = Y_d^{s_3 s_2}$. The general fibers of $p_d :
  Z_d(X_u,X^v) \to \Gamma_d(X_u,X^v)$ are projective lines, and $p_d :
  Z_{d-1,1}(X_u,X^v) \to \Gamma_d(X_u,X^v)$ is a morphism of degree 2. We also
  have $Y_{d-1} = X$, $Y_{d-1,d} = Z_d$, and $Y_{d-1,1}(X_u,X^v) = q_d
  p_d^{-1}(X_u \cap X^v)$. It follows that
  \[
    [Z_{d-1,1}(X_u,X^v)]
    = (q_d)^*(q_d)_*(p_d)^*([X_u] \cdot [X^v])
    = 2[Z_d^{s_3 s_1 s_2}] \,.
  \]
  Assume that $Z_{d-1,1}(X_u,X^v)$ is the intersection of $Z_d(X_u,X^v)$ with an
  effective Cartier divisor $D \subset Z_d$. Then we must have $[D] \cdot
  [Z_d(X_u,X^v)] = [Z_{d-1,1}(X_u,X^v)]$ in $H^*(Z_d)$. But we have
  $[Z_d(X_u,X^v)] = [Z_d^{s_3 s_2}]$ and $[D] = a [Z_d^{s_2}] + b [Z_d^{s_3}]$
  for some integers $a$ and $b$. Now compute the products
  \[
    [Z_d^{s_2}] \cdot [Z_d^{s_3 s_2}] =
    [Z_d^{s_3 s_1 s_2}] + [Z_d^{s_2 s_3 s_2}]
  \]
  and
  \[
    [Z_d^{s_3}] \cdot [Z_d^{s_3 s_2}] =
    [Z_d^{s_3 s_2 s_3}] + 2 [Z_d^{s_2 s_3 s_2}]
  \]
  It follows that the coefficient of $[Z_d^{s_3 s_2 s_3}]$ in $[D] \cdot
  [Z_d(X_u,X^v)]$ is $b$, and the coefficient of $[Z_d^{s_2 s_3 s_2}]$ is
  $a+2b$. Since these Schubert classes do not appear in $[Z_{d-1,1}(X_u,X^v)]$,
  we obtain $a=b=0$, a contradiction.
\end{example}

%% file: claims.tex

\section{Fibers of Gromov-Witten varieties}\label{sec:claims}%

Let $X$ be a cominuscule flag variety and fix $u, v \in W^X$ and $1 \leq d \leq
d_X(2)$. We finish this paper by proving that completions of the general fibers
of the rational maps $M_d(X_u,X^v) \dashrightarrow Z_d(X_u,X^v)$ and
$M_{d-1,1}(X_u,X^v) \dashrightarrow Z_{d-1,1}(X_u,X^v)$ are cohomologically
trivial. While this assertion from the introduction is not required for the
proofs of our main results, it provides additional details of the relationship
between the geometry of Gromov-Witten varieties and analogous varieties obtained
from the quantum-to-classical construction.

\targetsec{Bld11}{}%
Recall the maps of the diagram \eqn{qcdiagram}, and define the varieties
\[
  \begin{split}
    \Bl_{d-1,1} &\,=\, \pi^{-1}(M_{d-1,1}) \ \subset \Bl_d \,, \\
    \Bl_d(X_u,X^v) &\,=\, \pi^{-1}(M_d(X_u,X^v)) \, \text{, and} \\
    \Bl_{d-1,1}(X_u,X^v) &\,=\, \Bl_d(X_u,X^v) \cap \Bl_{d-1,1} \,.
  \end{split}
\]
Since the birational map $M_d \dashrightarrow Z_d$ is defined as a morphism
exactly on the open subset of $M_d$ over which $\pi : \Bl_d \to M_d$ is an
isomorphism, our assertion is justified by the following result. (We consider a
map between empty varieties to have cohomologically trivial fibers.)

\begin{thm}\label{thm:introclaims}%
  The general fibers of the maps $e_3 \phi : \Bl_d(X_u,X^v) \to Z_d(X_u,X^v)$
  and $e_3 \phi : \Bl_{d-1,1}(X_u,X^v) \to Z_{d-1,1}(X_u,X^v)$ are
  cohomologically trivial.
\end{thm}

The proof requires some additional results, starting with the following
consequence of \Theorem{kollar}.

\begin{cor}\label{cor:cohom_triv_compose}%
  Let $f : M \to N$ and $g : N \to P$ be morphisms of complex projective
  varieties with rational singularities. Assume that the general fibers of $f$
  are cohomologically trivial. Then the general fibers of $g$ are
  cohomologically trivial if and only if the general fibers of $\,gf$ are
  cohomologically trivial.
\end{cor}
\begin{proof}
  This follows from \Theorem{kollar}, as the Grothendieck spectral sequence
  shows that $R^i g_* \cO_N = R^i(gf)_*\cO_M$.
\end{proof}

\begin{lemma}\label{lemma:birat_div}%
  Let $f : M \to N$ be a birational morphism of irreducible varieties, with $N$
  normal. Let $D \subset M$ be an irreducible subvariety of codimension 1, and
  assume that $f(D)$ has codimension 1 in $N$. Then the restricted map $f : D
  \to \ov{f(D)}$ is birational.
\end{lemma}
\begin{proof}
  The assumptions imply that $f(D)$ meets the non-singular locus of $N$, so we
  may assume that $N$ is non-singular. Let $Z \subset N$ be the closed subset of
  points where the rational map $f^{-1}$ is not defined as a morphism. Then
  $f^{-1}(Z)$ is a proper closed subset of $M$, and the fibers of $f^{-1}(Z) \to
  Z$ have positive dimension by \cite[4.4, Thm.~2]{shafarevich:basic*2}. It
  follows that $Z$ has codimension at least 2 in $N$, so $f^{-1}$ is defined on
  a dense open subset of $f(D)$.
\end{proof}

\targetsec{hZd11}{}%
Recall the variety $Y_{d-1,d} = G/(P_{Y_{d-1}} \cap P_{Y_d})$ from
\Section{geomprod} and define
\[
  \begin{split}
    \wh Z_{d-1,1} &\,=\, Z_d \times_{Y_d} Y_{d-1,d} \\
    &\,=\,
    \{ (\eta,\om,z) \in Y_{d-1} \times Y_d \times X \mid
    \Gamma_\eta \subset \Gamma_\om \text{ and } z \in \Gamma_\om \} \,,
    \\
    \hZdd3 &\,=\, Z_{d-1} \times_{Y_{d-1}} Z_{d-1} \times_{Y_{d-1}} \Zhdd \\
    &\,=\,
    \{ (\eta,\om,x,y,z) \in Y_{d-1} \times Y_d \times X^3 \mid
    x,y \in \Gamma_\eta \subset \Gamma_\om \text{ and } z \in \Gamma_\om \}
    \,\text{, and}
    \\
    \Zdd3 &\,=\,
    \{ (\om,x,y,z) \in \Zd3 \mid \dist(x,y) \leq d-1 \} \ \subset \Zd3
    \,.
  \end{split}
\]

\begin{lemma}\label{lemma:Zdd13_birat}%
  The restricted morphism $\phi: \Bl_{d-1,1} \to \Zdd3$ and the projection
  $p' : \hZdd3 \to \Zdd3$ are birational.
\end{lemma}
\begin{proof}
  Let $(\om,x,y,z) \in \Zdd3$. By \Corollary{comin_kapd}(b) there exists $\eta
  \in Y_{d-1}$ such that $x, y \in \Gamma_\eta \subset \Gamma_\om$, and $\eta$
  is unique when $\dist(x,y) = d-1$. This shows that $p'$ is birational. It
  follows from \Lemma{prim_dual} that $\Gamma_\eta \cap \Gamma_1(z)$ contains at
  least one point $t$. There exists a stable curve in $\Gamma_\eta$ of degree
  $d-1$ through $x$, $y$, and $t$ by \Theorem{prim3pt}, and $t$ is connected to
  $z$ by a line. This shows that $(\om,x,y,z)$ belongs to
  $\phi(\Bl_{d-1,1})$. Since $\phi: \Bl_d \to \Zd3$ is birational by
  \Proposition{qc_birat}, it follows from \Proposition{D-class} and
  \Lemma{birat_div} that $\phi: \Bl_{d-1,1} \to \Zdd3$ is birational.
\end{proof}

\targetsec{Zd3uv}{}%
The proof of \Theorem{introclaims} uses the following varieties:
\[
  \begin{split}
    \Zd3(X_u,X^v) &\,=\,
    (p_d e_1)^{-1}(X_u) \cap (p_d e_2)^{-1}(X^v) \ \subset \Zd3 \,, \\
    \Zdd3(X_u,X^v) &\,=\, \Zd3(X_u,X^v) \cap \Zdd3 \,, \\
    \hZdd3(X_u,X^v) &\,=\, p_{d-1}^{-1}(X_u) \times_{Y_{d-1}} p_{d-1}^{-1}(X^v)
    \times_{Y_{d-1}} \Zhdd \,, \\
    Y_{d-1,d}(X_u,X^v) &\,=\,
    \psi_d^{-1}(Y_{d-1}(X_u,X^v)) \,\text{, and} \\
    \wh Z_{d-1,1}(X_u,X^v) &\,=\, Y_{d-1,d}(X_u,X^v) \times_{Y_d} Z_d \,.
  \end{split}
\]
The first three spaces are the subvarieties of $\Zd3$, $\Zdd3$, $\hZdd3$ defined
by $x \in X_u$ and $y \in X^v$. The last variety $\Zhdd(X_u,X^v)$ consists of
all triples $(\eta,\om,z) \in \Zhdd$ for which $X_u \cap \Gamma_\eta \neq
\emptyset$ and $X^v \cap \Gamma_\eta \neq \emptyset$.

\begin{proof}[Proof of \Theorem{introclaims}]
  It follows from \Proposition{qc_birat} and Kleiman's transversality theorem
  \cite{kleiman:transversality} that $\phi : \Bl_d(X_u,X^v) \to \Zd3(X_u,X^v)$
  is birational, and the fiber of $e_3 : \Zd3(X_u,X^v) \to Z_d(X_u,X^v)$ over
  $(\om,z)$ is isomorphic to $(\Gamma_\om \cap X_u) \times (\Gamma_\om \cap
  X^v)$, which is a product of Schubert varieties by \Theorem{schubfib}. Using
  that $\Bl_d(X_u,X^v)$, $\Zd3(X_u,X^v)$, and $Z_d(X_u,X^v)$ have rational
  singularities, it follows from \Corollary{cohom_triv_compose} that the general
  fibers of $e_3\phi : \Bl_d(X_u,X^v) \to Z_d(X_u,X^v)$ are cohomologically
  trivial.

  Consider the commutative diagram:
  \begin{equation}\label{eqn:zdd13-fibers}
    \xymatrix{
      \Zdd3(X_u,X^v) \ar[rrd]_{e_3} &&
      \Bl_{d-1,1}(X_u,X^v) \ar[ll]_{\approx}^\phi \ar[d]^{e_3 \phi} \\
      \hZdd3(X_u,X^v) \ar[u]^{\approx}_{p'} \ar[r]_{\wh p} &
      \Zhdd(X_u,X^v) \ar[d] \ar[r]_{\phi'_d} &
      Z_{d-1,1}(X_u,X^v) \ar[d]^{q_d} \ar[r]_{\hspace{2.3em}\subset} &
      Z_d \ar[d]^{q_d} \\
      & Y_{d-1,d}(X_u,X^v) \ar[r]_{\phi_d} &
      Y_{d-1,1}(X_u,X^v) \ar[r]_{\hspace{2.3em}\subset} &
      Y_d
    }
  \end{equation}
  Here $\wh p$ is the projection that forgets $x$ and $y$, and $\phi'_d$ is the
  base change of $\phi_d$ along $q_d$. It follows from \cite[Thm.~2.5 and
  Prop.~3.7]{buch.chaput.ea:finiteness} together with \Theorem{projrich} and
  \Proposition{D-class} that all varieties in the diagram \eqn{zdd13-fibers}
  have rational singularities. The maps $p'$ and $\phi$ with target
  $\Zdd3(X_u,X^v)$ are birational by \Lemma{Zdd13_birat} and Kleiman's
  transversality theorem \cite{kleiman:transversality}. The fiber of $\wh p$
  over $(\eta,\om,z) \in \Zhdd(X_u,X^v)$ is the product $(\Gamma_\eta \cap X_u)
  \times (\Gamma_\eta \cap X^v)$ of Schubert varieties by \Theorem{schubfib}.
  The fibers of $\phi'_d$ coincide with the fibers of $\phi_d$, and the general
  such fibers are Richardson varieties by \Theorem{richfib}. We deduce from
  \Corollary{cohom_triv_compose} that the general fibers of the maps $e_3$ and
  $e_3 \phi$ with target $Z_{d-1,1}(X_u,X^v)$ are cohomologically trivial. This
  completes the proof.
\end{proof}

\begin{cor}\label{cor:introclaims_birat}%
  The restricted maps $e_3 \phi : \Bl_d(X_u,X^v) \to Z_d(X_u,X^v)$ and $e_3 \phi
  : \Bl_{d-1,1}(X_u,X^v) \to Z_{d-1,1}(X_u,X^v)$ are birational for $d \leq
  \min(\dmax(u^\vee),\dmax(v))$.
\end{cor}
\begin{proof}
  It follows from \Corollary{dimMd}(a) and \Lemma{weylid} that $\dim
  \Bl_d(X_u,X^v) = \dim Z_d(X_u,X^v) = \ell(u) - \ell(v) + \int_d c_1(T_X)$
  (when these varieties are not empty), and from \Proposition{yd1projrich} that
  $\dim \Bl_{d-1,1}(X_u,X^v) = \dim Z_{d-1,1}(X_u,X^v) = \ell(u) - \ell(v) +
  \int_d c_1(T_X) - 1$.
\end{proof}

\begin{remark}
  The proof of \Theorem{introclaims} shows more generally that the general
  fibers of $e_3 \phi : \Bl_d(X_u,X^v) \to Z_d(X_u,X^v)$ are rational, and the
  general fibers of $e_3 \phi : \Bl_{d-1,1}(X_u,X^v) \to Z_{d-1,1}(X_u,X^v)$ are
  rationally connected. The last statement uses
  \cite[Cor.~1.3]{graber.harris.ea:families}.
\end{remark}

%% file: index.tex

\section*{Index of symbols}

We list the most important symbols used in this paper in approximate order of
appearance after the introduction. In electronic versions of this paper, each
symbol is clickable and linked to its definition.\smallskip

\indexsec{flagvar}{%
\slink{group}{$T$}%
\slink{group}{$B$}%
\slink{group}{$G$}%
\slink{group}{$B^-$}%
\slink{group}{$\Phi$}%
\slink{group}{$\Phi^+$}%
\slink{group}{$\Delta$}%
\slink{group}{$W$}%
\slink{flagvar}{$P_X$}%
\slink{flagvar}{$\Phi_X$}%
\slink{flagvar}{$\Delta_X$}%
\slink{flagvar}{$W_X$}%
\slink{flagvar}{$W^X$}%
\slink{flagvar}{$X_u$}%
\slink{flagvar}{$X^v$}%
\slink{flagvar}{$\oX_u$}%
\slink{flagvar}{$\oX^v$}%
\slink{bruhat}{$X_u^v$}%
\slink{bruhat}{$\oX_u^v$}%
\slink{factor}{$u^X$}%
\slink{factor}{$u_X$}%
\slink{factor}{$w_0$}%
\slink{factor}{$w_0^X$}%
\slink{factor}{$w_{0,X}$}%
\slink{factor}{$u^\vee$}%
\slink{factor2}{$u_X^Y$}%
\slink{hecke}{$u \cdot v$}%
\slink{weak}{$v \leq_L u$}%
\slink{weak}{$v \leq_X u$}%
\slink{fundweight}{$\om_\be$}%
}

\indexsec{semitrans}{%
\slink{stab}{$G_\Omega$}%
\slink{stab}{$G(\Omega_1,\Omega_2)$}%
}

\indexsec{projrich}{%
\slink{projrich}{$\Pi_u^v(X)$}%
\slink{projrich}{$\oPi_u^v(X)$}%
\slink{simx}{$\sim_X$}%
}

\indexsec{comin-schubert}{%
\slink{gamma}{$\ga$}%
\slink{poset}{$\cP_X$}%
\slink{poset}{$I(u)$}%
\slink{cominrep}{$w_\la$}%
\slink{cominrep}{$u/v$}%
\slink{labeling}{$\la(\al)$}%
\slink{labeling}{$\delta(\al)$}%
\slink{labeling}{$\delta$}%
\sref{tab:tablez1}{$\Gr(m,n)$}%
\sref{tab:tablez1}{$\OG(n,2n)$}%
\sref{tab:tablez1}{$\LG(n,2n)$}%
\sref{tab:tablez1}{$Q^N$}%
\sref{tab:tablez1}{$E_6/P_6$}%
\sref{tab:tablez1}{$E_7/P_7$}%
\slink{distlattice}{$u \cup v$}%
\slink{distlattice}{$u \cap v$}%
}

\indexsec{cohom-rich}{%
\slink{ktheory}{$K_T(X)$}%
\slink{ktheory}{$\euler{X}$}%
\slink{ktheory}{$\cO^v$}%
\slink{ktheory}{$\cO_u$}%
\slink{ktheory}{$J$}%
\slink{labelweight}{$\la(\cT)$}%
\slink{tabrep}{$\cT[m]$}%
\slink{tabrep}{$C^v_{u,[a,m)}$}%
}

\indexsec{qcintro}{%
\slink{kerspan}{$\Ker(C)$}%
\slink{kerspan}{$\Span(C)$}%
\slink{kerspan}{$\Fl(m-d,m+d;n)$}%
}

\indexsec{nbhd}{%
\slink{gwinv}{$M_d$}%
\slink{gwinv}{$\Mb_{0,3}(X,d)$}%
\slink{gwinv}{$\ev_i$}%
\slink{gwinv}{$\gw{\Omega_1,\Omega_2,\Omega_3}{d}$}%
\slink{gwinv}{$I_d(\cF_1,\cF_2,\cF_3)$}%
\slink{cnbhd}{$M_d(\Omega_1,\Omega_2)$}%
\slink{cnbhd}{$\Gamma_d(\Omega_1,\Omega_2)$}%
\slink{cnbhd}{$M_d(\Omega)$}%
\slink{cnbhd}{$\Gamma_d(\Omega)$}%
\slink{dist}{$\dist(x,y)$}%
\slink{diameter}{$d_X(2)$}%
\slink{diameter}{$\tal_i$}%
\sref{defn:zdkappad}{$\ka_d$}%
\sref{defn:zdkappad}{$z_d$}%
\sref{lemma:orbits}{$\ocZ_{d,2}$}%
\slink{Sd}{$\cS_d$}%
}

\indexsec{incidence}{%
\slink{dynkinint}{$[\ga,\be]$}%
}

\indexsec{prim-nbhd}{%
\slink{qcspaces}{$Y_d$}%
\slink{qcspaces}{$Z_d$}%
\slink{qcspaces}{$p_d$}%
\slink{qcspaces}{$q_d$}%
\slink{qcspaces}{$F_d$}%
\slink{qcspaces}{$\Gamma_d$}%
\slink{qcspaces}{$\Gamma_\om$}%
}

\indexsec{qcblowup}{%
\slink{blowup}{$\Bl_d$}%
\slink{Zd3}{$Z_d^{(3)}$}%
\slink{Zd3}{$e_i$}%
\slink{qcrich}{$Y_d(\Omega_1,\Omega_2)$}%
\slink{qcrich}{$Z_d(\Omega_1,\Omega_2)$}%
\slink{qcrich}{$Y_d(\Omega)$}%
\slink{qcrich}{$Z_d(\Omega)$}%
}

\indexsec{qcbiject}{%
\slink{posetFd}{$\cP_{F_d}$}%
\slink{posetFd}{$\delta'$}%
}

\indexsec{ss:qcfibers}{%
\sref{defn:fiber}{$u(d)$}%
\sref{defn:fiber}{$v(-d)$}%
\sref{defn:fiber}{$\wh u_d$}%
\sref{defn:fiber}{$u_d$}%
\sref{defn:fiber}{$v^d$}%
}

\indexsec{qcohom}{%
\slink{qhprod}{$[X_u]\star[X^v]$}%
\slink{qhprod}{$([X_u]\star [X^v])_d$}%
\slink{dminmax}{$\dmin(u,v)$}%
\slink{dminmax}{$\dmax(u,v)$}%
\slink{dminmax}{$\dmax(v)$}%
}

\indexsec{qh-min-max}{%
\slink{qbruhat}{$\cB$}%
}

\indexsec{dist-lattice}{%
\sref{defn:quantum-poset}{$\wh\cP_X$}%
\sref{defn:quantum-poset}{$I(q^d[X^u])$}%
\sref{defn:quantum-poset}{$\partial(\al)$}%
\sref{defn:quantum-poset}{$\xi(\al)$}%
\sref{defn:quantum-poset}{$\tau(\al)$}%
}

\indexsec{qshapes}{%
\sref{sec:qshapes}{$\nu/\la$}%
}

\indexsec{qcintproof}{%
\slink{Wcomin}{$W^\comin$}%
}

\indexsec{qkring}{%
\slink{qkring}{$\QK(X)$}%
\slink{qkprod}{$\Psi$}%
\slink{qkprod}{$(\cO_u \star \cO^v)_d$}%
\slink{qkprod}{$\cO_u \star \cO^v$}%
\slink{qkconst}{$N^{w,d}_{u,v}$}%
}

\indexsec{qkdual}{%
\slink{qktring}{$\QK_T(M)$}%
\slink{qktring}{$\cK$}%
\slink{qktring}{$\cK\llbracket q\rrbracket$}%
\slink{metric}{$(\!(\cF_1,\cF_2)\!)$}%
\slink{metric}{$\cI_q^v$}%
\slink{dual2schub}{$C_w^{u,d}$}%
\slink{dualconst}{$A_{u,v}^{w,d}$}%
\slink{qposet-basis}{$\cO^\la$}%
\slink{qposet-basis}{$\cI_q^\la$}%
\slink{jxi}{$J$}%
\slink{jxi}{$J_u$}%
\slink{jxi}{$\cJ$}%
\slink{jxi}{$\zeta$}%
\slink{jxi}{$\delta(w/u)$}%
\slink{jxi}{$\theta$}%
\sref{remark:zeta-qshapes}{$\delta(\nu/\la)$}%
\slink{qkdual}{$\ov A_{u,v}^{w,d}$}%
}

\indexsec{geomprod}{%
\slink{Ydd1}{$Y_{d-1,1}$}%
\slink{Ydd1}{$\psi_d$}%
\slink{Ydd1}{$\phi_d$}%
\slink{Ydd1}{$Y_{d-1,1}(X_u,X^v)$}%
\slink{Ydd1}{$Z_{d-1,1}(X_u,X^v)$}%
\slink{Ydd1}{$\Gamma_{d-1,1}(X_u,X^v)$}
}

\indexsec{strategy}{%
\slink{resolclass}{$[\cO_{\wt\Omega}]$}%
}

\indexsec{pf-cex}{%
\sref{notation:epsilon}{$\epsilon$}%
\slink{partialYd}{$[\partial Y_d]$}%
\slink{altsign}{$\lead(\cF)$}%
\slink{altsign}{$c_w(\cF)$}%
}

\indexsec{divisors}{%
\slink{Zd2}{$Z_d^{(2)}$}%
\sref{prop:D-class}{$\cD$}%
\slink{exactnbhd}{$\oGamma_d(\Omega)$}%
\slink{oYd}{$\oY_d(X_u,X^v)$}%
\slink{oYd}{$\oY_{d-1,1}(X_u,X^v)$}%
\slink{oYd}{$\sigma_i$}%
}

\indexsec{claims}{%
\slink{Md11}{$M_{d-1,1}$}%
\slink{Md11}{$M_{d-1,1}(X_u,X^v)$}%
\slink{Bld11}{$\Bl_{d-1,1}$}%
\slink{Bld11}{$\Bl_d(X_u,X^v)$}\sbreak%
\slink{Bld11}{$\Bl_{d-1,1}(X_u,X^v)$}%
\slink{hZd11}{$\wh Z_{d-1,1}$}%
\slink{hZd11}{$\wh Z^{(3)}_{d-1,1}$}%
\slink{hZd11}{$Z^{(3)}_{d-1,1}$}%
\slink{Zd3uv}{$Z^{(3)}_d(X_u,X^v)$}%
\slink{Zd3uv}{$Z^{(3)}_{d-1,1}(X_u,X^v)$}\sbreak%
\slink{Zd3uv}{$\wh Z^{(3)}_{d-1,1}(X_u,X^v)$}%
\slink{Zd3uv}{$Y_{d-1,d}(X_u,X^v)$}%
\slink{Zd3uv}{$\wh Z_{d-1,1}(X_u,X^v)$}%
}